\documentclass[9pt]{amsart}
\usepackage[margin=2.5cm, marginpar=2.5cm]{geometry}

\usepackage{booktabs}
\usepackage{amssymb}
\usepackage{braket}
\usepackage{mathrsfs}
\usepackage{ifthen}
\usepackage{graphicx}

\usepackage{comment} 
\usepackage{mleftright}
\usepackage[pagebackref,hypertexnames=false]{hyperref} 
\usepackage{cleveref}
 \usepackage[all]{xy}
 \usepackage{amscd}
 \usepackage{color}
 \usepackage{enumitem}
\newlist{steps}{enumerate}{1}
\setlist[steps, 1]{label = Step \arabic*:}
\usepackage{scrextend}
\usepackage[alphabetic,backrefs,msc-links]{amsrefs}

\makeatletter
\DeclareRobustCommand\widecheck[1]{{\mathpalette\@widecheck{#1}}}
\def\@widecheck#1#2{%
   \setbox\z@\hbox{\m@th$#1#2$}%
   \setbox\tw@\hbox{\m@th$#1%
      \widehat{%
         \vrule\@width\z@\@height\ht\z@
         \vrule\@height\z@\@width\wd\z@}$}%
   \dp\tw@-\ht\z@
   \@tempdima\ht\z@ \advance\@tempdima2\ht\tw@ \divide\@tempdima\thr@@
   \setbox\tw@\hbox{%
      \raise\@tempdima\hbox{\scalebox{1}[-1]{\lower\@tempdima\box\tw@}}}%
   {\ooalign{\box\tw@ \cr \box\z@}}}
\makeatother

\theoremstyle{plain}
\newtheorem{thm}{Theorem}[section]
\crefname{thm}{Theorem}{Theorems}
\Crefname{thm}{Theorem}{Theorems}
\newtheorem{prop}[thm]{Proposition}
\crefname{prop}{Proposition}{Propositions}
\Crefname{prop}{Proposition}{Propositions}
\newtheorem{lem}[thm]{Lemma}
\crefname{lem}{Lemma}{Lemmas}
\Crefname{lem}{Lemma}{Lemmas}
\newtheorem{cor}[thm]{Corollary}
\crefname{cor}{Corollary}{Corollaries}
\Crefname{cor}{Corollary}{Corollaries}

\crefname{claim}{Claim}{Claims}
\Crefname{claim}{Claim}{Claims}

\crefname{property}{Property}{Properties}
\Crefname{property}{Property}{Properties}

\crefname{problem}{Problem}{Problems}
\Crefname{problem}{Problem}{Problems}

\crefname{ques}{Question}{Questions}
\Crefname{ques}{Question}{Questions}

\theoremstyle{definition}
\newtheorem{defn}[thm]{Definition}
\crefname{defn}{Definition}{Definitions}
\Crefname{defn}{Definition}{Definitions}

\crefname{notation}{Notation}{Notations}
\Crefname{notation}{Notation}{Notations}

\crefname{convention}{Convention}{Conventions}
\Crefname{convention}{Convention}{Conventions}

\crefname{cond}{Condition}{Conditions}
\Crefname{cond}{Condition}{Conditions}

\crefname{assum}{Assumption}{Assumptions}
\Crefname{assum}{Assumption}{Assumptions}

\crefname{conj}{Conjecture}{Conjectures}
\Crefname{conj}{Conjecture}{Conjectures}

\theoremstyle{remark}
\newtheorem{rem}[thm]{Remark}
\crefname{rem}{Remark}{Remarks}
\Crefname{rem}{Remark}{Remarks}

\crefname{ex}{Example}{Examples}
\Crefname{ex}{Example}{Examples}

\crefname{section}{Section}{Sections}
\Crefname{section}{Section}{Sections}
\crefname{subsection}{Subsection}{Subsections}
\Crefname{subsection}{Subsection}{Subsections}
\crefname{figure}{Figure}{Figures}
\Crefname{figure}{Figure}{Figures}

\newcommand{\Z}{\mathbb{Z}}

\newcommand{\Q}{\mathbb{Q}}

\newcommand{\sign}{\mathrm{sign}}

\newcommand{\id}{\mathrm{id}}

\newcommand{\Spin}{\mathrm{Spin}}
\newcommand{\Pin}{\mathrm{Pin}}

\newcommand{\C}{\mathbb{C}}
\newcommand{\s}{\mathfrak{s}}

\newcommand{\pr}{\text{pr}}

\newcommand{\R}{\mathbb R}

\newcommand{\ctext}[1]{\raise0.2ex\hbox{\textcircled{\scriptsize{#1}}}}
\newcommand{\Sp}{\mathrm{Sp}}

\def\ker{\operatorname{Ker}}

\def\det{\operatorname{det}}

\def\dim{\operatorname{dim}}

\def\id{\operatorname{Id}}

\newcommand{\mbar}[1]{{\ooalign{\hfil#1\hfil\crcr\raise.167ex\hbox{--}}}}

\def\wt{\widetilde}

\title{Real Floer Homotopy Types and $w$-Invariants}
\author{Yoshihiro Fukumoto and Masaki Taniguchi}

\begin{document}

\maketitle

\begin{abstract}
We express the local-equivalence classes of real Seiberg--Witten Floer
homotopy types for spin Seifert $3$--manifolds with certain odd involutions in terms of the Fukumoto--Furuta $w$--invariant.  The proof uses real orbifold Bauer--Furuta invariant
for a class of spin $4$--orbifolds constructed by
Fukumoto--Furuta--Ue.  For torus knots, we prove that the local-equivalence
classes of real Floer homotopy types associated with their $2^k$--fold cyclic branched covers are
eventually periodic in $k$.  Finally, we prove that the integer-valued
concordance homomorphisms
$\{8\delta_R^{(k)}\}_{k\geq1}$ are linearly independent.
\end{abstract}

\section{Introduction}

Yang--Mills gauge theory on $4$--orbifolds with isolated singularities was
initiated by Fintushel and Stern \cite{Fintushel-Stern:1985-1} and was applied
to the study of smooth $4$--manifolds bounded by Seifert homology 3-spheres. See \cite{FL86, Lawson88, Lawson87, FS87,Fu90, FurutaOrbifold92,FF00,fukumoto2001w,   Ue01, UeSeifertSurgery, UeRochlin, Fuk11 } for further studies of gauge theory for 4-orbifolds with isolated singularities. 
For a Seifert fibered homology $3$--sphere
$Y=\Sigma(a_1,\ldots,a_n)$, Fintushel and Stern introduced a combinatorial
invariant $R(a_1,\ldots,a_n)$,
now called the {\it Fintushel--Stern invariant}, and used it to obstruct the
existence of compact positive-definite smooth $4$--manifolds bounded by
Seifert homology spheres.  From the gauge-theoretic point of view,
$R(a_1,\ldots,a_n)$ is the formal dimension of an instanton moduli space on a
certain $4$--orbifold bounded by $Y$. Furuta \cite{Furuta:1990-1} and
Fintushel--Stern \cite{FS90} extended this approach to linear
combinations of Seifert homology spheres, thereby proving that the
$3$--dimensional homology cobordism group is infinitely generated.

The Seiberg--Witten counterpart of this orbifold approach was developed
by Fukumoto and Furuta \cite{FF00}.  Let $Y$ be an integral homology
$3$--sphere and let $(X,\widetilde{\mathfrak{s}})$ be a compact spin
$4$--orbifold with boundary $(Y,\mathfrak{s})$.  Fukumoto and Furuta
defined an index-theoretic invariant $ 
w(X,\widetilde{\mathfrak{s}})$
from the orbifold spin Dirac operator.  In general, this invariant depends
on the choice of the bounding spin $4$--orbifold.  For the distinguished
orbifolds associated with Seifert and plumbed homology spheres,
Fukumoto--Furuta--Ue \cite{fukumoto2001w} and Saveliev \cite{Sa02}
identified it with the negative of the {\it Neumann--Siebenmann invariant} \cite{Neumann80}:
\[
w(X,\widetilde{\mathfrak{s}})=-\bar{\mu}(Y).
\]

The invariant $\bar{\mu}(Y)$ has several other descriptions using Floer theory.   For an
almost-rational plumbed homology sphere $Y$, oriented as the boundary of
a negative-definite plumbing, one has
\[
\beta(Y)=-\bar{\mu}(Y),
\qquad
\frac{\underline{d}(Y)}{2}=-\bar{\mu}(Y),
\]
by \cite{dai2018pin,stoffregen2020pin,DM19}.  Moreover, Dai--Sasahira--Stoffregen 
\cite{DSS2023} proved 
\[
\kappa(Y)=
\begin{cases}
-\bar{\mu}(Y)
& \text{if } \delta(Y)=-\bar{\mu}(Y),\\
-\bar{\mu}(Y)+2
& \text{if } \delta(Y)>-\bar{\mu}(Y).
\end{cases}
\]
Here $\beta$ is Manolescu's Pin$(2)$-equivariant Fr\o yshov-type invariant,
$\underline{d}$ is the involutive Heegaard Floer correction term,
$\delta$ is the monopole Fr{\o}yshov invariant, and $\kappa$ is
Manolescu's Pin$(2)$--equivariant Floer $K$--theoretic invariant.

Recently, Konno--Miyazawa and the second author \cite{KMT21, KMT:2023}\footnote{The authors of  \cite{KMT21, KMT:2023} only treat non-free involutions on 3- and 4-manifolds.  See also \cite{MPT25} for free involutions. } introduced real versions of these homology cobordism invariants for a spin 3-manifold $Y$ equipped with an odd involution 
$\tau$:
\begin{align}\label{main_inv}
\delta_R(Y,\tau), \qquad \underline{\delta}_R(Y,\tau), \qquad \overline{\delta}_R(Y,\tau), \qquad \kappa_R(Y,\tau)  \qquad \in \frac{1}{16} \Z.
\end{align}
Here, {\it odd involution}  $\tau$ means a smooth involution on $Y$ so that a lift of it to the spin structure is order $4$. Also, compared to the classical case, the homological conditions to define these invariants are milder, expressed as
\begin{align}\label{b_1_cond}
  \mathrm{id}  =  ( \tau^* \colon H^1(Y; \R) \to H^1(Y; \R))  . 
\end{align}
Using these invariants, we can obtain a family of concordance invariants: 
For a given knot $K$ in $S^3$, by considering its $2^k$-covering space $\Sigma_{2^k }(K)$ equipped with the unique spin structure, one can associate a pair $(\Sigma_{2^k }(K), \tau)$, where $\tau$ is an involution obtained as the $2^{k-1}$-th power of the covering action. Then we shall obtain 
\begin{align}\label{list_conc_inv}
\delta_R^{(k)} (K), \qquad \underline{\delta}_R^{(k)} (K), \qquad  \overline{\delta}_R^{(k)} (K), \qquad    \kappa_R^{(k)} (K)  \qquad \in \frac{1}{8}\Z
\end{align}
for $k \geq 1$.\footnote{Originally, these invariants take values in $\frac{1}{16}\Z $ but we can check the values are actually in $\frac{1}{8}\Z$ for knots in $S^3$. See \cref{prop:eighth-integrality-knot-invariants}. } In particular, the invariants $\{\kappa_R^{(k)} (K) \}_{k\in \Z_{\geq 1}}$  are enough to show that any non-trivial cable of the figure eight knot is not slice, see \cite{KPT24, KPT25}.

In this paper, we develop real Seiberg--Witten theory for odd spin 4--orbifolds in order to establish a general relation between Fukumoto--Furuta invariants and \eqref{main_inv}. 
Note that the numerical invariants listed in \eqref{main_inv} are all derived from the most refined invariant 
\[
[SWF_R(Y,\tau)]_{\mathrm{loc}} \in \mathcal{LE}_{\Z_4},
\]
called the {\it local equivalence class}, introduced in \cite{KMT21, KMT:2023, MPT25}, where \(\mathcal{LE}_{\Z_4}\) denotes the group of local equivalence classes of
stable homotopy types of the relevant \(\Z_4\)-equivariant finite and pointed CW complexes.
\footnote{See ~\cite{TW09, Na13, Nak15, KMT21, Ka22, Ji22, KMT:2023, Mi23, Li23, BH24} for the further background of real Seiberg--Witten theory. } 
We determine this local equivalence class for certain spin-odd Seifert 3-manifolds in terms of the Fukumoto--Furuta invariants as follows: 
\begin{thm}\label{thm:main}
Let $(Y, \s, \tau)$ be a closed spin Seifert 3-manifold with an odd involution and with genus-$0$ oriented base.   
If the involution $\tau$ commutes with the Seifert $S^1$-action and satisfies \eqref{b_1_cond}, we have 
    \[
    [SWF_R( Y,\mathfrak{s}, \tau)]_{\mathrm{loc}} =\left[\left(\C_+^{-\frac{1}{2} w(X, \wt{\mathfrak{s}})} \right)^+\right] _{\mathrm{loc}},
    \]
    where $w(X, \wt{\mathfrak{s}})$ is the Fukumoto--Furuta invariant for a spin 4-orbifold constructed in \cite{fukumoto2001w} and $\mathbb C_+$ denote the one-dimensional complex
$\Z_4 = \langle j \rangle (\subset \mathrm{Pin}(2))$--representation on which $j$ acts by multiplication by $i$. 
\end{thm}
Since the quantity $\frac{1}{2}w(X, \wt{\mathfrak{s}}) $ only depends on $( Y,\mathfrak{s}, \tau)$, we shall denote it by $w_R( Y,\mathfrak{s}, \tau)$ and call it {\it real $w$-invariant}.

Using the orbifold index theorem, we give a
concrete combinatorial formula for $w_R(Y,\mathfrak{s},\tau)$ in
\Cref{thm:explicit-w-formula}.
From this combinatorial formula, we can see $w_R( Y,\mathfrak{s}, \tau)$ coincides with $-\frac{1}{2}\bar{\mu}(Y)$ in many cases. The authors believe such equalities hold in general. 
 
\begin{rem}
    The same type of results hold for the case of general oriented base spaces with certain assumptions of spin structures on $Y$. See \cref{rem:positive-genus-case} for the details. The combinatorial formula corresponding to \cref{thm:explicit-w-formula} still works in this case. 
\end{rem}

Since the numerical cobordism invariants of $( Y,\mathfrak{s}, \tau)$ are determined by $   [SWF_R( Y,\mathfrak{s}, \tau)]_{\mathrm{loc}} $, we see 
\begin{cor}\label{thm:Quantities}
Under the hypotheses of \Cref{thm:main}, one has
\[
\kappa_R(Y,\mathfrak{s},\tau)
=
\overline{\delta}_R(Y,\mathfrak{s},\tau)
=
\underline{\delta}_R(Y,\mathfrak{s},\tau)
=
\delta_R(Y,\mathfrak{s},\tau)
=
w_R(Y,\mathfrak{s},\tau).
\]
\end{cor}

By applying \cref{thm:Quantities} to branched covers of the torus knot $T(p,q)$ of type $(p,q)$, we shall obtain: 

\begin{thm}\label{torus_computation}
Let $T(p,q)$ be a torus knot and let $k\geq1$. Then
\[
\begin{aligned}
\delta_R^{(k)}(T(p,q))
&=
\underline{\delta}_R^{(k)}(T(p,q))
=
\overline{\delta}_R^{(k)}(T(p,q)) = 
\kappa_R^{(k)}(T(p,q))
=
w_R\bigl(
\Sigma(2^k,p,q),\mathfrak{s}_k,\tau_k
\bigr),
\end{aligned}
\]
where $\Sigma(2^k,p,q)$ is the Brieskorn $3$--manifold which is the
$2^k$--fold cyclic branched cover of $T(p,q)$,
$\mathfrak{s}_k$ is its distinguished spin structure, and
$\tau_k$ is the order-two deck transformation.
\end{thm}

\begin{rem}
For torus knots $T(p,q)$ with $p$ and $q$ odd, the case $k=1$ was proved
in \cite{KMT:2023}.  Related results were also obtained in \cite{KPT24, KPT25}
under certain arithmetic restrictions on $k,p$, and $q$.  The methods in
these works are based on Mrowka--Ozsv\'ath--Yu's description of the
Seiberg--Witten moduli spaces of Seifert rational homology $3$--spheres
\cite{MOY96}, together with equivariant refinements \cite{KPT24, KPT25} of lattice homotopy type
introduced by Dai--Sasahira--Stoffregen \cite{DSS2023}.  In contrast, our
approach uses the real orbifold Bauer--Furuta invariant and does not rely on
these descriptions of the moduli spaces or on lattice homology.
\end{rem}

The explicit computations in \Cref{torus_computation} also reveal the
following periodicity phenomenon.

\begin{thm}\label{periodic}
Write $
        pq=2^{a_{p,q}}m_{p,q},
        \ 
        m_{p,q}\ \text{odd}.$
Let $d_{p,q}$ be the minimal positive integer such that \[
2^{d_{p,q}}\equiv 1 \pmod{m_{p,q}}.
\]
Then, for every $k\geq a_{p,q}+1$, we have
\[
    [     SWF_R^{(k)}(T(p,q))]_{\mathrm{loc}}
        =
        [SWF_R^{(k+d_{p,q})}(T(p,q))]_{\mathrm{loc}}.
\]
In particular, all numerical invariants in \eqref{list_conc_inv} are eventually periodic with the same period. 
\end{thm}
The proof of \cref{periodic} relies on 0-surgery invariance of real Floer homotopy type proven in \cite{MPT25}. Floer-theoretic concordance invariants obtained from branched covers
have previously been studied in
\cite{MO07,JN07, Jab12,HLS16, ACS21,AKS20,BH21,Baraglia24, IT24}.  To the best of our knowledge, the
eventual periodicity of the entire local equivalence class as the
covering exponent varies is new. 
Finally, using explicit computations obtained via \cref{torus_computation}, we shall see: 

\begin{thm}\label{thm:integral-linear-independence-delta}
The integer-valued homomorphisms
\[
8\delta_R^{(k)}
\colon
\mathcal C\longrightarrow\Z
\]
are linearly independent over $\Z$.
\end{thm}

\subsection*{Acknowledgements}

We dedicate this paper to the memory of Mikio Furuta.
The authors would like to thank Kimihiko Motegi and Tetsuya Ito for very helpful
discussions on Seifert invariants of 0-surgeries of torus knots. The second author was partially supported by JSPS KAKENHI Grant Number 22K13921. The authors used generative AI tools for language editing and proofreading.

\section{Preliminaries}

\subsection{Notations of Seifert fibrations}
For fundamental notions of orbifolds used in this paper, see \Cref{Apx:orbifold}. 
Let $\Sigma_g$ be a closed oriented surface of genus $g$. Fix pairwise disjoint disks $U_1,\ldots,U_n \subset \Sigma_g$ and put
$$S_0:=\Sigma_g\setminus \coprod_{i=1}^n \mathring{U}_i.$$
Fix a tuple of non-zero integers $(a_1,\ldots,a_n)$. For each $i$, let $\mathbb{Z}_{a_i}$ denote the cyclic group of order $|a_i|$, acting on $D^2\subset \mathbb{C}$ by $\zeta_i\cdot z=\zeta_i z,$
where $\zeta_i$ is a primitive $|a_i|$-th root of unity. We define the closed oriented orbifold surface
$$\check{\Sigma}=\check{\Sigma}(g;a_1,\ldots,a_n)$$
by
$$\check{\Sigma}=S_0\cup_{\{\bar{\varphi}_i\}}\coprod_{i=1}^n D^2/\mathbb{Z}_{a_i},$$
where the gluing maps are given by
$$\bar{\varphi}_i:\partial D^2/\mathbb{Z}_{a_i}\to \partial U_i,\qquad \bar{\varphi}_i([z])=z^{a_i}.$$

Fix integers $\{b_i\}$ such that $a_i$ and $b_i$ are relatively prime.
Let 
\[
\pi : L = L((a_1,b_1),\ldots,(a_n,b_n)) \to \check{\Sigma}
\]
be an orbifold complex line bundle over $\check{\Sigma} = \check{\Sigma}(g; a_1,\ldots,a_n)$ defined to be
\begin{eqnarray*}
    && L = S_0 \times \mathbb{C} \cup_{\{\varphi_i\}} \coprod_{i=1}^n (D_i \times \mathbb{C})/\mathbb{Z}_{a_i},
\end{eqnarray*}
where $\mathbb{Z}_{a_i}$ acts on $D_i\times D^2 \subset \mathbb{C} \times \mathbb{C}= \mathbb{C}^2$ by 
$\zeta_i \cdot (z,w) = (\zeta_i z ,\zeta_i^{b_i} w)$
and the gluing maps are given by 
$$\varphi_i : (\partial D_i \times \mathbb{C})/\mathbb{Z}_{a_i} \to \partial U_i \times \mathbb{C}, ~
\varphi_i([(z,w)]) = (z^{a_i},z^{-b_i}w).$$
Then $s = \{s_0, s_i\}$, $s_0(z) = 1$ for $z \in S_0$ and $s_i([z]) = [z,v_i(z)]$ for $[z] \in D_i/\mathbb{Z}_{a_i}$, where $v_i(z) = z^{b_i}$ for $b_i \geq 0$ and $v_i(z) = (\rho(|z|) \bar{z})^{-b_i}$ for $b_i < 0$ with a strictly positive function  $\rho \in C^{\infty}(\mathbb{R}_{\geq 0})$ satisfying $\rho(r) = 1$ for $0 \leq r \leq 1/3$ and $\rho(r) = 1/r^2$ for $r \geq 2/3$, gives a section $s : S \to L$ of $L$ and the rational Euler number of $\pi : L \to S$ is calculated to be 
$e(L) = \sum_{i=1}^n \frac{b_i}{a_i}.$

The following proposition describes the classification of orbifold line bundles over closed oriented orbifold surfaces. Let $\check{\Sigma}=\check{\Sigma}(g;a_1,\ldots,a_n)$ be a closed oriented orbifold surface. We denote by
$$
\mathrm{Pic}^t(\check{\Sigma})
=
\{L\mid L\to \check{\Sigma}\text{ is a smooth orbifold complex line bundle}\}/\cong
$$
the set of isomorphism classes of smooth orbifold complex line bundles over $\check{\Sigma}$.
We have the following classification proven by Furuta--Steer \cite{FurutaSteer}: 
\begin{prop}[\cite{FurutaSteer}]\label{prop:Pic-orbifold-surface}
Let $\check{\Sigma}=\check{\Sigma}(g;a_1,\ldots,a_n)$ be a closed oriented orbifold surface with $n\geq 1$. Then any element of $\mathrm{Pic}^t(\check{\Sigma})$ is represented by an orbifold line bundle of the form
$$
L((a_1,b_1),\ldots,(a_n,b_n))
$$
for some integers $b_i$. Moreover, the map
$$
\Theta_{\check{S}}:
\mathrm{Pic}^t(\check{\Sigma})
\to
\Z_{a_1}\oplus\cdots\oplus \Z_{a_n}\oplus \Q
$$
given by
$$
[L((a_1,b_1),\ldots,(a_n,b_n))]
\mapsto
(b_1\operatorname{mod} a_1,\ldots,b_n\operatorname{mod} a_n,e(L))
$$
is injective. Its image consists of the tuples
$
(r_1,\ldots,r_n,e)
\in
\Z_{a_1}\oplus\cdots\oplus\Z_{a_n}\oplus\Q
$
satisfying
$$
e-\sum_{i=1}^n\frac{\widetilde r_i}{a_i}\in\Z,
$$
where $\widetilde r_i\in\Z$ is any lift of $r_i\in\Z_{a_i}$.
\end{prop}

We put $Y := S(L)$, the sphere bundle of $L$. The following is the homological computation of $Y$ described by its Seifert invariant and genus.  
The following is proven by Mayer--Vietoris argument: 
\begin{prop}\label{prop:H1-Seifert}
Let
$$
Y=M(g;(a_1,b_1),\ldots,(a_n,b_n))
$$
be the circle bundle associated with the orbifold line bundle
$$
L((a_1,b_1),\ldots,(a_n,b_n))\to \check{\Sigma}(g;a_1,\ldots,a_n).
$$
For each singular fiber, let $\gamma_i$ denote the class of a small meridian of the $i$-th singular fiber, and let $\delta$ denote the class of a regular fiber. We also choose classes
$
\beta_1,\ldots,\beta_{2g}
$
whose projections to the underlying surface $\Sigma_g$ form a symplectic basis of $H_1(\Sigma_g;\mathbb{Z})$. Then
$$
H_1(Y;\mathbb{Z})
\cong
\left\langle
\gamma_1,\ldots,\gamma_n,\delta,\beta_1,\ldots,\beta_{2g}
\left|
\sum_{i=1}^n\gamma_i=0,\quad
a_i\gamma_i-b_i\delta=0\;\;(i=1,\ldots,n)
\right.
\right\rangle .
$$
In particular, the generators $\beta_1,\ldots,\beta_{2g}$ give a free direct summand.

Let $R$ be the relation matrix for the part generated by
$
\gamma_1,\ldots,\gamma_n,\delta .
$
Then
$$
\det R
=
\pm \alpha e
=
\pm \sum_{i=1}^n b_i a_1\cdots \widehat{a_i}\cdots a_n,
\qquad
\alpha=\prod_{i=1}^n a_i,\qquad
e=\sum_{i=1}^n\frac{b_i}{a_i}.
$$
Consequently, $M(g;(a_1,b_1),\ldots,(a_n,b_n))$ is a $\mathbb{Z}_m$-homology $3$-sphere if and only if
$
g=0
$
and $\det R$ is invertible in $\mathbb{Z}_m$. 
\end{prop}

\subsection{Classifications of involutions}
Let $Y$ be an oriented Seifert $3$--manifold with Seifert fibration
$$
S^1 \to Y \xrightarrow{\pi} \check{\Sigma}(g;a_1,\ldots,a_n).
$$
We first record the following elementary consequence.

\begin{lem}\label{classify_under}
Let $\tau:Y\to Y$ be an orientation-preserving involution which commutes with the Seifert $S^1$--action. Assume that
$$
\tau^*=\id:H^1(Y;\mathbb{R})\to H^1(Y;\mathbb{R}).
$$
Let $\overline{\tau}$ be the induced involution on $\check{\Sigma}$, and let $\underline{\tau}$ be the induced involution on the underlying smooth surface $\Sigma_g$. Then $\underline{\tau}:\Sigma_g\to \Sigma_g$ is conjugate to one of the following maps:
\begin{itemize}
    \item if $g=0$, either $\id:S^2\to S^2$ or the $180$-degree rotation about the $z$-axis;
    \item if $g=1$, either $\id:T^2\to T^2$ or a non-trivial translation of order two, for example
    $$
    (x,y)\mapsto \left(x+\frac{1}{2},y\right):S^1\times S^1\to S^1\times S^1;
    $$
    \item if $g\geq 2$, the identity map $\id:\Sigma_g\to\Sigma_g$.
\end{itemize}
\end{lem}

\begin{proof}
For the Seifert $U(1)$-fibration
$
U(1)\to Y\xrightarrow{\pi}\check{\Sigma},
$
we have the Gysin exact sequence over $\mathbb{R}$:
$$
0
\longrightarrow
H^1_{\mathrm{orb}}(\check{\Sigma};\mathbb{R})
\xrightarrow{\pi^*}
H^1(Y;\mathbb{R})
\longrightarrow
H^0_{\mathrm{orb}}(\check{\Sigma};\mathbb{R})
\xrightarrow{\cdot e(\pi)}
H^2_{\mathrm{orb}}(\check{\Sigma};\mathbb{R})
\longrightarrow
H^2(Y;\mathbb{R})
\longrightarrow \cdots .
$$
In particular, $\pi^*$ is injective. Since $\tau$ commutes with the Seifert $S^1$--action, the induced involution $\overline{\tau}$ on $\check{\Sigma}$ satisfies
$$
\pi\circ \tau=\overline{\tau}\circ \pi.
$$
Thus the assumption $\tau^*=\id$ on $H^1(Y;\mathbb{R})$ implies
$$
\overline{\tau}^*=\id:H^1_{\mathrm{orb}}(\check{\Sigma};\mathbb{R})\to H^1_{\mathrm{orb}}(\check{\Sigma};\mathbb{R}).
$$
Since the real orbifold cohomology of $\check{\Sigma}$ agrees with the real cohomology of its underlying surface $\Sigma_g$, we obtain
$$
\underline{\tau}^*=\id:H^1(\Sigma_g;\mathbb{R})\to H^1(\Sigma_g;\mathbb{R}).
$$

Suppose that $\underline{\tau}$ is not the identity. Since $\underline{\tau}$ is orientation-preserving, the quotient
$
S=\Sigma_g/\langle\underline{\tau}\rangle
$
is an oriented surface, and the natural projection
$
p:\Sigma_g\to S
$
is a two-fold branched covering. Let $h$ be the genus of $S$, and let $r$ be the number of branch points. The Riemann--Hurwitz formula gives
$$
2-2g=2(2-2h)-r=4-4h-r.
$$
On the other hand, by the transfer argument,
$
H^1(S;\mathbb{R})\cong H^1(\Sigma_g;\mathbb{R})^{\underline{\tau}^*}.
$
Since $\underline{\tau}^*=\id$, we have $2h=2g$, and hence $h=g$. Substituting this into the Riemann--Hurwitz formula gives
$
r=2-2g.
$
Therefore $g$ is either $0$ or $1$.

If $g=0$, the standard classification of periodic maps of $S^2$ implies that every non-trivial orientation-preserving involution is conjugate to the $180$-degree rotation about an axis. If $g=1$, then $r=0$, so $\underline{\tau}$ is fixed-point-free. By the standard classification of periodic maps of the torus, an orientation-preserving involution of $T^2$ acting trivially on $H^1(T^2;\mathbb{R})$ is conjugate to a non-trivial translation of order two. This proves the claim.
\end{proof}
From \cref{classify_under}, we obtain the following description of the induced involution on the orbifold base.

\begin{prop}\label{classify_over}
Under the assumptions of \cref{classify_under}, the induced involution
$$
\overline{\tau}:\check{\Sigma}(g;a_1,\ldots,a_n)\to \check{\Sigma}(g;a_1,\ldots,a_n)
$$
is conjugate to one of the following:
\begin{itemize}[leftmargin=0pt,itemindent=1.5em,labelsep=0.5em]
    \item if $g=0$, either $\id:S^2\to S^2$ or the $180$-degree rotation about the $z$-axis. In the latter case, after reordering the orbifold points, the orbifold data can be written as
    $$
    (a_1,\ldots,a_n)=(c_1,\ldots,c_m,c'_1,\ldots,c'_m,d_0,d_1),
    \qquad c_i=c'_i,
    $$
    where the points of orders $d_0$ and $d_1$ are fixed by $\overline{\tau}$, and the point of order $c_i$ is interchanged with the point of order $c'_i$;

    \item if $g=1$, either $\id:T^2\to T^2$ or a non-trivial translation of order two, for example
    $$
    (x,y)\mapsto \left(x+\frac{1}{2},y\right).
    $$
    In the latter case, after reordering the orbifold points, the orbifold data can be written as
    $$
    (a_1,\ldots,a_n)=(c_1,\ldots,c_m,c'_1,\ldots,c'_m),
    \qquad c_i=c'_i,
    $$
    and the points of orders $c_i$ and $c'_i$ are interchanged by $\overline{\tau}$;

    \item if $g\geq 2$, $\overline{\tau}$ is the identity on $\check{\Sigma}$.
\end{itemize}
Moreover, suppose that $\tau$ is non-trivial. Then the following holds:
\begin{itemize}[leftmargin=0pt,itemindent=1.5em,labelsep=0.5em]
    \item If $\overline{\tau}=\id$, then $\tau$ is the $(-1)$--gauge transformation of the orbifold circle bundle
    $
    S^1\to Y\to \check{\Sigma}.
    $

    \item If $g=0$ and $\overline{\tau}$ is the $180$-degree rotation, then the fiber-preserving conjugacy class of $\tau$ is determined by its action on the fibers over the two fixed points of $\overline{\tau}$. There are two conjugacy classes, which are interchanged by composition with the $(-1)$--gauge transformation.

    \item If $g=1$ and $\overline{\tau}$ is a non-trivial translation of order two, then, in the standard model considered here, there is one fiber-preserving conjugacy class of lifts.
\end{itemize}
\end{prop}

\begin{proof}
The first assertion follows from \cref{classify_under}. Indeed, the underlying involution on $\Sigma_g$ is orientation-preserving and acts trivially on $H^1(\Sigma_g;\mathbb{R})$. Hence it is conjugate to one of the involutions listed in the statement. Since an orbifold automorphism preserves the orders of orbifold points, the orbifold points are either fixed or paired as stated. 

We now classify the lifts to the Seifert circle bundle
$
S^1\to Y\to \check{\Sigma}.
$
If $\overline{\tau}=\id$, then any lift of $\overline{\tau}$ is a gauge transformation of this orbifold circle bundle. Since $\tau^2=\id$ and $\check{\Sigma}$ is connected, such a gauge transformation is multiplication by a constant element of order two in $S^1$. Thus the non-trivial lift is the $(-1)$--gauge transformation.

Suppose next that $g=0$ and that $\overline{\tau}$ is the $180$-degree rotation. Write
$$
\check{S}=\check{S}^2(c_1,\ldots,c_m,c'_1,\ldots,c'_m,d_0,d_1),
\qquad c_i=c'_i,
$$
where the points of orders $d_0,d_1$ are fixed and the points of orders $c_i,c'_i$ are interchanged. Then
$$
\check{S}_*:=\check{S}/\langle\overline{\tau}\rangle
\cong
\check{S}^2(c_1,\ldots,c_m,2d_0,2d_1).
$$
Let $p:\check{S}\to\check{S}_*$ be the quotient map. A lift of $\overline{\tau}$ to $Y$, up to fiber-preserving conjugacy, is equivalent to a descent of the orbifold circle bundle $Y\to\check{S}$ to $\check{S}_*$. Hence the set of such conjugacy classes is identified with
$$
(p^*)^{-1}([Y])\subset \operatorname{Pic}^t(\check{S}_*).
$$

Let
$$
\Theta_{\check{S}} :\operatorname{Pic}^t(\check{S})\to
\left(\bigoplus_i\Z_{c_i}\right)\oplus
\left(\bigoplus_i\Z_{c'_i}\right)\oplus \Z_{d_0}\oplus\Z_{d_1}\oplus\Q
$$
be Furuta--Steer's classification map
$$
L\mapsto
(\text{local isotropy data of }L,\ e(L)),
$$
and define $\Theta_{\check{S}_*}$ similarly. Then the pull-back is characterized by the commutative diagram
$$
\begin{CD}
\operatorname{Pic}^t(\check{S}_*) @>{p^*}>> \operatorname{Pic}^t(\check{S})\\
@V{\Theta_{\check{S}_*}}VV @VV{\Theta_{\check{S}}}V\\
\left(\bigoplus_i\Z_{c_i}\right)\oplus\Z_{2d_0}\oplus\Z_{2d_1}\oplus\Q
@>{P_S}>>
\left(\bigoplus_i\Z_{c_i}\right)\oplus
\left(\bigoplus_i\Z_{c'_i}\right)\oplus \Z_{d_0}\oplus\Z_{d_1}\oplus\Q,
\end{CD}
$$
where
$$
P_S((\zeta_i)_i,\widetilde{\alpha}_0,\widetilde{\alpha}_1,e)
=
((\zeta_i)_i,(\zeta_i)_i,\alpha_0,\alpha_1,2e),
\qquad
\alpha_j\equiv \widetilde{\alpha}_j \pmod {d_j}.
$$
The Furuta--Steer's classification map is injective. Since a lift is assumed to exist, the fiber $(p^*)^{-1}([Y])$ is non-empty. Fix one element of this fiber. The other possible preimages have the same $(\zeta_i)_i$ and the same rational Euler number, while
$
\widetilde{\alpha}_j
$
may be replaced by
$ 
\widetilde{\alpha}_j+\epsilon_jd_j, 
\epsilon_j\in\{0,1\}.
$
The Furuta--Steer's condition for an orbifold line bundle over $\check{S}_*$ says that
$$
e-\sum_i\frac{\zeta_i}{c_i}
-\frac{\widetilde{\alpha}_0}{2d_0}
-\frac{\widetilde{\alpha}_1}{2d_1}
\in \Z.
$$
Changing $\widetilde{\alpha}_j$ by $\epsilon_jd_j$ changes the left-hand side by
$
-\frac{\epsilon_0+\epsilon_1}{2}.
$
Hence the condition is preserved exactly when $\epsilon_0+\epsilon_1$ is even. Thus there are precisely two choices, namely
$$
(\epsilon_0,\epsilon_1)=(0,0)
\quad\text{and}\quad
(\epsilon_0,\epsilon_1)=(1,1).
$$
Therefore
$
|(p^*)^{-1}([Y])|=2.
$
The non-trivial choice changes the lift by multiplication by $-1$ on the fibers, and hence corresponds to composing with the $(-1)$--gauge transformation. Thus the two conjugacy classes are interchanged by the $(-1)$--gauge transformation.

Finally suppose that $g=1$ and that $\overline{\tau}$ is a non-trivial translation of order two. After conjugation, write
$
\overline{\tau}(x,y)=\left(x+\frac12,y\right).
$
The action on the base is free, so the orbifold points are paired:
$$
\check{T}=\check{T}^2(c_1,\ldots,c_m,c'_1,\ldots,c'_m),
\qquad c_i=c'_i,
$$
and
$$
\check{T}_*:=\check{T}/\langle\overline{\tau}\rangle
\cong
\check{T}^2(c_1,\ldots,c_m).
$$
Again, the pull-back is characterized by the commutative diagram
$$
\begin{CD}
\operatorname{Pic}^t(\check{T}_*) @>{p^*}>> \operatorname{Pic}^t(\check{T})\\
@V{\Theta_{\check{T}_*}}VV @VV{\Theta_{\check{T}}}V\\
\left(\bigoplus_i\Z_{c_i}\right)\oplus\Q
@>{P_T}>>
\left(\bigoplus_i\Z_{c_i}\right)\oplus
\left(\bigoplus_i\Z_{c'_i}\right)\oplus\Q,
\end{CD}
$$
where
$$
P_T((\zeta_i)_i,e)=((\zeta_i)_i,(\zeta_i)_i,2e).
$$
Since a lift exists, $[Y]$ lies in the image of $p^*$. Furuta--Steer's classification map is injective, and the above formula for $P_T$ shows that the preimage is unique: the local isotropy data determine each $\zeta_i$, and the rational Euler number determines $e$. Therefore
$
|(p^*)^{-1}([Y])|=1.
$
Thus, in the genus-one standard model, there is exactly one fiber-preserving conjugacy class of lifts. This completes the proof.
\end{proof}

\section{Construction of real spin 4-orbifolds}
\label{section : Construction of X_0}

In this section, we fix a closed orbifold surface
$$
\check{\Sigma}
=
\check{\Sigma}(g;a_1,\ldots,a_n,a'_1,\ldots,a'_n,a''_1,\ldots,a''_m)
$$
and the Seifert fibration
$$
Y
=
M(g;(a_1,b_1),\ldots,(a_n,b_n),(a'_1,b'_1),\ldots,(a'_n,b'_n),(a''_1,b''_1),\ldots,(a''_m,b''_m))
\to \check{\Sigma}.
$$
We shall construct a class of compact $4$--orbifolds $X$ with boundary $Y$. 

Let $H_g$ be a genus $g$ handlebody. Put
$$
I=[0,1],\qquad
I_1=\left[0,\frac{1}{3}\right],\qquad
I_2=\left[\frac{2}{3},1\right],\qquad
I''=\left[\frac{1}{3},1\right].
$$
Choose pairwise disjoint disks $U_i$ and $U''_j$ in $H_g$, and define
$$
W_0
=
H_g
\setminus
\left(
\coprod_{i=1}^n (U_i\times I)
\sqcup
\coprod_{j=1}^m (U''_j\times I'')
\right).
$$
The following construction was already considered by Furuta--Fukumoto--Ue \cite{fukumoto2001w}. In this paper, this construction will play a central role.  
\begin{defn}
We define a smooth compact $4$--manifold $X_0$ by
$$
\begin{aligned}
X_0
&=
W_0\times S^1
\cup_{\{\varphi_i,\varphi'_i,\varphi''_j\}}
\left(
\coprod_{i=1}^n
-\frac{\widetilde{D}_i\times I_1\times S^1}{\mathbb{Z}_{a_i}}
\sqcup
\coprod_{i=1}^n
\frac{\widetilde{D}'_i\times I_2\times S^1}{\mathbb{Z}_{a'_i}}
\sqcup
\coprod_{j=1}^m
\frac{\widetilde{D}''_j\times I_2\times S^1}{\mathbb{Z}_{a''_j}}
\right)  \\
&\cong
W_0\times S^1
\cup_{\{\phi_i,\phi'_i,\phi''_j\}}
\left(
\coprod_{i=1}^n -(D_i\times I_1\times S^1)
\sqcup
\coprod_{i=1}^n (D'_i\times I_2\times S^1)
\sqcup
\coprod_{j=1}^m (D''_j\times I_2\times S^1)
\right).
\end{aligned}
$$
Here the gluing maps in the orbifold description are given by
$$
\begin{aligned}
\varphi_i:
\frac{\partial\widetilde{D}_i\times I_1\times S^1}{\mathbb{Z}_{a_i}}
&\to
\partial U_i\times I\times S^1,
&
[z,t,w]
&\mapsto
(z^{-a_i},t,z^{-b_i}w),\\
\varphi'_i:
\frac{\partial\widetilde{D}'_i\times I_2\times S^1}{\mathbb{Z}_{a'_i}}
&\to
\partial U_i\times I\times S^1,
&
[z,t,w]
&\mapsto
(z^{a'_i},t,z^{-b'_i}w),\\
\varphi''_j:
\frac{\partial\widetilde{D}''_j\times I_2\times S^1}{\mathbb{Z}_{a''_j}}
&\to
\partial U''_j\times I''\times S^1,
&
[z,t,w]
&\mapsto
(z^{a''_j},t,z^{-b''_j}w).
\end{aligned}
$$
Equivalently, after passing to smooth solid tori, the gluing maps are
$$
\begin{aligned}
\phi_i:
\partial D_i\times I_1\times S^1
&\to
\partial U_i\times I\times S^1,
&
(\xi,t,\eta)
&\mapsto
(\xi^{-a_i}\eta^{-\nu_i},t,\xi^{-b_i}\eta^{\rho_i}),\\
\phi'_i:
\partial D'_i\times I_2\times S^1
&\to
\partial U_i\times I\times S^1,
&
(\xi',t',\eta')
&\mapsto
(\xi'{}^{a'_i}\eta'{}^{\nu'_i},t',\xi'{}^{-b'_i}\eta'{}^{\rho'_i}),\\
\phi''_j:
\partial D''_j\times I_2\times S^1
&\to
\partial U''_j\times I''\times S^1,
&
(\xi'',t'',\eta'')
&\mapsto
(\xi''{}^{a''_j}\eta''{}^{\nu''_j},t'',\xi''{}^{-b''_j}\eta''{}^{\rho''_j}).
\end{aligned}
$$
Choose the integers
\(\nu_i,\rho_i,\nu'_i,\rho'_i,\nu''_j,\rho''_j\)
so that
\[
a_i\rho_i+b_i\nu_i=1,
\qquad
a'_i\rho'_i+b'_i\nu'_i=1,
\qquad
a''_j\rho''_j+b''_j\nu''_j=1.
\]
With these choices, the displayed gluing maps are diffeomorphisms of
the boundary tori.
\end{defn}

We call the regions $U_i\times I$ {\it full tunnels}, and the regions $U''_j\times I''$ {\it half tunnels}.

The compact $4$--manifold $X_0$ has a natural $S^1$--action, given by rotation of the last $S^1$--factor on each piece. Hence $X_0$ can be regarded as the total space of an orbifold $U(1)$--bundle
$
X_0\to N:=X_0/S^1.
$
The quotient $N$ has a natural $3$--dimensional compact orbifold structure with codimension two singular strata and with boundary.

The construction of $X_0$ is schematically illustrated in
\Cref{fig:picture-X0}.  The figure should be understood at the level of the
quotient $N=X_0/S^1$.  The ambient three-dimensional object represents the
handlebody $H_g$.  The red cylinders represent the full tunnels
$U_i\times I$, while the green cylinders represent the half tunnels
$U''_j\times I''$.  After removing these tunnels and gluing in the corresponding
solid pieces, we obtain the compact $4$-manifold $X_0$ with the natural
$S^1$-action described above.

\begin{figure}[htbp]
\centering
\includegraphics[width=0.42\textwidth]{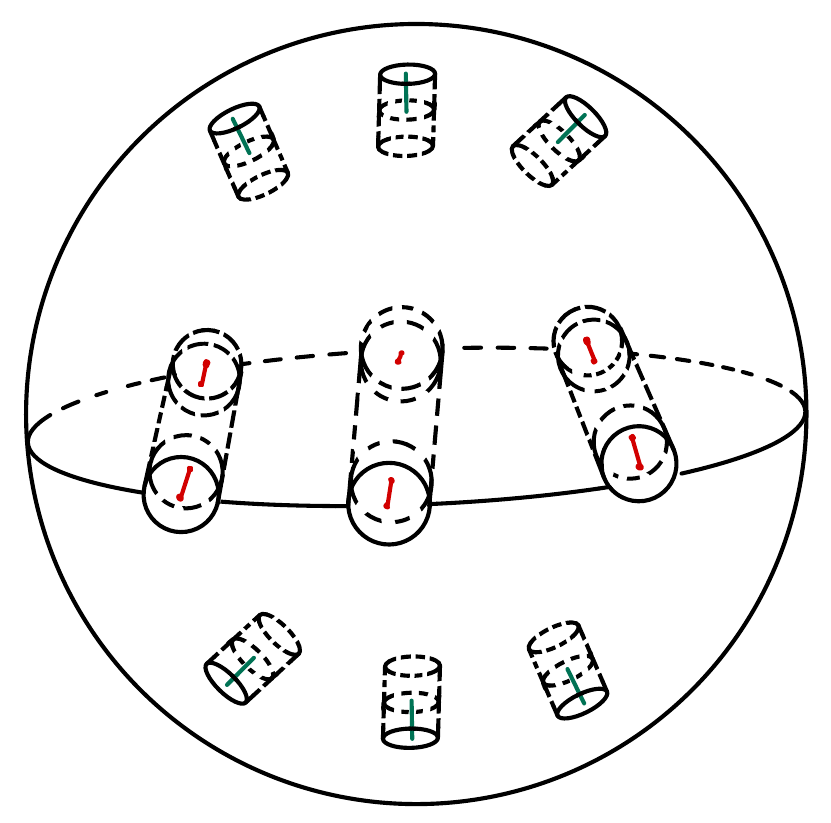}
\caption{
A schematic picture of the construction of $X_0$, drawn at the level of
the quotient $N=X_0/S^1$.  The red lines indicate 1-dimensional orbifold singularities for the full tunnels
$U_i\times I$, and the green lines indicate 1-dimensional orbifold singularities of the half tunnels
$U''_j\times I''$.
}
\label{fig:picture-X0}
\end{figure}

The boundary components of $X_0$ are illustrated in \Cref{fig:X0-boundary}.
The bottom component represents the original Seifert manifold $Y=\partial_-X_0$.
The other boundary components arise from the unused ends of the tunnels.  More
precisely, a full tunnel connecting the two Seifert pieces with data
$(a_i,b_i)$ and $(a'_i,b'_i)$ contributes the lens space
$-L(\widetilde a_i,\widetilde b_i)$, while a half tunnel contributes
$-L(\widetilde a''_j,\widetilde b''_j)$.  Thus these components form
$\partial_+X_0$.  Capping them off by the corresponding cyclic quotient
$4$-balls gives the orbifold $\check{X}$ with boundary $Y$.

\begin{figure}[htbp]
    \centering
    \includegraphics[width=0.22\textwidth]{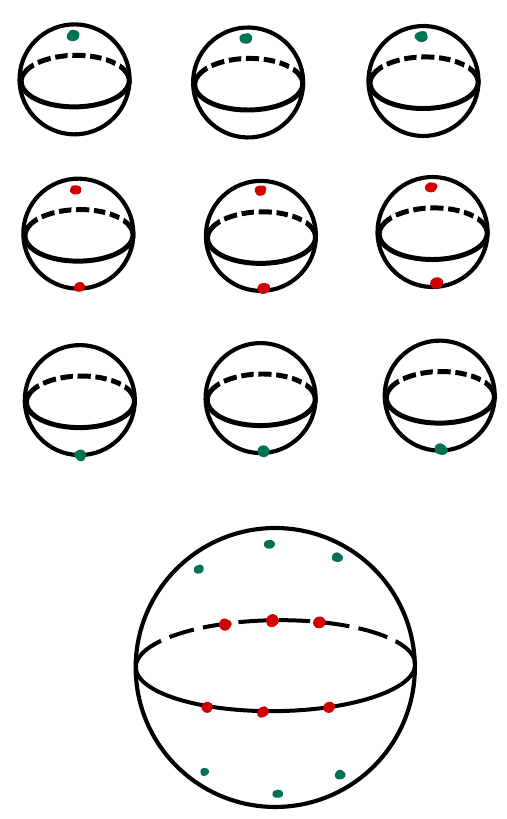}
    \caption{
    A schematic picture of the boundary of $X_0$. 
    }
    \label{fig:X0-boundary}
\end{figure}

For later use, we record the lens space parameters of the remaining boundary
components of $X_0$.  For a full tunnel connecting the two Seifert pieces
with data $(a_i,b_i)$ and $(a'_i,b'_i)$, define integers
$\widetilde a_i,\widetilde b_i,\widetilde\nu_i,\widetilde\rho_i$ by
\[
\begin{pmatrix}
-\widetilde b_i & \widetilde\rho_i\\
\widetilde a_i & \widetilde\nu_i
\end{pmatrix}
=
\begin{pmatrix}
a'_i & \nu'_i\\
-b'_i & \rho'_i
\end{pmatrix}^{-1}
\begin{pmatrix}
-a_i & -\nu_i\\
-b_i & \rho_i
\end{pmatrix}.
\]
The corresponding boundary component is the lens space
$
-L(\widetilde a_i,\widetilde b_i).
$

For a half tunnel corresponding to the Seifert datum $(a''_j,b''_j)$, we use
the convention that the missing end has data
$
(a,b,\nu,\rho)=(1,0,0,1).
$
Thus we define
\[
\begin{pmatrix}
-\widetilde b''_j & \widetilde\rho''_j\\
\widetilde a''_j & \widetilde\nu''_j
\end{pmatrix}
=
\begin{pmatrix}
a''_j & \nu''_j\\
-b''_j & \rho''_j
\end{pmatrix}^{-1}
\begin{pmatrix}
-1 & 0\\
0 & 1
\end{pmatrix}.
\]
The corresponding boundary component is
$
-L(\widetilde a''_j,\widetilde b''_j).
$ Therefore, the boundary of $X_0$ is
$$
\partial X_0
=
\left(
\bigcup_{i=1}^n -L(\widetilde{a}_i,\widetilde{b}_i)
\right)
\cup
\left(
\bigcup_{j=1}^m -L(\widetilde{a}''_j,\widetilde{b}''_j)
\right)
\cup
Y.
$$
The restricted $S^1$--action preserves each boundary component, and its restriction to $Y$ is the given Seifert $S^1$--action on $Y$.

We denote
$$
\partial_+X_0
=
\left(
\bigcup_{i=1}^n -L(\widetilde{a}_i,\widetilde{b}_i)
\right)
\cup
\left(
\bigcup_{j=1}^m -L(\widetilde{a}''_j,\widetilde{b}''_j)
\right),
\qquad
\partial_-X_0=Y.
$$
We note that, once the Seifert data are fixed, the diffeomorphism type of $X_0$ is determined by the choices of the integers appearing in the gluing maps. We write
$$
\boldsymbol{\nu}
=
(\nu_1,\ldots,\nu_n,\nu'_1,\ldots,\nu'_n,\nu''_1,\ldots,\nu''_m),
\qquad
\boldsymbol{\rho}
=
(\rho_1,\ldots,\rho_n,\rho'_1,\ldots,\rho'_n,\rho''_1,\ldots,\rho''_m),
$$
and denote the resulting manifold by
$
X_0(\boldsymbol{\nu},\boldsymbol{\rho}).
$
We call $(\boldsymbol{\nu},\boldsymbol{\rho})$ the {\it gluing data} of $X_0$.

For each lens space component $L(\widetilde{a},\widetilde{b})$ of $\partial_+X_0$, let
$
C(\widetilde{a},\widetilde{b})
$
denote the orbifold $4$--ball whose boundary is $L(\widetilde{a},\widetilde{b})$. More explicitly, it is the quotient of $D^4\subset \mathbb{C}^2$ by the standard cyclic action corresponding to the lens space $L(\widetilde{a},\widetilde{b})$.

\begin{defn}
We define the compact oriented $4$--orbifold
$
\check{X}=\check{X}(\boldsymbol{\nu},\boldsymbol{\rho})
$
by capping off the lens space components of $\partial_+X_0$:
$$
\check{X}
:=
X_0
\cup_{\partial_+X_0}
\left(
\coprod_{i=1}^n C(\widetilde{a}_i,\widetilde{b}_i)
\sqcup
\coprod_{j=1}^m C(\widetilde{a}''_j,\widetilde{b}''_j)
\right).
$$
Here $C(\widetilde{a}_i,\widetilde{b}_i)$ is attached along the boundary component $-L(\widetilde{a}_i,\widetilde{b}_i)$ of $X_0$, and similarly for $C(\widetilde{a}''_j,\widetilde{b}''_j)$. Thus
$
\partial \check{X}=Y.
$
\end{defn}

For simplicity, we assume $m=0$. For the general case, see \cref{general_homology}. 

\begin{prop}\label{homologyofx0}
The homology groups of $X_0$ are described as follows:
$$
H_2(X_0;\Z)
\cong
\bigoplus_{l=1}^g \Z\langle \beta_l\times \delta\rangle
\oplus
\left\{
\sum_{i=1}^n s_i(\bar{a}'_i m_i+\bar{a}_i m'_i)
\ \middle|\
\sum_{i=1}^n(\bar{a}'_i b_i+\bar{a}_i b'_i)s_i=0
\right\},
$$
and
$$
H_1(X_0;\Z)
\cong
\left\langle
\gamma'_1,\ldots,\gamma'_n,\delta,\beta_1,\ldots,\beta_g
\ \middle|\
-a_i\gamma'_i-b_i\delta=0,\quad
a'_i\gamma'_i-b'_i\delta=0\;\;(i=1,\ldots,n)
\right\rangle .
$$
Here $\delta$ denotes the class of a regular fiber of the $S^1$--fibration
$
X_0\to X_0/S^1=N,
$
and $\beta_1,\ldots,\beta_g$ are the classes coming from a chosen half of a symplectic basis of $H_1(\Sigma_g;\Z)$. The classes $m_i$ and $m'_i$ denote the meridians of $\partial D_i$ and $\partial D'_i$, respectively. Finally,
$$
\bar{a}_i=\frac{a_i}{d_i},
\qquad
\bar{a}'_i=\frac{a'_i}{d_i}, \text{ where  }\gcd(a_i,a'_i) = d_i. 
$$
\end{prop}

\begin{rem}
Set $\Delta_i:=\bar a'_i b_i+\bar a_i b'_i .$ We usually assume that $ (\Delta_1,\ldots,\Delta_n)\neq(0,\ldots,0)$.  Then \[ \left\{ \sum_{i=1}^n s_i(\bar a'_i m_i+\bar a_i m'_i) \ \middle|\ \sum_{i=1}^n\Delta_i s_i=0 \right\} \cong \Z^{n-1}. \] Consequently, $H_2(X_0;\Z)\cong \Z^{g+n-1}$. Let \[ A_X:= \left\langle \gamma'_1,\ldots,\gamma'_n,\delta \ \middle|\ -a_i\gamma'_i-b_i\delta=0,\quad a'_i\gamma'_i-b'_i\delta=0\;\;(i=1,\ldots,n) \right\rangle . \] Then $ H_1(X_0;\Z)\cong A_X\oplus \Z^g$. Moreover, if $ (\Delta_1,\ldots,\Delta_n)\neq(0,\ldots,0)$,  then $A_X$ is finite. In particular, $H_1(X_0;\Q)\cong \Q^g $. Whenever we consider the orbifold $\check X$, we assume that $
\Delta_i\neq 0$ for every $i$.
\end{rem}

\begin{proof}
All homology groups are taken with $\Z$--coefficients. Let $\mu_i$ be the meridian of the $i$-th tunnel $C_i$ in $W_0$. We apply the Mayer--Vietoris sequence to the decomposition of $X_0$ into $W_0\times S^1$ and the solid pieces $D_i\times I_1\times S^1$, $D'_i\times I_2\times S^1$. The relevant part is
$$
H_2(T)\xrightarrow{i_2}
H_2(W_0\times S^1)\oplus H_2(B)
\xrightarrow{j_2}
H_2(X_0)
\xrightarrow{\partial_2}
H_1(T)
\xrightarrow{i_1}
H_1(W_0\times S^1)\oplus H_1(B)
\xrightarrow{j_1}
H_1(X_0)\to 0,
$$
where
$$
T=\coprod_{i=1}^n(\partial D_i\times I_1\times S^1)
\sqcup
\coprod_{i=1}^n(\partial D'_i\times I_2\times S^1)
\text{ and }
B=
\coprod_{i=1}^n(D_i\times I_1\times S^1)
\sqcup
\coprod_{i=1}^n(D'_i\times I_2\times S^1).
$$
We use the generators
$$
H_2(T)=\left\langle m_i\times l_i,\; m'_i\times l'_i\right\rangle,
\qquad
H_2(W_0\times S^1)=\left\langle \mu_i\times\delta,\; \beta_l\times\delta\right\rangle,
$$
and
$$
H_1(T)=\left\langle m_i,l_i,m'_i,l'_i\right\rangle,
\qquad
H_1(W_0\times S^1)\oplus H_1(B)
=
\left\langle \mu_i,\beta_l,\delta,l_i,l'_i\right\rangle .
$$
Here $l=1,\ldots,g$. The map $i_2$ is given by
$$
m_i\times l_i\mapsto \mu_i\times\delta,\qquad
m'_i\times l'_i\mapsto \mu_i\times\delta.
$$
Hence
$
\operatorname{coker} i_2
\cong
\bigoplus_{l=1}^g \Z\langle \beta_l\times\delta\rangle .
$

The map $i_1$ is determined by the gluing maps:
$$
\begin{aligned}
m_i &\mapsto -a_i\mu_i-b_i\delta,\\
l_i &\mapsto -\nu_i\mu_i+\rho_i\delta+l_i,\\
m'_i &\mapsto a'_i\mu_i-b'_i\delta,\\
l'_i &\mapsto \nu'_i\mu_i+\rho'_i\delta+l'_i.
\end{aligned}
$$
Let
$$
x=\sum_{i=1}^n(p_i m_i+q_i l_i+p'_i m'_i+q'_i l'_i)\in H_1(T).
$$
If $x\in \ker i_1$, then the $l_i$ and $l'_i$ components imply
$
q_i=0, q'_i=0$
for all $i$. The remaining components give
$$
p_i a_i-p'_i a'_i=0\ (i=1,\ldots,n), \text{ and }
\sum_{i=1}^n(p_i b_i+p'_i b'_i)=0.
$$
Put
$
d_i=\gcd(a_i,a'_i), 
\bar{a}_i=\frac{a_i}{d_i}, 
\bar{a}'_i=\frac{a'_i}{d_i}.
$
Then the solutions of $p_i a_i=p'_i a'_i$ are
$
(p_i,p'_i)=s_i(\bar{a}'_i,\bar{a}_i),
\ s_i\in \Z.
$
Therefore
$$
\ker i_1
=
\left\{
\sum_{i=1}^n s_i(\bar{a}'_i m_i+\bar{a}_i m'_i)
\ \middle|\
\sum_{i=1}^n(\bar{a}'_i b_i+\bar{a}_i b'_i)s_i=0
\right\}.
$$

By exactness, we have a short exact sequence
$$
0\to \operatorname{coker} i_2\to H_2(X_0)\to \ker i_1\to 0.
$$
One can see this sequence splits. Hence
$$
H_2(X_0)
\cong
\bigoplus_{l=1}^g \Z\langle \beta_l\times \delta\rangle
\oplus
\left\{
\sum_{i=1}^n s_i(\bar{a}'_i m_i+\bar{a}_i m'_i)
\ \middle|\
\sum_{i=1}^n(\bar{a}'_i b_i+\bar{a}_i b'_i)s_i=0
\right\}.
$$

Finally, from the cokernel of $i_1$ we obtain
$$
H_1(X_0)
\cong
\left\langle
\mu_i,\beta_l,\delta,l_i,l'_i
\ \middle|\
-a_i\mu_i-b_i\delta=0,\;
-\nu_i\mu_i+\rho_i\delta+l_i=0,\;
a'_i\mu_i-b'_i\delta=0,\;
\nu'_i\mu_i+\rho'_i\delta+l'_i=0
\right\rangle .
$$
Eliminating $l_i$ and $l'_i$, this becomes
$$
H_1(X_0)
\cong
\left\langle
\mu_i,\delta,\beta_l
\ \middle|\
-a_i\mu_i-b_i\delta=0,\;
a'_i\mu_i-b'_i\delta=0
\right\rangle .
$$
This is the desired description.
\end{proof}

\begin{rem}\label{general_homology}
A half-tunnel can be incorporated into the above formula by the following convention. Suppose that, for some $j$, we replace a full tunnel by a half-tunnel
$
U_j\times \left[\frac{1}{3},1\right].
$
Then we formally put
$
a_j=1, b_j=0.
$
The corresponding relations in $H_1(X_0)$ become
$
-\mu_j=0, a'_j\mu_j-b'_j\delta=0,
$
and hence reduce to
$
\mu_j=0, b'_j\delta=0.
$

For $H_2(X_0)$, we have
$
d_j=\gcd(a_j,a'_j)=1,
\bar{a}_j=1,
\bar{a}'_j=a'_j.
$
Thus the corresponding generator
$
\bar{a}'_j m_j+\bar{a}_j m'_j
$
formally becomes
$
a'_j m_j+m'_j,
$
and the corresponding coefficient in the relation
$
\sum_i(\bar{a}'_i b_i+\bar{a}_i b'_i)s_i=0
$
becomes $b'_j$. Since a half-tunnel has no corresponding generator $m_j$, the term $a'_j m_j+m'_j$ is replaced simply by $m'_j$.
\end{rem}

\begin{cor}\label{cor:H1-X0-Z2}
The inclusion map 
$$
\iota : H_1(Y;\Z)\to H_1(X_0;\Z)
$$
is surjective. 
In particular, if $H_1(Y;\Z_2)=0$, then $H_1(X_0;\Z_2)=0$.
\end{cor}

\begin{proof}
We first show that the inclusion $Y\hookrightarrow X_0$ induces a surjection on first homology. The first homology of $Y$ has the presentation
$$
H_1(Y;\Z)
\cong
\left\langle
\gamma_i,\gamma'_i,\delta,\alpha_l,\beta_l
\ \middle|\
\sum_{i=1}^n(\gamma_i+\gamma'_i)=0,\quad
a_i\gamma_i-b_i\delta=0,\quad
a'_i\gamma'_i-b'_i\delta=0
\right\rangle ,
$$
where $l=1,\ldots,g$ and $\alpha_l,\beta_l$ is a symplectic basis of the genus $g$ part of the base.

Put
$
\eta_i=\gamma_i+\gamma'_i.
$
Then $\gamma_i=\eta_i-\gamma'_i$, and the relation
$
a_i\gamma_i-b_i\delta=0
$
becomes
$
a_i\eta_i-a_i\gamma'_i-b_i\delta=0.
$
The classes $\eta_i$ are killed in $X_0$, and one half of the symplectic basis, say the classes $\alpha_l$, is also killed by the handlebody. Therefore the quotient of $H_1(Y;\Z)$ by the subgroup generated by $\eta_i$ and $\alpha_l$ has the presentation
$$
\left\langle
\gamma'_i,\delta,\beta_l
\ \middle|\
-a_i\gamma'_i-b_i\delta=0,\quad
a'_i\gamma'_i-b'_i\delta=0
\right\rangle .
$$
By \cref{homologyofx0}, this is precisely $H_1(X_0;\Z)$. Hence the inclusion-induced map
$
H_1(Y;\Z)\to H_1(X_0;\Z)
$
is surjective. 
\end{proof}
We next compute the rational intersection form of $X_0$. Let
\[
\operatorname{PD}\colon
H_2(X_0;\mathbb Q)
\overset{\cong}{\longrightarrow}
H^2(X_0,\partial X_0;\mathbb Q)
\]
denote the Poincaré--Lefschetz duality isomorphism, and let $
\iota\colon
H^2(X_0,\partial X_0;\mathbb Q)
\longrightarrow
H^2(X_0;\mathbb Q)$
be the natural map. We define
\[
Q_0\colon
H_2(X_0;\mathbb Q)\otimes H_2(X_0;\mathbb Q)
\longrightarrow\mathbb Q
\]
by
\[
Q_0(x,y)
=
\left\langle
\iota\bigl(\operatorname{PD}(x)\bigr)
\smile
\operatorname{PD}(y),
[X_0,\partial X_0]
\right\rangle .
\]
This is the usual rational intersection form of the compact oriented
$4$--manifold $X_0$, and it is not necessarily nondegenerate.

\begin{prop}\label{prop:Q-intersection-form-on-X0}
Assume $m=0$. For each $i=1,\ldots,n$, set
\[
d_i:=\gcd(a_i,a'_i),
\qquad
\bar a_i:=\frac{a_i}{d_i},
\qquad
\bar a'_i:=\frac{a'_i}{d_i},
\]
as in \Cref{homologyofx0}, and put $\Delta_i:=\bar a_i b'_i+\bar a'_i b_i$.
After reordering the indices, assume that
$
\Delta_n\neq 0.
$
Let
\[
        E_i:=\bar a'_i m_i+\bar a_i m'_i
        \in H_2(X_0;\Q)
        \qquad
        (i=1,\ldots,n).
\]
For $i=1,\ldots,n-1$, define $S_i:=E_i-\frac{\Delta_i}{\Delta_n}E_n$.
Then
$
S_1,\ldots,S_{n-1}
$
form a basis of
\[
V:=
\left\{
\sum_{i=1}^n s_iE_i
\;\middle|\;
\sum_{i=1}^n \Delta_i s_i=0
\right\}
\subset H_2(X_0;\Q).
\]
With respect to this basis, the restriction of the relative rational
intersection form
$
Q_0
$
to $V$ is given by
\[
Q_0(S_i,S_j)
=
-\delta_{ij}\Delta_i \bar a'_i a_i
-
\frac{\Delta_i\Delta_j}{\Delta_n}\bar a'_n a_n
\qquad
(1\leq i,j\leq n-1).
\]
Moreover, $
\bigoplus_{l=1}^g \Q\langle \beta_l\times\delta\rangle
\subset
\operatorname{Rad}(Q_0)$,
where
\[
\operatorname{Rad}(Q_0)
:=
\left\{
x\in H_2(X_0;\Q)
\;\middle|\;
Q_0(x,y)=0
\text{ for all } y\in H_2(X_0;\Q)
\right\}.
\]
In particular, if $g=0$, then $V=H_2(X_0;\Q)$ and the above matrix is the
whole relative rational intersection form on $H_2(X_0;\Q)$.
\end{prop}

\begin{proof}
We use the notation
\[
\Delta_i=\bar a_i b_i'+\bar a_i' b_i,
\qquad
E_i=\bar a_i' m_i+\bar a_i m'_i .
\]
Recall from the computation of $H_2(X_0;\Q)$ that the part generated by the
full-tunnel pieces is
\[
V=
\left\{
\sum_{i=1}^n s_iE_i
\;\middle|\;
\sum_{i=1}^n \Delta_i s_i=0
\right\}.
\]
Since $\Delta_n\neq0$, every element of $V$ can be written uniquely as
\[
\sum_{i=1}^{n-1} s_i
\left(
E_i-\frac{\Delta_i}{\Delta_n}E_n
\right).
\]
Thus $
S_i:=E_i-\frac{\Delta_i}{\Delta_n}E_n,
\ 
i=1,\ldots,n-1$,
form a basis of $V$.

We now compute the intersection form with respect to this basis. Fix
\[
t_0=\frac{1}{6},\qquad t_1=\frac{5}{6},
\qquad I=[t_0,t_1].
\]
Choose points $x_i\in\partial U_i$, and choose paths in $W_0\times I$
\[
\alpha_{i0}^n:x_i\times t_0\to x_n\times t_1, 
\alpha_{i1}^n:x_i\times t_1\to x_n\times t_1,
\text{ and }
\alpha_{n0}^n:x_n\times t_0\to x_n\times t_1.
\]
Let $\delta$ denote the $S^1$--factor. We represent $S_i$ by the following $2$--cycle:
\[
\begin{aligned}
\Sigma_i
:={}&
\bar a_i'[D_i\times t_0\times 1]
+
\bar a_i[D_i'\times t_1\times 1]
-
\frac{\Delta_i}{\Delta_n}
\left(
\bar a_n'[D_n\times t_0\times 1]
+
\bar a_n[D_n'\times t_1\times 1]
\right)
\\
&+
\bar a_i'a_i\,\partial U_i\times I\times 1
-
\frac{\Delta_i}{\Delta_n}\bar a_n'a_n\,\partial U_n\times I\times 1
-
\bar a_i'b_i\,\alpha_{i0}^n\times\delta
-
\bar a_i b_i'\,\alpha_{i1}^n\times\delta
+
\frac{\Delta_i}{\Delta_n}\bar a_n'b_n\,\alpha_{n0}^n\times\delta.
\end{aligned}
\]
Indeed, the boundary of the disk part is
\[
\begin{aligned}
&
\bar a_i'a_i
\left(
-\gamma_i\times t_0\times1+\gamma_i\times t_1\times1
\right)
-
\bar a_i'b_i\, x_i\times t_0\times\delta
-
\bar a_i b_i'\, x_i\times t_1\times\delta
\\
&\quad
-
\frac{\Delta_i}{\Delta_n}
\bar a_n'a_n
\left(
-\gamma_n\times t_0\times1+\gamma_n\times t_1\times1
\right)
+
\frac{\Delta_i}{\Delta_n}
\left(
\bar a_n'b_n\, x_n\times t_0\times\delta
+
\bar a_n b_n'\, x_n\times t_1\times\delta
\right).
\end{aligned}
\]
The annulus terms cancel the $\gamma_i$ and $\gamma_n$ contributions.  The path
terms cancel the remaining $\delta$--terms.  The cancellation at
$x_n\times t_1$ uses precisely $\Delta_n=\bar a_n'b_n+\bar a_n b_n'$.
Therefore $\Sigma_i$ is a cycle representing $S_i$.

To compute intersections, perturb the second representative.  Let
\[
t_0''=\frac{1}{7},\qquad t_1''=\frac{6}{7},
\qquad I''=[t_0'',t_1''],
\]
and choose a point $1''\in S^1$ with $1''\neq1$.  Let $\Sigma_j''$ be the cycle
obtained from $\Sigma_j$ by replacing $t_0,t_1,I,1$ and the paths
$\alpha^n_{j0},\alpha^n_{j1},\alpha^n_{n0}$ by
$t_0'',t_1'',I'',1''$ and nearby paths
\[
\alpha^{\prime\prime n}_{j0},\qquad
\alpha^{\prime\prime n}_{j1},\qquad
\alpha^{\prime\prime n}_{n0}.
\]

After this perturbation, the only intersections contributing to
$\Sigma_i\bullet\Sigma_j''$ are intersections between the path terms in
$\Sigma_i$ and the annulus terms in $\Sigma_j''$.  We arrange the paths so that
\[
(\alpha_{i0}^n\times\delta)\bullet
(\partial U_j\times I''\times1'')
=
\delta_{ij},
\]
\[
(\alpha_{i1}^n\times\delta)\bullet
(\partial U_j\times I''\times1'')
=
\delta_{ij},
\]
and
\[
(\alpha_{i0}^n\times\delta)\bullet
(\partial U_n\times I''\times1'')
=
(\alpha_{i1}^n\times\delta)\bullet
(\partial U_n\times I''\times1'')
=
-1.
\]
Moreover, $
(\alpha_{n0}^n\times\delta)\bullet
(\partial U_n\times I''\times1'')=0$,
and all other intersections are zero.

Hence
\[
\begin{aligned}
Q_0(S_i,S_j)
&=
\Sigma_i\bullet\Sigma_j''
\\
&=
(-\bar a_i'b_i-\bar a_i b_i')
\left(
\delta_{ij}\bar a_j'a_j
+
\frac{\Delta_j}{\Delta_n}\bar a_n'a_n
\right)
\\
&=
-\Delta_i
\left(
\delta_{ij}\bar a_j'a_j
+
\frac{\Delta_j}{\Delta_n}\bar a_n'a_n
\right)
\\
&=
-\delta_{ij}\Delta_i\bar a_i'a_i
-
\frac{\Delta_i\Delta_j}{\Delta_n}\bar a_n'a_n.
\end{aligned}
\]
This is the desired formula.

Finally, the classes $\beta_l\times\delta$ are represented entirely in the
product part $W_0\times S^1$.  By choosing the paths $\alpha_{i0}^n$,
$\alpha_{i1}^n$, and $\alpha_{n0}^n$ disjoint from the curves $\beta_l$ in
$W_0$, the representatives $\Sigma_i$ can be made disjoint from
$\beta_l\times\delta$.  Also, two classes of the form $\beta_l\times\delta$
can be made disjoint after a small perturbation in the $S^1$--direction.
Therefore
\[
\bigoplus_{l=1}^g \Q\langle \beta_l\times\delta\rangle
\subset
\operatorname{Rad}(Q_0),
\]
where
$
\operatorname{Rad}(Q_0)
=
\left\{
x\in H_2(X_0;\Q)
\;\middle|\;
Q_0(x,y)=0
\text{ for all } y\in H_2(X_0;\Q)
\right\}$.

When $g=0$, there are no such classes $\beta_l\times\delta$, and the summand
$V$ is all of $H_2(X_0;\Q)$.  Hence the displayed matrix gives the whole
relative rational intersection form on $H_2(X_0;\Q)$.
\end{proof}

We next discuss spin structures on $X_0$. In this subsection, we focus on the case $g=0$ and $m=0$.

Recall that
\[
X_0(\boldsymbol{\nu},\boldsymbol{\rho})
=
W_0\times S^1
\cup_{\{\phi_i,\phi'_i\}_{i=1}^n}
\coprod_{i=1}^n
\left(
-\bigl(D_i\times I_1\times S^1\bigr)
\sqcup
\bigl(D'_i\times I_2\times S^1\bigr)
\right),
\]
where $
I_1=\left[0,\frac13\right], 
I_2=\left[\frac23,1\right]$,
and
\[
\begin{aligned}
\phi_i(\xi,t,\eta)
&=
\left(
\xi^{-a_i}\eta^{-\nu_i},
t,
\xi^{-b_i}\eta^{\rho_i}
\right),\\
\phi'_i(\xi',t,\eta')
&=
\left(
\xi'{}^{a'_i}\eta'{}^{\nu'_i},
t,
\xi'{}^{-b'_i}\eta'{}^{\rho'_i}
\right).
\end{aligned}
\]

For each oriented circle, we identify its two spin structures with
$\Z_2$ by the convention
\[
0=\text{the  nonbounding or Lie group spin structure},
\qquad
1=\text{the bounding spin structure}.
\]

Let
\[
T_i^-=
\partial U_i\times\left\{\frac13\right\}\times S^1,
\qquad
T_i^+=
\partial U_i\times\left\{\frac23\right\}\times S^1
\]
be the two gluing tori in the product piece. On each of these tori, let
\[
\alpha_i=
\partial U_i\times\{t\}\times\{1\},
\qquad
\delta=
\{1\}\times\{t\}\times S^1
\]
denote, respectively, the boundary-circle direction of the $i$-th
tunnel and the regular $S^1$-fiber direction.

\begin{defn}
For a spin structure $\mathfrak{s}$ on $X_0$, define $\lambda_i,\mu\in\Z_2$
by requiring that the restriction of $\mathfrak{s}$ to either
$T_i^-$ or $T_i^+$ is represented by $(\lambda_i,\mu)$
with respect to the ordered circle directions $(\alpha_i,\delta)$.
\end{defn}
The parameter $\lambda_i$ is the same at the two ends of the tunnel
because the corresponding $\alpha_i$-circles cobound an annulus in
$\partial U_i\times I\times\{1\}$. The parameter $\mu$ is independent
of $i$ because it records the restriction to the global $S^1$-fiber.

On $
\partial D_i\times\left\{\frac13\right\}\times S^1$,
let
\[
m_i=
\partial D_i\times\left\{\frac13\right\}\times\{1\},
\qquad
\ell_i=
\{1\}\times\left\{\frac13\right\}\times S^1.
\]
Since $m_i$ bounds the disk $D_i$, the restriction of $\mathfrak{s}$
to this torus is represented by $(1,e_i)$
with respect to $(m_i,\ell_i)$, for a uniquely determined
$e_i\in\Z_2$.

Similarly, on $\partial D'_i\times\left\{\frac23\right\}\times S^1$,
let
\[
m'_i=
\partial D'_i\times\left\{\frac23\right\}\times\{1\},
\qquad
\ell'_i=
\{1\}\times\left\{\frac23\right\}\times S^1.
\]
The restriction of $\mathfrak{s}$ to this torus is represented by $(1,e'_i)$
with respect to $(m'_i,\ell'_i)$, for a uniquely determined
$e'_i\in\Z_2$.

\begin{prop}\label{prop:spin structures on X_0}
The set of isomorphism classes of spin structures on
$X_0(\boldsymbol{\nu},\boldsymbol{\rho})$ is naturally identified with
the set of tuples
\[
\lambda_i,e_i,e'_i,\mu\in\Z_2,
\qquad i=1,\ldots,n,
\]
satisfying
\[
(\lambda_i,\mu)
\begin{pmatrix}
-a_i & -\nu_i\\
-b_i & \rho_i
\end{pmatrix}
\equiv
(1,e_i)
 \text{ and }
(\lambda_i,\mu)
\begin{pmatrix}
a'_i & \nu'_i\\
-b'_i & \rho'_i
\end{pmatrix}
\equiv
(1,e'_i)
\pmod2
\]
for every $i=1,\ldots,n$.
\end{prop}

\begin{proof}
All restrictions of spin structures to the circle factors below are
understood using the product normal framings determined by the displayed
product decompositions.

We first recall the action of a torus diffeomorphism on spin structures.
Let $T^2=S^1\times S^1$
with its ordered coordinate circles. A spin structure on \(T^2\) is represented by an element
$(\epsilon_1,\epsilon_2)\in\Z_2^2$.
The spin structure represented by \((0,0)\) is the Lie-group spin
structure induced by the product framing. Equivalently, it is the unique
odd spin structure on \(T^2\). It is therefore preserved by every
orientation-preserving diffeomorphism of \(T^2\).

Consequently, after taking \((0,0)\) as the origin of the affine space
of spin structures, the pullback action is the usual action on
\(H^1(T^2;\Z_2)\). More precisely, suppose that $f:T^2\longrightarrow T^2$
is an orientation-preserving diffeomorphism and that $f_*:H_1(T^2;\Z)\longrightarrow H_1(T^2;\Z)$
is represented by
\[
A=
\begin{pmatrix}
p&q\\
r&s
\end{pmatrix}
\in SL(2,\Z)
\]
with respect to the ordered coordinate circles. Then $
f^*(\epsilon_1,\epsilon_2)
=
(\epsilon_1,\epsilon_2)A
\pmod 2$.

We apply this observation to the decomposition of \(X_0\). Put $
P=W_0\times S^1$,
and denote the solid pieces by
\[
B_i^-=D_i\times I_1\times S^1,
\qquad
B_i^+=D'_i\times I_2\times S^1.
\]
Since \(g=0\) and \(m=0\), the manifold \(W_0\) is a handlebody of
genus \(n\). The classes represented by $
\alpha_1,\ldots,\alpha_n$
form a basis of \(H_1(W_0;\Z_2)\). Hence
$
H^1(P;\Z_2)
\cong
\bigoplus_{i=1}^n\Z_2\langle\alpha_i^*\rangle
\oplus
\Z_2\langle\delta^*\rangle$.

Choose a reference spin structure on \(P\) whose restrictions to all
the circles \(\alpha_i\) and to the regular fiber \(\delta\) are
nonbounding. Since the set of spin structures on \(P\) is an affine
space over \(H^1(P;\Z_2)\), every spin structure on \(P\) is uniquely
specified by $
\lambda_1,\ldots,\lambda_n,\mu\in\Z_2$.
Its restriction to either of the two tori
\[
T_i^-=
\partial U_i\times\left\{\frac13\right\}\times S^1,
\qquad
T_i^+=
\partial U_i\times\left\{\frac23\right\}\times S^1
\]
is represented by $(\lambda_i,\mu)$
with respect to the ordered circle directions $(\alpha_i,\delta)$.
The same value of \(\lambda_i\) occurs at the two ends because the two
copies of \(\alpha_i\) cobound the annulus $\partial U_i\times
\left[\frac13,\frac23\right]\times\{1\}$.

Next, each solid piece \(B_i^-\) deformation retracts onto its
\(S^1\)-factor. Thus its spin structures are parametrized by a unique
element $e_i\in\Z_2$,
which records the restriction to the longitudinal circle \(\ell_i\).
On the boundary torus, the meridian \(m_i\) bounds the disk \(D_i\).
Its induced spin structure is therefore bounding. Hence the restriction
of the spin structure on \(B_i^-\) is represented by
$(1,e_i)$
with respect to the ordered circle directions \((m_i,\ell_i)\).

Similarly, the spin structures on \(B_i^+\) are parametrized by
\(e'_i\in\Z_2\), and their boundary restrictions are represented by $
(1,e'_i)$
with respect to \((m'_i,\ell'_i)\).

We now impose the gluing conditions. The first gluing map is
\[
\phi_i(\xi,t,\eta)
=
\left(
\xi^{-a_i}\eta^{-\nu_i},
t,
\xi^{-b_i}\eta^{\rho_i}
\right).
\]
Thus, on first homology,
\[
(\phi_i)_*[m_i]
=
-a_i[\alpha_i]-b_i[\delta],
\qquad
(\phi_i)_*[\ell_i]
=
-\nu_i[\alpha_i]+\rho_i[\delta].
\]
With respect to the ordered bases $(m_i,\ell_i)$ and $
(\alpha_i,\delta)$
the induced homomorphism is represented by $
A_i^-=
\begin{pmatrix}
-a_i&-\nu_i\\
-b_i&\rho_i
\end{pmatrix}$.
It follows from the torus calculation above that the pullback by
\(\phi_i\) of the spin structure represented by
\((\lambda_i,\mu)\) is represented by $(\lambda_i,\mu)A_i^-$.
Compatibility with the spin structure on \(B_i^-\) is therefore
equivalent to
\[
(\lambda_i,\mu)
\begin{pmatrix}
-a_i&-\nu_i\\
-b_i&\rho_i
\end{pmatrix}
\equiv
(1,e_i)
\pmod2.
\]
Compatibility of the second map is proven in a similar way. This completes the proof.

\end{proof}

We next describe the spin structures on the boundary Seifert manifold
in the type A case.  Write
\[
Y=
M(0;(a''_1,b''_1),(a_2,b_2),(a'_2,b'_2),\ldots,
(a_{n-1},b_{n-1}),(a'_{n-1},b'_{n-1}),(a''_n,b''_n)).
\]
Let \(S_0\) denote the complement in \(S^2\) of disjoint disks around
the corresponding Seifert fibers.  Thus
\[
Y=
S_0\times S^1
\cup
\left(
\coprod_{i=2}^{n-1}
\bigl(-D_i\times S^1\bigr)
\sqcup
\coprod_{i=2}^{n-1}
\bigl(D'_i\times S^1\bigr)
\sqcup
D''_1\times S^1
\sqcup
D''_n\times S^1
\right).
\]

For a spin structure \(\mathfrak{s}\) on \(Y\), let $
(\lambda_i,\mu),\ 
(\lambda'_i,\mu),\ 
(\lambda''_j,\mu)$
denote its restrictions to the boundary tori in \(S_0\times S^1\),
with respect to the ordered directions given by the boundary circle of
\(S_0\) and the regular \(S^1\)-fiber.  Here $
2\leq i\leq n-1,\ 
j=1,n$.
On the corresponding Seifert solid tori, write the boundary restrictions
as $
(1,e_i),\ 
(1,e'_i),\ 
(1,e''_j),$
with respect to the meridional and longitudinal directions.

\begin{prop}\label{prop:spin structures on type-A Y}
The set of isomorphism classes of spin structures on \(Y\) is naturally
identified with the set of tuples
\[
\mu\in\Z_2,\ 
\lambda_i,\lambda'_i,e_i,e'_i\in\Z_2
(2\leq i\leq n-1),
\text{ and }
\lambda''_j,e''_j\in\Z_2
\qquad
(j=1,n),
\]
satisfying $
\lambda''_1+\lambda''_n
+
\sum_{i=2}^{n-1}
\bigl(\lambda_i+\lambda'_i\bigr)
\equiv 0
\pmod 2$,
together with
\[
(\lambda_i,\mu)
\begin{pmatrix}
-a_i&-\nu_i\\
-b_i&\rho_i
\end{pmatrix}
\equiv
(1,e_i) \text{ and }
(\lambda'_i,\mu)
\begin{pmatrix}
a'_i&\nu'_i\\
-b'_i&\rho'_i
\end{pmatrix}
\equiv
(1,e'_i)
\pmod 2
\qquad
(2\leq i\leq n-1),
\]
and
\[
(\lambda''_j,\mu)
\begin{pmatrix}
a''_j&\nu''_j\\
-b''_j&\rho''_j
\end{pmatrix}
\equiv
(1,e''_j)
\pmod 2
\qquad
(j=1,n).
\]

Moreover, under the restriction map $
\operatorname{Spin}(X_0)\longrightarrow\operatorname{Spin}(Y)$,
the image consists precisely of the tuples satisfying
\[
\lambda_i=\lambda'_i
\ 
(2\leq i\leq n-1), \ 
\lambda''_1=\lambda''_n=1.
\]
\end{prop}

\begin{proof}
The set of spin structures on \(S_0\times S^1\) is an affine space over
\[
H^1(S_0\times S^1;\Z_2)
\cong
H^1(S_0;\Z_2)\oplus\Z_2\langle\delta^*\rangle.
\]
The second summand is recorded by \(\mu\). The surface \(S_0\) is a sphere with \(2n-2\) boundary components.
For any spin structure on \(S_0\), the disjoint union of the induced
spin structures on its boundary is spin null-bordant.  Hence the number
of boundary components carrying the nonbounding spin structure is even.
Since \(S_0\) has an even number of boundary components, this condition
is equivalent, under our convention, to
\[
\lambda''_1+\lambda''_n
+
\sum_{i=2}^{n-1}
\bigl(\lambda_i+\lambda'_i\bigr)
\equiv0
\pmod2.
\]
Each Seifert solid torus has two spin structures.  Its meridian bounds a
disk and therefore carries the bounding spin structure.  Thus its
restriction to the boundary torus is of the form
\[
(1,e_i),\qquad
(1,e'_i),\qquad
(1,e''_j).
\]
By the calculation used in
\cref{prop:spin structures on X_0}, the gluing maps act on the boundary
spin structures by the corresponding matrices.  Therefore the matching
conditions are exactly
\[
(\lambda_i,\mu)
\begin{pmatrix}
-a_i&-\nu_i\\
-b_i&\rho_i
\end{pmatrix}
\equiv
(1,e_i),\ 
(\lambda'_i,\mu)
\begin{pmatrix}
a'_i&\nu'_i\\
-b'_i&\rho'_i
\end{pmatrix}
\equiv
(1,e'_i),
\]
and
\[
(\lambda''_j,\mu)
\begin{pmatrix}
a''_j&\nu''_j\\
-b''_j&\rho''_j
\end{pmatrix}
\equiv
(1,e''_j)
\pmod2.
\]

Finally, a spin structure extending over \(X_0\) has the same parameter
at the two ends of every full tunnel, and hence $\lambda_i=\lambda'_i$.
At a half tunnel, the corresponding boundary circle bounds in the
product piece of \(X_0\), so its induced spin structure is bounding;
thus 
$
\lambda''_j=1$.
Conversely, these conditions allow the spin structure on
\(S_0\times S^1\) to extend over the product piece and all full and half
tunnels of \(X_0\).  This proves the final assertion.
\end{proof}

\section{An involution on the oribfold $X$}
\label{section:Involution}
In this section, we again fix a closed orbifold surface
$$
\check{\Sigma}
=
\check{\Sigma}(g;a_1,\ldots,a_n,a'_1,\ldots,a'_n,a''_1,a''_2)
$$
together with a Seifert fibration
$$
Y
=
M(g;(a_1,b_1),\ldots,(a_n,b_n),(a'_1,b'_1),\ldots,(a'_n,b'_n),(a''_1,b''_1),(a''_2,b''_2))
\to
\check{\Sigma}.
$$
We shall classify involutions $\tau:Y\to Y$ satisfying the assumptions of \cref{classify_under}. Then $\tau$ induces an involution
$
\overline{\tau}:\check{\Sigma}\to\check{\Sigma}.
$
By \cref{classify_over}, after reordering the orbifold points, $\overline{\tau}$ is conjugate to one of the following three cases: 
\begin{itemize}[leftmargin=2pt,itemindent=4em,labelsep=0.5em]
    \item[{\bf \underline{Type A}}]
    We have $g=0$, and $\overline{\tau}:S^2\to S^2$ is the
    $180$--degree rotation about the $z$--axis. The points of orders
    $a''_1$ and $a''_2$ are fixed, while the point of order $a_i$ is
    interchanged with the point of order $a'_i$. In particular,
    \[
    a_i=a'_i
    \qquad
    (i=1,\ldots,n).
    \]

    \item[{\bf \underline{Type B}}]
    We have $g=1$, and $\overline{\tau}:T^2\to T^2$ is a non-trivial
    translation of order two, for example
    \[
    (x,y)\longmapsto
    \left(x+\frac12,y\right).
    \]
    This involution has no fixed points. Hence all orbifold points are
    paired: the point of order $a_i$ is interchanged with the point of
    order $a'_i$, so $
    a_i=a'_i$.
    Under the convention used above, the remaining two points are
    ordinary points, and hence $
    a''_1=a''_2=1.$

    \item[{\bf \underline{Type C}}]
    We have $
    \overline{\tau}=\id_{\check{\Sigma}}.$ 
    In this case, there is no restriction on the orbifold orders.
\end{itemize}

We now describe the involution induced on the boundary Seifert manifold. This is the restriction to
$
\partial_-X_0=Y
$
of the involution on $X_0$ constructed above.

Recall that the boundary component $Y$ is obtained from the following boundary faces of $X_0$. The product part is
$$
S_0\times S^1
=
(\partial W_0\cap \partial H_g)\times S^1
\subset \partial(W_0\times S^1).
$$
For each full tunnel $U_i\times I$, the two Seifert solid tori in $Y$ are the boundary faces
$$
-D_i\times\{0\}\times S^1
\subset
-(D_i\times I_1\times S^1),
\qquad
D'_i\times\{1\}\times S^1
\subset
D'_i\times I_2\times S^1.
$$
We identify them with $-D_i\times S^1$ and $D'_i\times S^1$, respectively. Similarly, for each half tunnel $U''_j\times I''$, the corresponding Seifert solid torus in $Y$ is
$$
D''_j\times\{1\}\times S^1
\subset
D''_j\times I_2\times S^1,
$$
which we identify with $D''_j\times S^1$.

Now, we construct involutions of type A, for a fixed $\lambda\in\{\pm1\}$. 
\begin{defn}
We define an involution
$
\tau_\lambda:Y\to Y
$
as follows.

\begin{itemize}[leftmargin=0pt,itemindent=1.5em,labelsep=0.5em]
    \item 
On the product part $S_0\times S^1$, its restriction is given near the boundary components by
$$
(u,v)\mapsto (u^{-1},\lambda v)
$$
on the paired components $\partial U_i\times S^1$ and $\partial U'_i\times S^1$, and by
$$
(u'',v'')\mapsto (-u'',\lambda v'')
$$
on each fixed component $\partial U''_j\times S^1$.

\item On the Seifert solid tori coming from the full tunnels, the restriction is
$$
\tau_{\lambda i}:
-D_i\times S^1\to D'_i\times S^1,
\qquad
(\xi,\eta)\mapsto(\lambda^{-\nu_i}\xi,\lambda^{a_i}\eta),
$$
and
$$
\tau'_{\lambda i}:
D'_i\times S^1\to -D_i\times S^1,
\qquad
(\xi',\eta')\mapsto(\lambda^{-\nu_i}\xi',\lambda^{a_i}\eta').
$$
\item 
On the Seifert solid torus coming from the half tunnel, the restriction is
$$
\tau''_{\lambda j}:
D''_j\times S^1\to D''_j\times S^1,
\qquad
(\xi'',\eta'')
\mapsto
(\lambda^{-\nu''_j}\zeta_2^{\rho''_j}\xi'',
\lambda^{a''_j}\zeta_2^{b''_j}\eta'').
$$
Equivalently, in the orbifold solid-torus charts, these maps are
$
[z,w]\mapsto [z,\lambda w]
$
on the paired components, and
$
[z,w]\mapsto [\zeta_{2a''_j}z,\lambda\zeta_{2a''_j}^{b''_j}w]
$
on the fixed component of type $(a''_j,b''_j)$.
\end{itemize}

\end{defn}
One can check the compatibilities of these involutions. 
Hence the above maps glue together to define an involution
$
\tau_\lambda:Y\to Y
$
commuting with the Seifert $S^1$--action.

This actually classifies certain involutions in the following sense:

\begin{prop}\label{prop:classification-tau-lambda}

Let $\tau:Y\to Y$ be an orientation-preserving involution which commutes with the Seifert $S^1$--action of type A. 
Then, after conjugating $\tau$ by a fiber-preserving diffeomorphism of $Y$ preserving the Seifert data, $\tau$ is of the form
$
\tau_\lambda
$
for some $\lambda\in\{\pm1\}$ as in the above definition.
\end{prop}

\begin{proof}
This follows from the classification result \cref{classify_over} combined with the fact that $\tau_1$ and $\tau_{-1}$ give different involutions.
\end{proof}

Now, we state our main result on the extension of involutions.

\begin{prop}\label{Prop:extension}
Let $(Y,\mathfrak{s},\tau)$ be an odd spin Seifert $3$--manifold with
sphere base, and suppose that $\tau$ is a type A involution commuting
with the Seifert $S^1$--action. Fix an odd lift $
\widehat{\tau}:P_{\mathfrak{s}}\longrightarrow P_{\mathfrak{s}}$
of $\tau$.

Then the gluing data in the construction of $X_0$, including the
meridional slopes at the unused ends of the half tunnels, can be chosen
so that there exist $
\widetilde{\mathfrak{s}}\in\operatorname{Spin}(X_0)$,
 $\widetilde{\tau}:X_0\longrightarrow X_0$, 
and an odd lift 
$\widehat{{\tau}_0}:
P_{\widetilde{\mathfrak{s}}}
\longrightarrow
P_{\widetilde{\mathfrak{s}}}$ 
such that $
\partial_-
(X_0,\widetilde{\mathfrak{s}},
\widetilde{\tau},
\widehat{{\tau}_0})
=
(Y,\mathfrak{s},\tau,\widehat{\tau}).$
\end{prop}

\begin{proof} By \cref{prop:classification-tau-lambda}, after conjugating by a fiber-preserving diffeomorphism and transporting the spin structure and its fixed lift, we may assume that \(\tau=\tau_\lambda\) for some \(\lambda\in\{\pm1\}\). Put \(\varepsilon_\lambda:=(1-\lambda)/2\in\{0,1\}\) and \(\kappa:=\varepsilon_\lambda\mu\in\{0,1\}\), where \(\mu\in\mathbb Z_2\) is the spin parameter in the regular-fiber direction. 

Choose a symmetric Seifert presentation \[Y=M\bigl(0;(a''_1,b''_1),(a_2,b_2),(a'_2,b'_2),\ldots,(a_{n-1},b_{n-1}),(a'_{n-1},b'_{n-1}),(a''_n,b''_n)\bigr),\] where \((a_i,b_i)=(a'_i,b'_i)\) for \(2\leq i\leq n-1\), and denote the boundary parameters of \(\mathfrak{s}\) by \((\lambda_i,\mu)\), \((\lambda'_i,\mu)\), and \((\lambda''_j,\mu)\), respectively. Since \(\widehat{\tau}\) is a lift of \(\tau\), the spin structure \(\mathfrak{s}\) is preserved by \(\tau\). On the two boundary tori corresponding to the \(i\)-th full tunnel, the involution sends \(\alpha_i\) to \(-\alpha'_i\) and preserves \(\delta\), and therefore \[\lambda_i=\lambda'_i\qquad(2\leq i\leq n-1).\tag{1}\] On a fixed boundary torus \(T''_j=\partial U''_j\times S^1\), write \(u=e^{i\theta}\) and \(v=e^{i\phi}\). 
Then \[\tau_\lambda(u,v)=(-u,\lambda v)=\bigl(e^{i(\theta+\pi)},e^{i(\phi+\varepsilon_\lambda\pi)}\bigr).\] The local lift calculation gives \(\widehat{\tau}^{\,2}=(-1)^{\lambda''_j+\varepsilon_\lambda\mu}\). Since \(\widehat{\tau}\) is odd, \[\lambda''_j+\varepsilon_\lambda\mu\equiv1\pmod2\qquad(j=1,n).\tag{2}\] 

Set \(t_1=\kappa\) and \(t_n=-\kappa\), and replace the two fixed Seifert invariants by 
$(a''_j,b''_j)\longmapsto(a''_j,\bar b''_j),\  \bar b''_j:=b''_j+t_ja''_j\ (j=1,n)$.
Since \(t_1+t_n=0\), this is a standard equivalence of Seifert invariants and hence gives a fiber-preservingly diffeomorphic presentation of the same Seifert manifold \(Y\). We henceforth use this equivalent presentation. Under the corresponding identification of the boundary tori, we have  $\bar\alpha''_j=\alpha''_j+t_j\delta,$ as is seen from the equality $a''_j\alpha''_j-b''_j\delta=a''_j\bar\alpha''_j-\bar b''_j\delta$. In the new boundary basis, the spin parameter is
$\bar\lambda''_j=\lambda''_j+t_j\mu$. 

Since \(t_j\equiv\kappa\pmod2\) for \(j=1,n\), (2) gives \[\bar\lambda''_j\equiv\lambda''_j+\kappa\mu=\lambda''_j+\varepsilon_\lambda\mu\equiv1\pmod2.\tag{3}\]

The underlying fiber-preserving involution is still of type A and
commutes with the Seifert \(S^1\)--action. Hence, after a further
fiber-preserving conjugacy and transporting the spin structure and its
lift, \cref{prop:classification-tau-lambda} allows us to write it as
\(\tau_{\bar\lambda}\) for some
\(\bar\lambda\in\{\pm1\}\). Renaming \(\bar\lambda\) as \(\lambda\),
we use the notation \(\tau_\lambda\) from now on.

We now choose the gluing data for \(X_0\). For each full tunnel, choose \(\nu_i,\rho_i\in\mathbb Z\) satisfying \(a_i\rho_i+b_i\nu_i=1\), and set \[\nu'_i=\nu_i,\qquad \rho'_i=\rho_i.\tag{4}\] For \(j=1,n\), first choose \(\nu''_j,\rho''_j\in\mathbb Z\) satisfying \(a''_j\rho''_j+b''_j\nu''_j=1\), and put \[\bar\nu''_j=\nu''_j,\qquad \bar\rho''_j=\rho''_j-t_j\nu''_j.\tag{5}\] Then \[a''_j\bar\rho''_j+\bar b''_j\bar\nu''_j=a''_j(\rho''_j-t_j\nu''_j)+(b''_j+t_ja''_j)\nu''_j=1,\] so (5) gives valid gluing data for the fixed half tunnels in the new
Seifert presentation. At the unused end of each half tunnel, we use
the standard missing-end gluing data
\[
(a,b,\nu,\rho)=(1,0,0,1).
\]
By \cref{prop:spin structures on type-A Y}, the image of \(\operatorname{Spin}(X_0)\to\operatorname{Spin}(Y)\) is characterized by \(\lambda_i=\lambda'_i\) at every full tunnel and \(\bar\lambda''_j=1\) at every fixed half tunnel. These conditions follow from (1) and (3), respectively.
Hence \(\mathfrak{s}\) extends to a spin structure \(\widetilde{\mathfrak{s}}\in\operatorname{Spin}(X_0)\). The symmetric choices in (4) allow \(\tau_{\bar\lambda}\) to extend across every full tunnel, while its standard local form on each fixed component allows it to extend across the corresponding half tunnel with the gluing data (5). We therefore obtain an involution \(\widetilde{\tau}:X_0\to X_0\) whose restriction to \(Y\) is \(\tau\).

For each full tunnel, the two gluing matrices appearing in
\cref{prop:spin structures on X_0} are identical modulo \(2\), because
the gluing data are chosen symmetrically. Hence
\(\lambda_i=\lambda'_i\) implies \(e_i=e'_i\), and the involution
identifies the spin structures on the two exchanged pieces. On each
fixed half tunnel, the involution is a product of rotations of the disk
and circle factors and therefore preserves each spin structure on the
solid piece. Consequently, the piecewise extension may be chosen so
that $
\widetilde{\tau}^{*}\widetilde{\mathfrak{s}}
\cong
\widetilde{\mathfrak{s}}$. Thus \(\widetilde{\tau}\) admits a lift to
\(P_{\widetilde{\mathfrak{s}}}\).

Since \(X_0\) is connected, its two lifts differ by the central element \(-1\), and we may therefore choose a lift \(\widehat{\tau}_0:P_{\widetilde{\mathfrak{s}}}\to P_{\widetilde{\mathfrak{s}}}\) whose restriction to \(P_{\mathfrak{s}}\) is the prescribed lift \(\widehat{\tau}\). The square \(\widehat{\tau}_0^{\,2}\) is a continuous function from the connected manifold \(X_0\) to \(\{\pm1\}\), and is therefore constant. Its restriction to \(Y\) is \(\widehat{\tau}^{\,2}=-1\), so \(\widehat{\tau}_0^{\,2}=-1\) on \(X_0\). Thus \(\widehat{\tau}_0\) is odd, and, after undoing the initial fiber-preserving conjugacy, we obtain $\partial_-\bigl(X_0,\widetilde{\mathfrak{s}},\widetilde{\tau},\widehat{\tau}_0\bigr)=\bigl(Y,\mathfrak{s},\tau,\widehat{\tau}\bigr).$
\end{proof}

\begin{prop}\label{prop:involution}
Let
$
\widetilde{\tau}_\lambda:X_0\to X_0
$
be the extension of the boundary involution
$
\tau_\lambda:Y\to Y
$
constructed in \cref{Prop:extension}. Then
$$
H_1(X_0;\Q)^{-\widetilde{\tau}_{\lambda *}}=0
\qquad\text{and}\qquad
H_2(X_0;\Q)^{-\widetilde{\tau}_{\lambda *}}=0.
$$
\end{prop}

\begin{proof}
We first consider $H_1$. In the Mayer--Vietoris description of $H_1(X_0;\Q)$,
let $\delta$ denote the $S^1$--fiber direction, let $\gamma_i,\gamma_i'$ be the
boundary-circle generators corresponding to the two ends of the $i$-th full
tunnel, and let $\gamma_j''$ be the generator corresponding to the $j$-th half
tunnel. Since we suppose type A, the full-tunnel data satisfy
\[
(a_i,b_i)=(a'_i,b'_i)
\qquad
(2\leq i\leq n-1).
\]
The two half tunnels are indexed by \(j=1,n\), and their data in the
new Seifert presentation are
$
(a''_j,\bar b''_j)
\ 
(j=1,n)$.
The relations coming from the solid pieces give, over $\Q$,
$$
\gamma_i=-\frac{b_i}{a_i}\delta,
\qquad
\gamma_i'=\frac{b_i}{a_i}\delta,
\qquad
\gamma_j''\in \Q\langle \delta\rangle.
$$
Thus $H_1(X_0;\Q)$ is generated by the class of $\delta$.

The involution $\widetilde{\tau}_\lambda$ acts on the full tunnels by
interchanging the two ends and reversing the boundary-circle coordinate:
$$
\widetilde{\tau}_{\lambda *}(\gamma_i)=-\gamma_i',
\qquad
\widetilde{\tau}_{\lambda *}(\gamma_i')=-\gamma_i.
$$
It preserves the half-tunnel generators:
$
\widetilde{\tau}_{\lambda *}(\gamma_j'')=\gamma_j''.
$
Moreover, multiplication by $\lambda$ on the $S^1$--factor is isotopic to the
identity, so
$
\widetilde{\tau}_{\lambda *}(\delta)=\delta.
$
Using the relations above, we have
$$
-\gamma_i'=-\frac{b_i}{a_i}\delta=\gamma_i,
\qquad
-\gamma_i=\frac{b_i}{a_i}\delta=\gamma_i'.
$$
Hence every generator of $H_1(X_0;\Q)$ is fixed by
$\widetilde{\tau}_{\lambda *}$. Therefore
$
H_1(X_0;\Q)^{-\widetilde{\tau}_{\lambda *}}=0.
$

We next consider \(H_2\). By the homology computation of \(X_0\),
with the half-tunnel convention of \cref{general_homology}, we have
\[
H_2(X_0;\mathbb Z)
\cong
\left\{
\sum_{i=2}^{n-1}s_i(m_i+m'_i)
+
\sum_{j\in\{1,n\}}s''_jm''_j
\;\middle|\;
\sum_{i=2}^{n-1}2b_is_i
+
\sum_{j\in\{1,n\}}\bar b''_js''_j
=0
\right\}.
\]
Here \(m_i,m'_i,m''_j\) denote the meridian generators appearing in
the Mayer--Vietoris description of \(H_2(X_0;\mathbb Z)\).

On each full tunnel, $\widetilde{\tau}_\lambda$ interchanges the two pieces
$$
-D_i\times I_1\times S^1
\quad\text{and}\quad
D'_i\times I_2\times S^1.
$$
With the above orientation conventions, this gives
$
\widetilde{\tau}_{\lambda *}(m_i)=m'_i,
\widetilde{\tau}_{\lambda *}(m'_i)=m_i.
$
On each half tunnel, $\widetilde{\tau}_\lambda$ preserves the corresponding
piece. The map on the meridian is a rotation, and the multiplication by
$\lambda$ in the $S^1$--factor is isotopic to the identity. Hence
$
\widetilde{\tau}_{\lambda *}(m''_j)=m''_j.
$
Therefore
\[
\widetilde{\tau}_{\lambda *}
\left(
\sum_{i=2}^{n-1}s_i(m_i+m'_i)
+
\sum_{j\in\{1,n\}}s''_jm''_j
\right)
=
\sum_{i=2}^{n-1}s_i(m_i+m'_i)
+
\sum_{j\in\{1,n\}}s''_jm''_j.
\]
Thus every class in $H_2(X_0;\Z)$ is fixed by
$\widetilde{\tau}_{\lambda *}$. 
This proves the proposition.
\end{proof}

\begin{lem}\label{lem:extension over orbifold caps}
The involution $
\widetilde{\tau}:X_0\longrightarrow X_0$
and its odd lift
$
\widehat{\tau}_0:
P_{\widetilde{\mathfrak{s}}}
\longrightarrow
P_{\widetilde{\mathfrak{s}}}$ 
constructed in \cref{Prop:extension} extend over the orbifold caps and
hence define an odd spin involution on \(\check X\).
\end{lem}

\begin{proof}
We prove the extension over each orbifold cap attached to $X_0$.
There are two types of caps: those coming from full tunnels and those
coming from half tunnels.

Put $ 
\varepsilon_\lambda:=\frac{1-\lambda}{2}\in\{0,1\}$.
Throughout the proof, for $r\in\mathbb Q$, we use the convention $
\lambda^r:=\exp(\pi i\varepsilon_\lambda r),
\zeta_m^r:=\exp\left(\frac{2\pi i r}{m}\right)$.

We first recall a standard fact about spin structures on cyclic
quotient caps. Consider
\[
C=
\frac{\widetilde D\times\widetilde D'}
{\mathbb Z_{\widetilde a}},
\qquad
\gamma(z,w)
=
(\zeta_{\widetilde a}z,
 \zeta_{\widetilde a}^{\widetilde b}w).
\]
Put $p=|\widetilde a|$. A lift of $\gamma$ to
$\Spin(4)=\Sp(1)\times\Sp(1)$ is, up to a central sign,
\[
\widehat\gamma_C
=
\varepsilon_C
\left(
\zeta_{\widetilde a}^{(1-\widetilde b)/2},
\zeta_{\widetilde a}^{-(1+\widetilde b)/2}
\right),
\qquad
\varepsilon_C\in\{\pm1\}.
\]
Its $p$-th power is $
\widehat\gamma_C^{\,p}
=
\varepsilon_C^p(-1)^{1+\widetilde b}$.
Since $\gcd(\widetilde a,\widetilde b)=1$, if $p$ is even, then
$\widetilde b$ is odd. It follows that one can choose
$\varepsilon_C$ so that $\widehat\gamma_C^{\,p}=1$. If $p$ is odd,
the boundary lens space has a unique spin structure, while if $p$ is
even, the two choices of $\varepsilon_C$ induce its two spin
structures. Thus the spin structure already fixed on the boundary
extends to an orbifold spin structure on the cap. In what follows,
$\varepsilon_C$ is always chosen so that this orbifold spin structure
restricts to the prescribed boundary spin structure.

\medskip
\noindent
{\bf Full-tunnel case.}
For this full tunnel, write $(a,b,\nu,\rho)$ for the gluing data at
one end. By the symmetric choice in \cref{Prop:extension}, the gluing
data at the other end are also $(a,b,\nu,\rho)$, where $
a\rho+b\nu=1$.
Consider a cap of the form $
C=
\frac{\widetilde D\times\widetilde D'}
{\mathbb Z_{\widetilde a}}$,
where the generator $\gamma$ of the cyclic action is
$
\gamma(z,w)
=
(\zeta_{\widetilde a}z,
 \zeta_{\widetilde a}^{\widetilde b}w)$.
The boundary charts are
\[
\varphi:
\frac{\partial\widetilde D\times\widetilde D'}
{\mathbb Z_{\widetilde a}}
\longrightarrow S^1\times D',
\qquad
\varphi[z,w]
=
(z^{\widetilde a},z^{-\widetilde b}w),
\]
and
\[
\psi:
\frac{\widetilde D\times\partial\widetilde D'}
{\mathbb Z_{\widetilde a}}
\longrightarrow D\times S^1,
\qquad
\psi[z,w]
=
(w^{-\widetilde\nu}z,w^{\widetilde a}).
\]
Let $
\iota(\xi,\eta)=(\eta,\xi)$
and put $
\varphi'=\iota\circ\varphi$.

On the covering space $\widetilde D\times\widetilde D'$, define
\[
\tau_C(z,w)
=
(s_\lambda w,s_\lambda z),
\qquad
s_\lambda:=\lambda^{a/\widetilde a}.
\]
The symmetric gluing data give
$
\widetilde b=a\rho-b\nu,
\ 
\widetilde\nu=\widetilde b,
\ 
\widetilde a=-2ab$.
We first have
\[
\begin{aligned}
\widetilde b^2-1
&=
(a\rho-b\nu)^2-(a\rho+b\nu)^2=
-4ab\rho\nu
=
2\rho\nu\,\widetilde a.
\end{aligned}
\]
Hence $
\widetilde b^2\equiv1\pmod{\widetilde a}$,
and therefore $\tau_C\circ\gamma
=
\gamma^{\widetilde b}\circ\tau_C$.
Thus $\tau_C$ descends to the quotient.

We next check that the induced map is an involution. Since $
s_\lambda^2
=
\zeta_{\widetilde a}^{\varepsilon_\lambda a}$
and $
a(\widetilde b-1)
=
-2ab\nu
=
\widetilde a\nu$,
we have $
\zeta_{\widetilde a}^{
\varepsilon_\lambda a\widetilde b}
=
\zeta_{\widetilde a}^{\varepsilon_\lambda a}$.
Consequently,
\[
\begin{aligned}
\tau_C^2(z,w)
&=
\left(
\zeta_{\widetilde a}^{\varepsilon_\lambda a}z,
\zeta_{\widetilde a}^{\varepsilon_\lambda a}w
\right)=
\gamma^{\varepsilon_\lambda a}(z,w).
\end{aligned}
\]
Thus the induced map on $C$ has square equal to the identity.

Define $ 
\tau_{\lambda0}(\xi,\eta)
=
(\lambda^{-\nu}\xi,\lambda^a\eta)$.
Since $
a(1-\widetilde b)
=
-\widetilde a\nu$,
we have $
s_\lambda^{1-\widetilde b}
=
\lambda^{-\nu}, 
s_\lambda^{\widetilde a}
=
\lambda^a$.
Therefore
\[
\begin{aligned}
\varphi'\circ\tau_C[z,w]
&=
\left(
s_\lambda^{1-\widetilde b}
w^{-\widetilde b}z,
s_\lambda^{\widetilde a}w^{\widetilde a}
\right)=
\left(
\lambda^{-\nu}w^{-\widetilde b}z,
\lambda^aw^{\widetilde a}
\right)=
\tau_{\lambda0}\circ\psi[z,w].
\end{aligned}
\]
Hence the underlying involution extends over the full-tunnel cap.

We next lift this map to the spin bundle. The coordinate-switching
map $\iota$ has a spin lift
\[
\widetilde\iota_*
=
\widehat\varepsilon_0(-k,i)
\in\Spin(4),
\qquad
\widehat\varepsilon_0\in\{\pm1\}.
\]
Since $\tau_C$ is obtained from $\iota$ by the scalar rotation
$s_\lambda$, a spin lift of $\tau_C$ is
\[
\widetilde\tau_{C*}
=
\widehat\varepsilon
(-k,s_\lambda^{-1}i),
\qquad
\widehat\varepsilon\in\{\pm1\}.
\]

The lift of the generator of the cyclic action defining the chosen
orbifold spin structure is
\[
\widehat\gamma_C
=
\varepsilon_C
\left(
\zeta_{\widetilde a}^{(1-\widetilde b)/2},
\zeta_{\widetilde a}^{-(1+\widetilde b)/2}
\right).
\]
A direct computation, using
$\widetilde b^2\equiv1\pmod{\widetilde a}$, gives
$
\widetilde\tau_{C*}\circ\widehat\gamma_C
=
\widehat\gamma_C^{\,\widetilde b}
\circ\widetilde\tau_{C*}$
from $ 
\varepsilon_C=\varepsilon_C^{\widetilde b}$.
But
\[
\widetilde b
=
a\rho-b\nu
=
a\rho+b\nu-2b\nu
\equiv1\pmod2.
\]
Thus this condition is automatic, and
$\widetilde\tau_{C*}$ descends to a spin lift over $C$.

The restriction of this lift and the spin lift already fixed on
$X_0$ are lifts of the same boundary involution with respect to the
same spin structure. Since the boundary lens space is connected, they
differ by a constant element of $\{\pm1\}$. Choosing
$\widehat\varepsilon$ appropriately, the two lifts agree. Hence the
spin lift extends over the full-tunnel cap.

\medskip
\noindent
{\bf Half-tunnel case.}
Fix a half tunnel indexed by \(j\in\{1,n\}\). Write
$
(a',b',\nu',\rho')
=
(a''_j,\bar b''_j,\bar\nu''_j,\bar\rho''_j)$
for the gluing data at the end belonging to \(Y\), and use $
(a,b,\nu,\rho)=(1,0,0,1)$
at the missing end. Thus
$
a\rho+b\nu=1,
\ 
a'\rho'+b'\nu'=1$.

The boundary half-tunnel model is
$
\tau_\lambda(u,t,v)
=
(-u,t,\lambda v)$,
and on the two orbifold solid-torus charts it is written as
\[
[z,0,w]
\longmapsto
[\zeta_{2a}^{-1}z,0,
 \lambda\zeta_{2a}^{-b}w]
\]
and
\[
[z',1,w']
\longmapsto
[\zeta_{2a'}z',1,
 \lambda\zeta_{2a'}^{b'}w'].
\]

After identifying the boundary with the boundary of the orbifold cap $
C=
\frac{\widetilde D\times\widetilde D'}
{\mathbb Z_{\widetilde a}}$,
the generator of the cyclic action is $
\gamma(z,w)
=
(\zeta_{\widetilde a}z,
 \zeta_{\widetilde a}^{\widetilde b}w)$,
and the cap gluing is determined by
\[
\begin{pmatrix}
-\widetilde b & \widetilde\rho\\
\widetilde a & \widetilde\nu
\end{pmatrix}
=
\begin{pmatrix}
a' & \nu'\\
-b' & \rho'
\end{pmatrix}^{-1}
\begin{pmatrix}
-a & -\nu\\
-b & \rho
\end{pmatrix}.
\]

Put $ 
c_+
=
\zeta_{2\widetilde a}^{b'}
\lambda^{a'/\widetilde a},
\ 
c_-
=
\zeta_{2\widetilde a}^{-b}
\lambda^{a/\widetilde a}.$
On the covering space, define
\[
\tau_C(z,w)
=
(c_+z,c_-w).
\]
Since $\tau_C$ is diagonal, it commutes with $\gamma$ and hence
descends to the quotient.

An elemental computation ensures $
\widetilde b
\bigl(b'+\varepsilon_\lambda a'\bigr)
\equiv
-b+\varepsilon_\lambda a
\pmod{\widetilde a}$.
Since
$
c_+^2
=
\zeta_{\widetilde a}^{
b'+\varepsilon_\lambda a'},
\ 
c_-^2
=
\zeta_{\widetilde a}^{
-b+\varepsilon_\lambda a}$,
we obtain $
\tau_C^2
=
\gamma^{b'+\varepsilon_\lambda a'}$.
Thus the induced map on $C$ is an involution.

We next check that it restricts to the prescribed boundary maps. Put
\[
\varphi[z,w]
=
(z^{\widetilde a},z^{-\widetilde b}w),
\ 
\varphi'[z,w]
=
(z^{-\widetilde b}w,z^{\widetilde a}),
 \text{ and }
\psi[z,w]
=
(w^{-\widetilde\nu}z,w^{\widetilde a}).
\]
In addition to the two identities above, the gluing matrix gives
\[
b\widetilde\nu+b'
=
-\rho\widetilde a,
\qquad
a\widetilde\nu-a'
=
\nu\widetilde a.
\]
It follows that
\[
c_+^{-\widetilde b}c_-
=
\lambda^{-\nu'}\zeta_2^{\rho'},
\ 
c_+^{\widetilde a}
=
\lambda^{a'}\zeta_2^{b'},
\text{ and }
c_-^{-\widetilde\nu}c_+
=
\lambda^{-\nu}\zeta_2^{-\rho},
\ 
c_-^{\widetilde a}
=
\lambda^a\zeta_2^{-b}.
\]
Therefore
\[
\varphi'\circ\tau_C
=
\tau_\lambda^+\circ\varphi',
\qquad
\psi\circ\tau_C
=
\tau_\lambda^-\circ\psi,
\]
where
\[
\tau_\lambda^+(\xi',\eta')
=
\left(
\lambda^{-\nu'}\zeta_2^{\rho'}\xi',
\lambda^{a'}\zeta_2^{b'}\eta'
\right) \text{ and }
\tau_\lambda^-(\xi,\eta)
=
\left(
\lambda^{-\nu}\zeta_2^{-\rho}\xi,
\lambda^a\zeta_2^{-b}\eta
\right).
\]
Hence the underlying involution extends over the half-tunnel cap.

The derivative of $\tau_C$ is diagonal:
\[
d\tau_C
=
\begin{pmatrix}
c_+ & 0\\
0 & c_-
\end{pmatrix}.
\]
A spin lift is therefore
\[
\widetilde\tau_{C*}
=
\widehat\varepsilon
\left(
\zeta_{2\widetilde a}^{(b'+b)/2}
\lambda^{(a'-a)/(2\widetilde a)},
\zeta_{2\widetilde a}^{(-b'+b)/2}
\lambda^{(-a'-a)/(2\widetilde a)}
\right),
\]
where $\widehat\varepsilon\in\{\pm1\}$.

Let
\[
\widehat\gamma_C
=
\varepsilon_C
\left(
\zeta_{\widetilde a}^{(1-\widetilde b)/2},
\zeta_{\widetilde a}^{-(1+\widetilde b)/2}
\right)
\]
be the lift of the cyclic action defining the orbifold spin structure
whose restriction is the prescribed boundary spin structure. Both
$\widetilde\tau_{C*}$ and $\widehat\gamma_C$ lie in the same maximal
torus of $\Spin(4)$. Hence
$
\widetilde\tau_{C*}\circ\widehat\gamma_C
=
\widehat\gamma_C\circ\widetilde\tau_{C*}$,
and $\widetilde\tau_{C*}$ descends to a spin lift over $C$.

The restriction of this lift and the spin lift already fixed on
$X_0$ are lifts of the same boundary involution with respect to the
same spin structure. Since the boundary lens space is connected, they
differ by a constant element of $\{\pm1\}$. Choosing
$\widehat\varepsilon$ appropriately, the two lifts agree. Hence the
spin lift extends over the half-tunnel cap.

\medskip
The above arguments apply to every cap attached to $X_0$. On each
connected cap, the square of the extended spin lift covers the
identity and is therefore multiplication by a locally constant
element of $\{\pm1\}$. Its restriction to the boundary is the square
of the prescribed odd lift on $X_0$, and hence is $-1$. Therefore the
extended spin lift is odd on every cap.

Thus the spin lift on $X_0$ extends to an odd spin lift on the whole
orbifold $\check X$. This proves the lemma.
\end{proof}

\section{Real Seiberg--Witten theory}
\subsection{Real Floer homotopy type}
Let $(Y,\mathfrak{s})$ be an oriented spin $3$--manifold equipped with an
orientation-preserving smooth involution
$
\tau:Y\to Y.
$
Assume that $\tau$ preserves the isomorphism class of $\mathfrak{s}$, and fix a
lift
$$
\widetilde{\tau}:P\to P
$$
of $d\tau:SO(TY)\to SO(TY)$ to the principal $\Spin(3)$--bundle $P$ associated
with $\mathfrak{s}$. Since any two such lifts differ by the central element
$-1\in\Spin(3)$, the order of the lift is independent of this choice. We say that
the lift is even if
$
\widetilde{\tau}^2=\id,
$
and odd if
$
\widetilde{\tau}^2=-\id.
$
In the latter case, we call $(Y,\mathfrak{s},\tau,\widetilde{\tau})$ an odd spin
involutive $3$--manifold.

Throughout this subsection, we assume that $\widetilde{\tau}$ is odd and that
$$
H^1(Y;\R)^{\tau^*}=H^1(Y;\R).
$$

Choose a $\tau$--invariant Riemannian metric $g$ on $Y$, and let $\mathbb S$ be
the spinor bundle associated with $\mathfrak{s}$. We use the spin connection as
the base connection and consider the Coulomb slice
$$
V(Y):=L^2_k(i\ker d^*)\oplus L^2_k(\mathbb S).
$$
The Chern--Simons--Dirac functional
$$
CSD:V(Y)\to\R
$$
is invariant under the usual $\Pin(2)$--action. The lift $\widetilde{\tau}$
acts on spinors, and $\tau$ acts on $1$--forms by pull-back. Define an operator
$$
I:V(Y)\to V(Y)
$$
by
$$
I(a,\psi):=(-\tau^*a,\;j\,\widetilde{\tau}(\psi)),
$$
where $j\in\Pin(2)=S^1\cup jS^1$ is the quaternionic element.

Since $\widetilde{\tau}$ is odd, we have $\widetilde{\tau}^2=-1$ on spinors.
Also $j^2=-1$, and $j$ commutes with $\widetilde{\tau}$. Hence
$
I^2=\id.
$
Thus $I$ is an involution on $V(Y)$. Moreover, $I$ preserves $CSD$, and therefore
the formal gradient vector field of $CSD$ restricts to the fixed-point subspace
$
V(Y)^I.
$ Write the restricted vector field as
$
l+c,
$
where $l$ is the restriction of the self-adjoint elliptic linear part and $c$ is
the compact nonlinear part. We decompose $V(Y)^I$ into eigenspaces of $l$. For
sufficiently large $\lambda>0$, let
$$
(V^\lambda_{-\lambda}(Y))^I\oplus (W^\lambda_{-\lambda}(Y))^I
$$
be the direct sum of the eigenspaces of $l|_{V(Y)^I}$ with eigenvalues in
$(-\lambda,\lambda]$, where the first summand denotes the $1$--form part and
the second summand denotes the spinor part.

Restrict the cut-off vector field
$
l+p^\lambda_{-\lambda}c
$
to
$
(V^\lambda_{-\lambda}(Y))^I\oplus (W^\lambda_{-\lambda}(Y))^I.
$
By the compactness theorem for the  finite-dimensional approximation under
the condition $H^1(Y;\R)^{-\tau^*}=0$, this finite-dimensional flow admits a large
ball as an isolating neighborhood. Let
$
(N^I,L^I)
$
be the corresponding Conley index.
Recall that the action of $j$ commutes with $I$. Therefore $j$ preserves the
fixed-point space $V(Y)^I$, and also preserves the finite-dimensional
approximation and the Conley index $N^I/L^I$. We put
$$
G:=\langle j\rangle\cong \Z_4,
\qquad
H:=\langle -1\rangle\subset G.
$$
Thus $N^I/L^I$ is naturally a pointed finite $G$--CW complex.

On the spinor part, the operator $j$ satisfies $j^2=-1$. Hence the real vector
space $\Gamma(\mathbb S)^I$ carries a natural complex structure. We denote this
complex $G$--representation by $\C_+$; equivalently, a generator $j\in G$ acts
on $\C_+$ by multiplication by $i$.

Let
$
n(Y,\mathfrak{s},g)
$
denote the correction term appearing in Manolescu's construction, with the
same normalization as in the real finite-dimensional approximation. Namely, if
$X$ is a compact spin $4$--manifold with $\partial X=Y$, whose spin structure
restricts to $\mathfrak{s}$, and whose metric is product near the boundary and
restricts to $g$ on $Y$, then
$$
n(Y,\mathfrak{s},g)
:=
\operatorname{ind}^{APS}_{\C}(D_X^+)+\frac{\sigma(X)}{8}.
$$
Here, we emphasize that we do not require the involution to extend to $X$. 
This quantity is independent of the choice of such $X$. 
We define the real Seiberg--Witten Floer homotopy type of
$(Y,\mathfrak{s},\tau,\widetilde{\tau})$ by
$$
SWF_R(Y,\mathfrak{s},\tau)
:=
\left[
\left(
N^I/L^I,\;
\dim_{\R}(V^0_{-\lambda}(Y))^I,\;
\dim_{\C}(W^0_{-\lambda}(Y))^I+\frac{n(Y,\mathfrak{s},g)}{2}
\right)
\right]
\in\mathfrak C_G.
$$
Here the complex dimension is taken with respect to the complex structure on
$(W^0_{-\lambda}(Y))^I$ induced by the action of $j$.
We briefly recall the category $\mathfrak C_G$. Let
$$
G=\langle j\rangle\cong\Z_4,
\qquad
H=\langle -1\rangle\subset G.
$$
Let $\widetilde{\R}$ denote the real one-dimensional representation of $G$ on
which $j$ acts by multiplication by $-1$, and let $\C_+$ denote the complex
one-dimensional representation of $G$ on which $j$ acts by multiplication by $i$.

An object of $\mathfrak C_G$ is a triple
$$
(X,m,n),
$$
where $X$ is a pointed finite $G$--CW complex, $m\in\Z$, and $n\in\Q$, satisfying:
\begin{itemize}[leftmargin=0pt,itemindent=1.5em,labelsep=0.5em]
    \item $X^H$ is $G$--homotopy equivalent to $(\widetilde{\R}^s)^+$ for some
    $s\geq0$;
    \item $G$ acts freely on $X\setminus X^H$.
\end{itemize}
The morphisms are defined by stabilization. Namely, the set of morphisms from
$(X,m,n)$ to $(X',m',n')$ is
$$
\operatorname*{colim}_{\ell\to\infty}
\left[
(\widetilde{\R}^{m+\ell}\oplus\C_+^{\,n+\ell})^+\wedge X,\;
(\widetilde{\R}^{m'+\ell}\oplus\C_+^{\,n'+\ell})^+\wedge X'
\right]_G. 
$$
When the lift $\widetilde{\tau}$ is fixed or clear
from the context, we simply write
$
SWF_R(Y,\mathfrak{s},\tau)
$
or
$
SWF_R(Y).
$
We now define local equivalence. Two objects $(X,m,n)$ and $(X',m',n')$ of
$\mathfrak C_G$ are said to be locally equivalent if, for sufficiently large
$\ell$, there exist $G$--equivariant maps
$$
f:
(\widetilde{\R}^{m+\ell}\oplus\C_+^{\,n+\ell})^+\wedge X
\to
(\widetilde{\R}^{m'+\ell}\oplus\C_+^{\,n'+\ell})^+\wedge X',
$$
and
$$
f':
(\widetilde{\R}^{m'+\ell}\oplus\C_+^{\,n'+\ell})^+\wedge X'
\to
(\widetilde{\R}^{m+\ell}\oplus\C_+^{\,n+\ell})^+\wedge X,
$$
whose restrictions to the $H$--fixed point sets are $G$--homotopy equivalences.
We denote the local equivalence class of $(X,m,n)$ by
$
[(X,m,n)]_{\operatorname{loc}}.
$

Let $\mathcal{LE}_G$ denote the set of local equivalence classes. The smash
product gives $\mathcal{LE}_G$ the structure of an abelian group:
$$
[(X,m,n)]_{\operatorname{loc}}
+
[(X',m',n')]_{\operatorname{loc}}
:=
[(X\wedge X',m+m',n+n')]_{\operatorname{loc}}.
$$
The identity element is
$
[(S^0,0,0)]_{\operatorname{loc}},
$
and the inverse is given by the Spanier--Whitehead dual, together with the
grading change $(-m,-n)$.
The invariant we shall consider in this paper is
\[
[SWF_R(Y,\mathfrak{s},\tau)]_{\mathrm{loc}}
\in
\mathcal{LE}_G.
\]
The rational number $n$ is a formal $\C_+$--grading. In particular,
when $q\in\Q$, the notation $
\C_+^q$
denotes the formal grading shift by $q\C_+$ and does not mean an actual
$q$--fold direct sum unless $q$ is a non-negative integer. Morphisms,
suspensions, and local equivalences are understood in this formally
rationally graded stable category.

Once this local equivalence class is defined, we can apply the existing
equivariant Floer-theoretic formalism to extract numerical invariants from it.
In particular, the invariants $
\delta_R(Y,\mathfrak{s},\tau), 
\underline{\delta}_R(Y,\mathfrak{s},\tau),
\overline{\delta}_R(Y,\mathfrak{s},\tau)$,
and the corresponding K-theoretic invariants, which we denote by $
\kappa_R(Y,\mathfrak{s},\tau)$
and its variants, are defined from the local equivalence class $
[SWF_R(Y,\mathfrak{s},\tau)]_{\mathrm{loc}}$. See \cite{KMT21, KMT:2023}. 

We do not recall the definitions of these numerical invariants here. We only
use the fact that they depend only on the local equivalence class of
$SWF_R(Y,\mathfrak{s},\tau)$. Hence they are invariants of the odd spin
involutive $3$--manifold $(Y,\mathfrak{s},\tau,\widetilde{\tau})$.
When the lift $\widetilde{\tau}$ is fixed or clear from the context, we suppress
it from the notation.

\begin{rem} 
Let us briefly explain why the above object is independent of the choices of
the invariant metric, the cut-off parameter, and the finite-dimensional
approximation data. The independence of the cut-off parameter is proved by the usual continuation
argument for finite-dimensional approximation. If the interval of eigenvalues is
enlarged, the Conley index is suspended by the additional eigenspaces. This
change is exactly recorded by the two grading coordinates
$$
\dim_{\R}(V^0_{-\lambda}(Y))^I
\qquad\text{and}\qquad
\dim_{\C}(W^0_{-\lambda}(Y))^I.
$$

For the metric independence, let $g_t$, $0\leq t\leq1$, be a path of
$\tau$--invariant metrics. The linearized operators on the fixed-point space
$V(Y)^I$ form a path of self-adjoint Fredholm operators. The continuation
isomorphism changes the finite-dimensional approximation by the spectral flow
of this path.

The real $1$--form part contributes to the $\widetilde{\R}$--suspension. The
spinor part is the important point. Since $j$ preserves the $I$--fixed spinor
space and gives it a complex structure, the spinorial contribution to the real
finite-dimensional approximation is measured in complex $\C_+$--dimension. The
spectral flow on the $I$--fixed spinor part is one half of the usual complex
spectral flow of the Dirac operators:
$$
\operatorname{sf}_{\C_+}(D_{g_t}|_{\Gamma(\mathbb S)^I})
=
\frac{1}{2}\operatorname{sf}_{\C}(D_{g_t}).
$$
On the other hand, the APS index formula gives
$$
n(Y,\mathfrak{s},g_1)-n(Y,\mathfrak{s},g_0)
=
-\,\operatorname{sf}_{\C}(D_{g_t})
$$
with the present convention.  
Thus the variation of the spinor suspension is exactly cancelled by the
correction term $\frac{1}{2}n(Y,\mathfrak{s},g)\C_+$. This is the reason why the
real Floer homotopy type is normalized by $n(Y,\mathfrak{s},g)/2$ rather than by
$n(Y,\mathfrak{s},g)$. Again, the compact spin $4$--manifold used in the definition of
$n(Y,\mathfrak{s},g)$ is an ordinary spin bounding of $(Y,\mathfrak{s})$; no
extension of the involution $\tau$ over this bounding is needed.
\end{rem}

We now explain two standard sources of examples to which the above construction
applies.

\begin{defn}\label{def:kth-real-SWF-knot}
Let $K\subset S^3$ be a knot and let
$
\Sigma_{2^k}(K)\to S^3
$
be the $2^k$--fold cyclic branched cover of $S^3$ along $K$. Let $g_k$ be a
generator of the covering transformation group. We put
$
\tau_k:=g_k^{2^{k-1}}.
$
Then $\tau_k$ is an involution on $\Sigma_{2^k}(K)$. We denote by
$\mathfrak{s}_k$ the spin structure on $\Sigma_{2^k}(K)$ used in the real
branched-cover construction; with respect to this spin structure, $\tau_k$ is
odd. We define
$$
SWF_R^{(k)}(K)
:=
SWF_R(\Sigma_{2^k}(K),\mathfrak{s}_k,\tau_k).
$$
\end{defn}
At the level of local equivalence classes,
$$
[SWF_R^{(k)}(K)]_{\mathrm{loc}}
:=
[SWF_R(\Sigma_{2^k}(K),\mathfrak{s}_k,\tau_k)]_{\mathrm{loc}}
\in \mathcal{LE}_G.
$$
The corresponding numerical invariants are defined by
$
\delta_R^{(k)}(K)
:=
\delta_R(\Sigma_{2^k}(K),\mathfrak{s}_k,\tau_k),
$
and similarly
$
\underline{\delta}_R^{(k)}(K),
\overline{\delta}_R^{(k)}(K),
\kappa_R^{(k)}(K).
$ These are all concordance invariants of $ K$.

\begin{defn}\label{def:kth-real-SWF-homology-S1xS2}
Let $Y$ be a homology $S^1\times S^2$. For each $k\geq1$, let
$
\widetilde{Y}_k\to Y
$
be the connected $\Z_{2^k}$--cover corresponding to the reduction
$
H_1(Y;\Z)\cong\Z\to \Z_{2^k}.
$
Let $g_k$ be a generator of the deck transformation group and put
$
\tau_k:=g_k^{2^{k-1}}.
$
Then $\tau_k$ is an involution on $\widetilde{Y}_k$. We take the spin structure
$\mathfrak{s}_k$ on $\widetilde{Y}_k$ which does not descend to
$\widetilde{Y}_k/\langle\tau_k\rangle$. With respect to this spin structure,
$\tau_k$ is odd. We define
$$
SWF_R^{(k)}(Y)
:=
SWF_R(\widetilde{Y}_k,\mathfrak{s}_k,\tau_k),
$$
and
$$
[SWF_R^{(k)}(Y)]_{\mathrm{loc}}
:=
[SWF_R(\widetilde{Y}_k,\mathfrak{s}_k,\tau_k)]_{\mathrm{loc}}
\in \mathcal{LE}_G.
$$
\end{defn}
The corresponding numerical invariants are defined by
$
\delta_R^{(k)}(Y)
:=
\delta_R(\widetilde{Y}_k,\mathfrak{s}_k,\tau_k),
$
and similarly
$
\underline{\delta}_R^{(k)}(Y),
\overline{\delta}_R^{(k)}(Y),
\kappa_R^{(k)}(Y).
$
As proved in \cite{MPT25}, if
$
Y=S^3_0(K)
$
is the $0$--surgery on a knot $K$, then the invariants defined from the
branched covers of $K$ agree with the invariants defined from the cyclic covers
of $Y$:
$$
\delta_R^{(k)}(K)=\delta_R^{(k)}(Y), 
\underline{\delta}_R^{(k)}(K)=\underline{\delta}_R^{(k)}(Y), 
\overline{\delta}_R^{(k)}(K)=\overline{\delta}_R^{(k)}(Y), 
\kappa_R^{(k)}(K)=\kappa_R^{(k)}(Y).
$$

\begin{prop}\label{prop:eighth-integrality-knot-invariants}
For every knot \(K\subset S^3\) and every \(k\geq1\), one has
\[
\delta_R^{(k)}(K),
\quad
\underline{\delta}_R^{(k)}(K),
\quad
\overline{\delta}_R^{(k)}(K),
\quad
\kappa_R^{(k)}(K)
\in
\frac18\mathbb Z.
\]
\end{prop}

\begin{proof}
Put $Y_k:=\Sigma_{2^k}(K)$.
One can see that \(Y_k\) is a \(\mathbb Z_2\)--homology sphere. 
Let \(\mathfrak s_k\) be the unique spin structure on \(Y_k\).  Choose a simply connected compact spin \(4\)--manifold \(W\) bounding
\((Y_k,\mathfrak s_k)\), and set $
L:=H_2(W;\mathbb Z)/\operatorname{Tor}$.
Since \(H_1(Y_k;\mathbb Z)\) has odd order, the long exact sequence of
the pair \((W,Y_k)\), together with Poincaré--Lefschetz duality, implies
that the intersection form $
Q_W\colon L\times L\longrightarrow\mathbb Z$
has odd determinant. Since \(W\) is spin, \(Q_W\) is even. Therefore
\(Q_W\bmod 2\) is a nonsingular alternating form, and hence $
b_2(W)=\operatorname{rank}L\equiv0\pmod2$.
It follows that \(\sigma(W)\equiv0\pmod2\).

Recall that the grading correction term is $
n(Y_k,\mathfrak s_k,g)
=
\operatorname{ind}^{\mathrm{APS}}_{\mathbb C}(D_W^+)
+
\frac{\sigma(W)}8$.
The APS index is an integer, and \(\sigma(W)\) is even.  Consequently, $
n(Y_k,\mathfrak s_k,g)\in\frac14\mathbb Z$.

The real Floer spectrum is represented by a triple whose complex grading
coordinate is
\[
\dim_{\mathbb C}(W^0_{-\lambda})^I
+
\frac12n(Y_k,\mathfrak s_k,g).
\]
Hence its rational grading coordinate belongs to $\frac18\mathbb Z$.
The definitions of
\(\delta_R,\underline{\delta}_R,\overline{\delta}_R\) and
\(\kappa_R\) differ from this grading coordinate by half-integral
quantities.  Therefore $
\delta_R^{(k)}(K), 
\underline{\delta}_R^{(k)}(K), 
\overline{\delta}_R^{(k)}(K), 
\kappa_R^{(k)}(K)
\in
\frac18\mathbb Z$.
\end{proof}

\subsection{Real 4--orbifolds and real Bauer--Furuta invariants}

Let $(X,\mathfrak{s})$ be an oriented spin $4$--orbifold with only isolated
singularities and with non-empty boundary. Let
$
\tau:X\to X
$
be an orientation-preserving smooth involution preserving the spin structure.
We fix a lift
$
\widetilde{\tau}:P_{\mathfrak{s}}\to P_{\mathfrak{s}}
$
of $d\tau$ to the orbifold principal $\Spin(4)$--bundle associated with
$\mathfrak{s}$. We say that $(X,\mathfrak{s},\tau,\widetilde{\tau})$ is odd if
$
\widetilde{\tau}^2=-1
$
on the spin bundle.

Put
$
G:=\langle j\rangle\cong\Z_4, 
H:=\langle -1\rangle\subset G.
$
We also write
$
b^+_-(X,\tau):=
\dim H^+(X;\R)^{-\tau^*}.
$
Under the assumptions below, this number is equal to $
b^+(X)-b^+(X/\Z_2).
$
The following is the main technical theorem in this paper. 
\begin{thm}\label{thm:BF}
Let $(X,\mathfrak{s},\tau,\widetilde{\tau})$ be an odd spin $4$--orbifold with
only isolated singularities and with connected boundary
$
\partial X=Y.
$
Assume
$
b_1(X)-b_1(X/\Z_2)=b_1(Y)-b_1(Y/\Z_2)=0.
$
Then one can construct a $G$--equivariant stable homotopy class represented by
a pointed $G$--equivariant map
$$
BF_{(X,\mathfrak{s},\tau,\widetilde{\tau})}:
\left(\C_+^{\,-\frac{1}{2}w}\right)^+
\longrightarrow
\left(\widetilde{\R}^{\,b^+_-(X,\tau)}\right)^+
\wedge
SWF_R(Y,\mathfrak{s}|_Y,\tau|_Y,\widetilde{\tau}|_Y).
$$
Here
\[
w
:=
\operatorname{ind}_{\C}^{\operatorname{orb}}D_X^+
-
n(Y,\mathfrak{s}|_Y,g_Y),
\]
where $\operatorname{ind}_{\C}^{\operatorname{orb}}D_X^+$ is the orbifold Dirac index with APS condition with a fixed $\Z_2$-invariant metric $g$ on $X$ which is product near the boundary and $g_Y$ is the boundary metric induced by  $g$.  This combination is
independent of the metric and the APS auxiliary choices.  It will be
identified with the Fukumoto--Furuta invariant in
\cref{prop:wR-vs-orbifold-index}.

Moreover, after taking $H$--fixed points and using the standard identification
of the $H$--fixed part of the normalized boundary Floer object with $S^0$, the
induced map is $G$--equivariantly stably homotopic to the map
$$
S^0
\longrightarrow
\left(\widetilde{\R}^{\,b^+_-(X,\tau)}\right)^+
$$
obtained by one-point compactifying the linear inclusion of the zero vector
space into $H^+(X;\R)^{-\tau^*}$.
\end{thm}

The proof of this theorem is essentially a combination of the original construction of (real versions of) Bauer--Furuta invariants for 4-manifolds with boundary \cite{Ma03,Kha15, KMT21} and orbifold Seiberg--Witten theory for 4-manifolds with cyclic isolated singularities \cite{FF00}. We do not repeat the arguments; instead, we just give an outline of the proof.

We now recall the analytic setup used to construct the map in
\cref{thm:BF}. Choose an orbifold Riemannian metric $g$ on $X$ which is
$\tau$--invariant and product near the boundary. Let
$
P_{\mathfrak{s}}\to X
$
be the orbifold principal $\Spin(4)$--bundle associated with the spin structure.
The positive and negative spinor bundles are
$$
S^\pm_{\mathfrak{s}}
:=
P_{\mathfrak{s}}\times_{\Spin(4)}\C^2_\pm.
$$
Since $\mathfrak{s}$ is a spin structure, the determinant line bundle is
trivial. We denote by $A_0$ the product connection on this trivial determinant
line bundle near the boundary.

The Clifford multiplication is an orbifold bundle homomorphism
$$
\rho:T^*X\to \operatorname{Hom}_{\C}(S^+_{\mathfrak{s}},S^-_{\mathfrak{s}}).
$$
A $U(1)$--connection $A$ on the determinant line, together with the
Levi--Civita connection of $g$, determines the orbifold Dirac operator
$$
D_A^+:\Gamma(S^+_{\mathfrak{s}})\to\Gamma(S^-_{\mathfrak{s}}).
$$
We denote by $F_A^+\in\Omega^+(X;i\R)$ the self-dual part of the curvature of
$A$.

Fix $k\gg0$. The configuration space is
$$
\mathcal C_k(X)
:=
\left(A_0+\check L^2_k(i\Lambda^1_X)\right)
\oplus
\check L^2_k(S^+_{\mathfrak{s}}),
$$
where $\check L^2_k$ denotes the Sobolev completion defined using orbifold
charts, the orbifold metric, and orbifold bundle connections. The gauge group is
$$
\mathcal G_{k+1}(X)
:=
\check L^2_{k+1}(X,S^1),
$$
the group of orbifold $U(1)$--gauge transformations of Sobolev class
$\check L^2_{k+1}$. It acts on $\mathcal C_k(X)$ by
$$
u\cdot(A,\Phi)
=
(A-u^{-1}du,u\Phi).
$$

We work in the double Coulomb slice
$
\mathcal U_k(X)
:=
\check L^2_k(i\Lambda^1_X)_{\operatorname{CC}}
\oplus
\check L^2_k(S^+_{\mathfrak{s}}),
$
where
$$
\check L^2_k(i\Lambda^1_X)_{\operatorname{CC}}
:=
\left\{
a\in\check L^2_k(i\Lambda^1_X)
\;\middle|\;
d_X^*a=0,\quad d_Y^*(r(a))=0
\right\}.
$$
Here $r(a)$ denotes the restriction of $a$ to the boundary, and $d_Y^*$ is the
formal adjoint of $d$ on $Y=\partial X$.

Let
$$
V(Y)=L^2_{k-1/2}(i\ker d_Y^*)\oplus L^2_{k-1/2}(\mathbb S_Y)
$$
be the Coulomb slice on the boundary, and let
$
r:\mathcal U_k(X)\to V(Y)
$
be the restriction map. For a real number $\mu$, let
$$
p^\mu_{-\infty}:V(Y)\to V^\mu_{-\infty}(Y)
$$
be the projection to the direct sum of eigenspaces of the boundary linearized
operator with eigenvalues in $(-\infty,\mu]$.

The Seiberg--Witten map with boundary projection is
$$
\mathcal F_X^\mu:
\mathcal U_k(X)
\to
\check L^2_{k-1}(i\Lambda_X^+)
\oplus
\check L^2_{k-1}(S^-_{\mathfrak{s}})
\oplus
V^\mu_{-\infty}(Y),
$$
defined by
$$
\mathcal F_X^\mu(a,\phi)
:=
\left(
\frac{1}{2}F^+_{A_0+a}
-
\rho^{-1}((\phi\phi^*)_0),
\;
D^+_{A_0+a}\phi,
\;
p^\mu_{-\infty}r(a,\phi)
\right).
$$
We write
$$
\mathcal F_X^\mu=L_X^\mu+C_X,
$$
where the linear part is
$$
L_X^\mu(a,\phi)
=
\left(
d^+a,
D^+_{A_0}\phi,
p^\mu_{-\infty}r(a,\phi)
\right),
$$
and the nonlinear part is
$$
C_X(a,\phi)
=
\left(
-\rho^{-1}((\phi\phi^*)_0),
\rho(a)\phi,
0
\right).
$$
The map $C_X$ is compact as a map from $\mathcal U_k(X)$ to the target Sobolev
space above.

The odd lift $\widetilde{\tau}$, together with the quaternionic element $j$,
defines an involution on the configuration space by
$$
I_X(a,\phi)
=
(-\tau^*a,\;j\,\widetilde{\tau}(\phi)).
$$
The map $I_X$ preserves the double Coulomb slice, and the Seiberg--Witten map
$\mathcal F_X^\mu$ is compatible with the corresponding involution on the
target. Therefore, we may restrict $\mathcal F_X^\mu$ to the fixed-point space
$
\mathcal U_k(X)^{I_X}.
$
Taking finite-dimensional approximation of this restricted map gives the
$G$--equivariant stable homotopy class
$
BF_{(X,\mathfrak{s},\tau,\widetilde{\tau})}
$
appearing in \cref{thm:BF}.

Now, we suppose there is a smooth involution $\tau : X \to X$ in the orbifold sense. We suppose $\tau^* \s \cong \s$ as an orbifold spin structure, and we suppose there is an order four lift 
\[
\wt{\tau} : P_\s \to P_\s
\]
of $\tau$ in the orbifold sense. We have an induced action of $\wt{\tau}$:
\[
\wt{\tau}^* : \mathcal{U}_{k}(X)  \to \mathcal{U}_{k}(X) \text{ and } \wt{\tau}^* : \check L^2_{k-1}(i\Lambda^+_{X} \oplus S^-_{X})\to \check L^2_{k-1}(i\Lambda^+_{X} \oplus S^-_{X})
\]
and $ \wt{\tau}^* |_Y : V^\mu_{-\infty}(Y) \to  V^\mu_{-\infty}(Y) $. 
Since the action $\wt{\tau}$ commutes with $\mathrm{Spin}(4)$-action, we see that the Seiberg--Witten map $\mathcal{F}_{X}$ commutes with $\wt{\tau}^*$.
Now, we define 
\[
I := j \circ \wt{\tau}^* : \mathcal{U}_{k}(X) \to \mathcal{U}_{k}(X),
\]
where $j$ acts as $j \in  \mathrm{Pin(2)}$.
Similarly we have actions of $I$ on $L^2_{k-1}(i\Lambda^+_{X} \oplus S^-_{X})$ and $V^\mu_{-\infty}$. This gives a real Seiberg--Witten map 
\[
\mathcal{F}_{X}^I+p^{\mu}_{-\infty} \circ r^I : \mathcal{U}_{k}^I(X) \to \check L^2_{k-1}(i\Lambda^+_{X} \oplus S^-_{\s})^I \oplus V^\mu_{-\infty}(Y)^I
\]
as the restriction of $\mathcal{F}_{X}+p^{\mu}_{-\infty} \circ r$. 
We now construct the finite-dimensional approximation of the real Seiberg--Witten
map on $X$. Let
$$
\mathcal V_{k-1}^I(X)
:=
\check L^2_{k-1}(i\Lambda_X^+)^I
\oplus
\check L^2_{k-1}(S^-_{\mathfrak{s}})^I.
$$
We write the $I$--fixed Seiberg--Witten map with boundary restriction as
$$
\mathcal F_X^I:
\mathcal U_k(X)^I
\longrightarrow
\mathcal V_{k-1}^I(X)\oplus V(Y)^I.
$$
It is of the form
$
\mathcal F_X^I
=
\left(
L_X^I+C_X^I,\; r^I
\right),
$
where $L_X^I$ is the restriction of the linear part, $C_X^I$ is compact, and
$r^I$ is the boundary restriction.

Let
$
l_Y^I:V(Y)^I\to V(Y)^I
$
be the linearized boundary operator. For $\lambda>0$, let
$
V^\lambda_{-\lambda}(Y)^I
$
denote the direct sum of the eigenspaces of $l_Y^I$ with eigenvalues in
$(-\lambda,\lambda]$.

Choose increasing finite-dimensional subspaces
$$
V_1\subset V_2\subset\cdots\subset \mathcal V_{k-1}^I(X)
$$
and an increasing sequence $\lambda_n\to\infty$ such that the orthogonal
projections
$$
P_n:
\mathcal V_{k-1}^I(X)\oplus V(Y)^I
\longrightarrow
V_n\oplus V^{\lambda_n}_{-\lambda_n}(Y)^I
$$
converge strongly to the identity. We further assume that
$$
\operatorname{Coker}
\left(
L_X^I\oplus p^{\lambda_n}_{-\lambda_n}r^I
\right)
\subset
V_n\oplus V^{\lambda_n}_{-\lambda_n}(Y)^I
$$
for every sufficiently large $n$. Here
$
p^{\lambda_n}_{-\lambda_n}:V(Y)^I\to V^{\lambda_n}_{-\lambda_n}(Y)^I
$
is the $L^2$--orthogonal projection.

Define
$$
U_n
:=
\left(
L_X^I\oplus p^{\lambda_n}_{-\lambda_n}r^I
\right)^{-1}
\left(
V_n\oplus V^{\lambda_n}_{-\lambda_n}(Y)^I
\right).
$$
 The
finite-dimensional approximation is
$
\mathcal F_n:
U_n
\longrightarrow
V_n\oplus V^{\lambda_n}_{-\lambda_n}(Y)^I,
$
given by
$$
\mathcal F_n(u)
:=
P_n
\left(
L_X^Iu+C_X^I(u),\;
p^{\lambda_n}_{-\lambda_n}r^I(u)
\right).
$$
Let $(N_n,L_n)$ be an index pair for the finite-dimensional approximation of the
downward gradient flow of the Chern--Simons--Dirac functional on
$V^{\lambda_n}_{-\lambda_n}(Y)^I$. Thus
$
N_n/L_n
$
is a $\Z_4$-equivariant Conley index used in the definition of
$SWF_R(Y,\mathfrak{s}|_Y,\tau|_Y,\widetilde{\tau}|_Y)$. If we take an index pair nicely, we can ensure: 

\begin{prop}\label{prop:relative-real-BF-map}
For all sufficiently large $n$, there exist $R>0$ and $\varepsilon_n>0$ such that
the finite-dimensional approximation $\mathcal F_n$ induces a well-defined
pointed $G$--equivariant map
$$
B(U_n;R)/S(U_n;R)
\longrightarrow
\left(
V_n/B(V_n;\varepsilon_n)^c
\right)
\wedge
(N_n/L_n).
$$
\end{prop}

This follows from the standard compactness argument for the relative
finite-dimensional approximation; see \cite{Ma03,Kha15}.

\begin{proof}
This follows from the standard compactness theorem for the relative
finite-dimensional approximation: for sufficiently large $n$, after choosing
$R\gg0$ and $\varepsilon_n>0$ sufficiently small, the map $\mathcal F_n$ sends
$S(U_n;R)$ and the preimage of $B(V_n;\varepsilon_n)^c\cup L_n$ to the base
point. Hence $\mathcal F_n$ descends to the desired pointed $G$--equivariant
map.
\end{proof}

Using \cref{prop:relative-real-BF-map}, we prove \cref{thm:BF}. 
\begin{proof}[Proof of \cref{thm:BF}]
Since
$$
B(U_n;R)/S(U_n;R)\cong U_n^+, \qquad
V_n/B(V_n;\varepsilon_n)^c\cong V_n^+,
$$
\cref{prop:relative-real-BF-map} gives a pointed $G$--equivariant map
$
U_n^+
\longrightarrow
V_n^+\wedge (N_n/L_n).
$
After stabilizing with respect to $n$, this map defines a stable $G$--equivariant
homotopy class
$
BF_{(X,\mathfrak{s},\tau,\widetilde{\tau})}.
$
Using the index calculation for
$
L_X^I\oplus p^{\lambda_n}_{-\lambda_n}r^I,
$
the domain and target suspensions can be rewritten in the normalized form
$$
BF_{(X,\mathfrak{s},\tau,\widetilde{\tau})}:
\left(\C_+^{\,-\frac{1}{2}w}\right)^+
\longrightarrow
\left(\widetilde{\R}^{\,b^+_-(X,\tau)}\right)^+
\wedge
SWF_R(Y,\mathfrak{s}|_Y,\tau|_Y,\widetilde{\tau}|_Y).
$$
Here
$
b^+_-(X,\tau)=\dim H^+(X;\R)^{-\tau^*}.
$

The stable homotopy class is independent of the choices of the finite-dimensional
subspaces, the cut-off parameters $\lambda_n$, the radius $R$, the numbers
$\varepsilon_n$, and the index pairs. Indeed, different choices can be compared
after passing to a common cofinal sequence of finite-dimensional approximations,
and the corresponding maps are related by the standard continuation homotopy for
the Seiberg--Witten map and for the Conley index of the boundary flow. Thus the
stable class
$
BF_{(X,\mathfrak{s},\tau,\widetilde{\tau})}
$
is well-defined.

It remains to identify the $H$--fixed point part. On the $H$--fixed point set, the
spinor directions vanish and the Seiberg--Witten map reduces to its linear
$1$--form part. More precisely, the $H$--fixed part of the finite-dimensional
approximation is the one-point compactification of the linear map
$$
L_{X,\R}^{I,H}
:
\left(\check L^2_k(i\Lambda^1_X)_{\operatorname{CC}}\right)^{I,H}
\longrightarrow
\left(\check L^2_{k-1}(i\Lambda_X^+)\oplus V^0_{-\infty}(Y)\right)^{I,H}.
$$
By the standard calculation, this linear map is injective and its cokernel
is naturally identified with
$
H^+(X;\R)^{-\tau^*}.
$
Therefore, after identifying the $H$--fixed point set of the normalized boundary
Floer object with $S^0$, the $H$--fixed part of the Bauer--Furuta map is stably
homotopic to the one-point compactification of a linear inclusion
$
0
\longrightarrow
H^+(X;\R)^{-\tau^*}.
$
This completes the proof. 
\end{proof}

Let $(Y,\mathfrak{s})$ be an oriented closed spin 
$3$--manifold. Let
$(X,\widetilde{\mathfrak{s}})$ be an oriented compact spin
$4$--orbifold with only isolated singularities and with boundary
$
\partial (X,\widetilde{\mathfrak{s}})
=
(Y,\mathfrak{s}).
$
Choose a compact smooth spin $4$--manifold $(W,\mathfrak{s}_W)$ such that
$
\partial(W,\mathfrak{s}_W)=-(Y,\mathfrak{s}).
$
We emphasize that no
involution on $W$ is required. Define the closed spin $4$--orbifold
$$
Z_{X,W}:=X\cup_Y W.
$$

\begin{defn}\label{def:real_w}
The {\it Fukumoto--Furuta's $w$-invariant} is defined as 
$$
w(X,\widetilde{\mathfrak{s}})
:=
\operatorname{ind}_{\C}^{\operatorname{orb}}D_{Z_{X,W}}
+
\frac{1}{8}\sigma(W)
\in\Q,
$$
where $D_{Z_{X,W}}$ is the complex spin Dirac operator on the closed spin
orbifold $Z_{X,W}$.
\end{defn}
By the excision argument, $
w(X,\widetilde{\mathfrak{s}})$
is independent of the choice of the spin filling
$(W,\mathfrak{s}_W)$.  We shall call $
w_R(X,\widetilde{\mathfrak{s}})
:=
\frac12w(X,\widetilde{\mathfrak{s}})$ 
the real $w$--invariant.

We next compare this invariant with the relative orbifold Dirac index appearing
in the Bauer--Furuta map. Choose a product-type orbifold metric on $X$ near the
boundary, and let $g_Y$ be its restriction to $Y$. Let
$
\operatorname{ind}_{\C}^{\operatorname{orb}}D_X^+
$
denote the APS index of the positive spin Dirac operator on the spin orbifold
$(X,\widetilde{\mathfrak{s}})$ with respect to the boundary metric $g_Y$. This
is the analytical index denoted by
$
\operatorname{ind}_{\C}^{\operatorname{orb}}D_{(X,\widetilde{\mathfrak{s}})}^+
$
in the construction of the relative Bauer--Furuta invariant.

\begin{prop}\label{prop:wR-vs-orbifold-index}
With the above convention,
$$
\operatorname{ind}_{\C}^{\operatorname{orb}}D_X^+
=
w(X,\widetilde{\mathfrak{s}})
+
n(Y,\mathfrak{s},g_Y).
$$
\end{prop}

\begin{proof}
Let $D_Y$ denote the spin Dirac operator on the boundary $Y$, and put
$
h_Y:=\dim_{\C}\ker D_Y.
$
We use the standard APS boundary condition, without assuming that $D_Y$ is
invertible.

Let $W$ be a compact smooth spin $4$--manifold with
$
\partial W=-Y.
$
Equip $W$ with a product-type metric near the boundary whose boundary value is
$g_Y$ with the reversed orientation. By the APS gluing formula, with the kernel
correction included, we have
$$
\operatorname{ind}_{\C}^{\operatorname{orb}}D_{Z_{X,W}}^+
=
\operatorname{ind}_{\C}^{\operatorname{orb}}D_X^+
+
\operatorname{ind}_{\C}D_W^+
+
h_Y.
$$
Therefore
$$
\begin{aligned}
 w(X,\widetilde{\mathfrak{s}})
&=
\operatorname{ind}_{\C}^{\operatorname{orb}}D_{Z_{X,W}}^+
+
\frac{\sigma(W)}{8}
\\
&=
\operatorname{ind}_{\C}^{\operatorname{orb}}D_X^+
+
\left(
\operatorname{ind}_{\C}D_W^+
+
\frac{\sigma(W)}{8}
\right)
+
h_Y.
\end{aligned}
$$
The expression in parentheses is the Manolescu correction term for $-Y$:
$$
n(-Y,\mathfrak{s},g_Y)
=
\operatorname{ind}_{\C}D_W^+
+
\frac{\sigma(W)}{8}.
$$
When $\ker D_Y$ is non-zero, the orientation reversal formula for this correction
term is
$
n(-Y,\mathfrak{s},g_Y)
=
-n(Y,\mathfrak{s},g_Y)-h_Y.
$
Hence
$$
\begin{aligned}
 w(X,\widetilde{\mathfrak{s}})
&=
\operatorname{ind}_{\C}^{\operatorname{orb}}D_X^+
+
n(-Y,\mathfrak{s},g_Y)
+
h_Y
\\
&=
\operatorname{ind}_{\C}^{\operatorname{orb}}D_X^+
-
n(Y,\mathfrak{s},g_Y).
\end{aligned}
$$
This proves the desired formula.
\end{proof}
The following is the main theorem.

\begin{thm}\label{general BF}
Let $(Y,\mathfrak{s},\tau,\widetilde{\tau})$ be an odd spin closed
$3$--manifold satisfying
$
H^1(Y;\R)^{-\tau^*}=0.
$
Let $(X,\widetilde{\mathfrak{s}},\widetilde{\tau})$ be an odd spin compact
$4$--orbifold with
$
\partial(X,\widetilde{\mathfrak{s}},\widetilde{\tau})
=
(Y,\mathfrak{s},\tau).
$
Assume
$$
H^1(X;\R)^{-\widetilde{\tau}^*}=0
\qquad\text{and}\qquad
H^2(X;\R)^{-\widetilde{\tau}^*}=0.
$$
Then
$$
[SWF_R(Y,\mathfrak{s},\tau)]_{\mathrm{loc}}
=
[(\C_+^{\,-\frac{1}{2} w(X,\widetilde{\mathfrak{s}})})^+]_{\mathrm{loc}}
\in\mathcal{LE}_G.
$$
In particular, the real $w$--invariant
$
\frac{1}{2} w(X,\widetilde{\mathfrak{s}})
$
depends only on the boundary odd spin involutive $3$--manifold
$(Y,\mathfrak{s},\tau,\widetilde{\tau})$.
\end{thm}
\begin{proof}
The assumptions on the first cohomology groups imply
\[
b_1(X)-b_1(X/\langle\widetilde{\tau}\rangle)=0,
\qquad
b_1(Y)-b_1(Y/\langle\tau\rangle)=0.
\]
Moreover, the condition $H^2(X;\R)^{-\widetilde{\tau}^*}=0$
implies
\[
b^+_-(X,\widetilde{\tau})
=
\dim H^+(X;\R)^{-\widetilde{\tau}^*}
=
0.
\]
Hence \cref{thm:BF} gives a stable \(G\)--equivariant map
\[
\left(
\C_+^{\,\frac12\left(
-\operatorname{ind}_{\C}^{\operatorname{orb}}D_X^+
+
n(Y,\mathfrak{s},g_Y)
\right)}
\right)^+
\longrightarrow
SWF_R(Y,\mathfrak{s},\tau).
\]
Its restriction to the \(H\)--fixed point sets is a stable equivalence, since
the \(H\)--fixed part of the relative Bauer--Furuta map is the
one-point compactification of
\[
0\longrightarrow H^+(X;\R)^{-\widetilde{\tau}^*}=0.
\]
Applying the same construction to \(-X\), together with equivariant
Spanier--Whitehead duality as in \cite[Section~3.3]{KMT:2023}, gives a
local map in the opposite direction. Therefore, the two objects are
locally equivalent. 
Therefore, we see 
\[
[SWF_R(Y,\mathfrak{s},\tau)]_{\mathrm{loc}}
=
\left[
\left(
\C_+^{\,\frac12\left(-
\operatorname{ind}_{\C}^{\operatorname{orb}}D_X^+
+
n(Y,\mathfrak{s},g_Y)
\right)}
\right)^+
\right]_{\mathrm{loc}}.
\]
By \cref{prop:wR-vs-orbifold-index},
\[
\operatorname{ind}_{\C}^{\operatorname{orb}}D_X^+
-
n(Y,\mathfrak{s},g_Y)
=
w(X,\widetilde{\mathfrak{s}}).
\]
Consequently,
\[
[SWF_R(Y,\mathfrak{s},\tau)]_{\mathrm{loc}}
=
\left[
\left(
\C_+^{\,-\frac12w(X,\widetilde{\mathfrak{s}})}
\right)^+
\right]_{\mathrm{loc}}.
\]

The left-hand side depends only on the boundary odd spin involutive
\(3\)--manifold. Since the \(\C_+\)--suspension grading is preserved under
local equivalence, the number $
\frac12w(X,\widetilde{\mathfrak{s}})$
also depends only on
\((Y,\mathfrak{s},\tau)\).
\end{proof}

\begin{thm}[\Cref{thm:main}]
Let $(Y,\mathfrak{s},\tau)$ be a spin Seifert $3$--manifold with genus-zero oriented base, equipped with an odd involution $\tau$ commuting with the Seifert $S^1$--action. Then there exists a compact spin $4$--orbifold $(\check X,\widetilde{\mathfrak{s}})$ constructed as above such that $\partial(\check X,\widetilde{\mathfrak{s}})=(Y,\mathfrak{s})$ and
\[
[SWF_R(Y,\mathfrak{s},\tau)]_{\mathrm{loc}}
=
\left[
\left(
\C_+^{\, -\frac12 w(\check X,\widetilde{\mathfrak{s}})}
\right)^+
\right]_{\mathrm{loc}}.
\]
Here $w(\check X,\widetilde{\mathfrak{s}})$ denotes the Fukumoto--Furuta $w$--invariant.
\end{thm}

\begin{proof}
By \cref{classify_over}, the involution $\tau$ is of type A or type C. We first consider the type A case. By \cref{Prop:extension}, the gluing data can be chosen so that there exist a compact spin $4$--manifold $(X_0,\widetilde{\mathfrak{s}}_0)$, an extension
$\widetilde{\tau}:X_0\longrightarrow X_0$
of $\tau$, and an odd lift of $\widetilde{\tau}$ extending the prescribed odd lift on $(Y,\mathfrak{s})$. By  \cref{lem:extension over orbifold caps}, the spin structure, the involution, and its odd lift extend over the orbifold caps. We therefore obtain an odd spin compact $4$--orbifold
$
(\check X,\widetilde{\mathfrak{s}},\widetilde{\tau})
$
satisfying
$
\partial(\check X,\widetilde{\mathfrak{s}},\widetilde{\tau})
=
(Y,\mathfrak{s},\tau)$.

We verify the cohomological assumptions of \cref{general BF}. By \cref{prop:involution},
\[
H_1(X_0;\Q)^{-\widetilde{\tau}_*}=0,
\qquad
H_2(X_0;\Q)^{-\widetilde{\tau}_*}=0.
\]
Each orbifold cap is a rational $4$--ball and is attached along a rational homology $3$--sphere. Thus, we see
\[
H^1(\check X;\R)^{-\widetilde{\tau}^*}=0,
\qquad
H^2(\check X;\R)^{-\widetilde{\tau}^*}=0.
\]
Applying \cref{general BF} to
$(\check X,\widetilde{\mathfrak{s}},\widetilde{\tau})$, we obtain
\[
[SWF_R(Y,\mathfrak{s},\tau)]_{\mathrm{loc}}
=
\left[
\left(
\C_+^{\, -\frac12w(\check X,\widetilde{\mathfrak{s}})}
\right)^+
\right]_{\mathrm{loc}}.
\]

In the type C case, $\tau$ is the corresponding gauge transformation of the Seifert circle bundle. It extends over $X_0$ and the orbifold caps by the same gauge transformation, and the resulting action on the relevant real cohomology groups is trivial. The same argument therefore proves the result.
\end{proof}

\begin{rem}\label{rem:positive-genus-case}
The genus-zero assumption in the preceding theorem is used only to verify,
for the 4-orbifold $\check X$ constructed above, the
extension of the odd spin involution and the cohomological assumptions in
Theorem \ref{general BF}. More generally, let $(Y,\mathfrak{s},\tau)$ be a
spin Seifert $3$--manifold with an odd involution whose oriented base has
arbitrary genus with $H^1(Y;\R)^{-\tau^*}=0$.
Suppose that $\tau$ commutes with the Seifert $S^1$--action. 
Then, these are classified as type A, B, and C, which are given in \cref{section:Involution}. 
There still exists a 4-orbifold $\check X$ with
$\partial \check X=Y$, an extension
$
\widetilde{\tau}:\check X\to \check X
$
of $\tau$.  
Moreover, one can check 
\[
        H^1(\check X;\R)^{-\widetilde{\tau}^*}=0,
        \qquad
        H^2(\check X;\R)^{-\widetilde{\tau}^*}=0. 
\]
However, in general, we do not have a spin structure on $\check{X}$ extending the spin structure on the boundary.
This is the place we need to suppose there is an orbifold spin structure
$\widetilde{\mathfrak{s}}$ on $\check X$ extending $\mathfrak{s}$, such that
the chosen odd lift of $\tau$ extends to a lift of $\widetilde{\tau}$ on
$\widetilde{\mathfrak{s}}$. 
then the same proof gives
\[
[SWF_R(Y,\mathfrak{s},\tau)]_{\mathrm{loc}}
=
\left[
\left(
\C_+^{\,-\frac{1}{2}w(\check X,\widetilde{\mathfrak{s}})}
\right)^+
\right]_{\mathrm{loc}}.
\]

\end{rem}

\subsection{Calculation of the $w$--invariant}

Let $(Y,\mathfrak{s},\tau)$ be an oriented Seifert $3$--manifold with a spin structure
$\mathfrak{s}$ and an odd involution $\tau$. 
Suppose $Y$ has a description
\[
Y= M((a_1, b_1), (a_1',b_1'),\cdots, (a_n,b_n), (a_n',b_n')). 
\]

Let
$(X,\widetilde{\mathfrak{s}})$ be a compact spin $4$--orbifold
with boundary
$
\partial(X,\widetilde{\mathfrak{s}})
=
(Y,\mathfrak{s}).
$
Choose a compact smooth spin $4$--manifold $(W,\mathfrak{s}_W)$ such that
$
\partial(W,\mathfrak{s}_W)=-(Y,\mathfrak{s})
$
 We emphasize that no involution on $W$
is required. Set
$
Z_{X,W}:=X\cup_Y W.
$
We have 
\[
w(X,\widetilde{\mathfrak{s}})
= \operatorname{ind}_{\C}D_{Z_{X,W}}
+
\frac{1}{8}\sigma(W).
\]
We now take $X$ to be $\check{X} $ and state the computation of Fukumoto--Furuta's invariant for $\check{X} $: 

\begin{thm}\label{thm:explicit-w-formula}
Assume the hypotheses of \cref{thm:main}, and suppose that $
\check X=\check X(\boldsymbol{\nu},\boldsymbol{\rho})$
is constructed with \(m=0\), so that all components of
\(\partial_+X_0\) arise from full tunnels. For \(i=1,\ldots,n\), put
\[
d_i:=\gcd(a_i,a'_i),\qquad
p_i:=\frac{a_i}{d_i},\qquad
q_i:=\frac{a'_i}{d_i},\qquad
\Delta_i:=p_i b'_i+q_i b_i.
\]
Assume that \(\wt a_i\neq0\) for every \(i\), and, after reordering the indices, that \(\Delta_n\neq0\). Choose integers
\[
a_i\rho_i+b_i\nu_i=1,\qquad
a'_i\rho'_i+b'_i\nu'_i=1.
\]
As in the construction of \(\check X\), set
\[
\wt a_i:=-a_i b'_i-a'_i b_i=-d_i\Delta_i,\qquad
\wt b_i:=\rho'_i a_i-\nu'_i b_i
\quad\in\Z/\wt a_i\Z.
\]
Let \(\wt{\mathfrak{s}}\) be an orbifold spin structure on \(\check X\) extending the spin structure on \(X_0(\boldsymbol{\nu},\boldsymbol{\rho})\) determined by
\[
\lambda_i,e_i,e'_i,\mu\in\Z_2
\]
as in \cref{prop:spin structures on X_0}. Then
\[
w(\check X,\wt{\mathfrak{s}})
=
-\frac18
\left(
\sign(Q)
+
\sum_{i=1}^n\frac1{\wt a_i}
\sum_{l=1}^{|\wt a_i|-1}
\left(
\cot\frac{\pi l}{\wt a_i}
\cot\frac{\pi\wt b_i l}{\wt a_i}
+
2(-1)^{e_i l}
\csc\frac{\pi l}{\wt a_i}
\csc\frac{\pi\wt b_i l}{\wt a_i}
\right)
\right),
\]
where
\[
Q=
\left(
-\delta_{ij}\Delta_iq_i a_i
-\frac{\Delta_i\Delta_j}{\Delta_n}q_n a_n
\right)_{1\leq i,j\leq n-1}.
\]
Here \(\sign(Q)\) denotes the signature of the rational quadratic form represented by \(Q\).
\end{thm}

\begin{proof}
Let \(\wt{\mathfrak{s}}_0\) denote the restriction of
\(\wt{\mathfrak{s}}\) to \(X_0\), parametrized by
\[
\lambda_i,e_i,e'_i,\mu\in\Z_2
\]
as in \cref{prop:spin structures on X_0}. Recall that our convention is
\[
0=\text{the nonbounding spin structure},\qquad
1=\text{the bounding spin structure}.
\]
With respect to the product framing on the boundary circle, the
nonbounding spin structure has spin holonomy \(+1\), whereas the
bounding spin structure has spin holonomy \(-1\).  Hence the central
sign distinguishing the two lifts of the local cyclic action is
\((-1)^{e_i}\).  Therefore, the lift of the local cyclic action on the
\(i\)-th orbifold cap is
\[
\begin{aligned}
\wt\zeta_i\cdot(z,w,(q_-,q_+))
=
\biggl(
\wt\zeta_i z,\,
\wt\zeta_i^{\wt b_i}w,\,
(-1)^{e_i}
\left(
\wt\zeta_i^{(1-\wt b_i)/2},
\wt\zeta_i^{(-1-\wt b_i)/2}
\right)(q_-,q_+)
\biggr),
\end{aligned}
\]
where $
\wt\zeta_i=e^{2\pi\sqrt{-1}/\wt a_i}$.
Thus the \(l\)-th power of the isotropy generator contributes the sign $
(-1)^{e_i l}.$

Choose a compact smooth spin \(4\)--manifold \(W\) with $
\partial(W,\mathfrak{s}_W)=-(Y,\mathfrak{s})$
and set \(Z_{\check X,W}:=\check X\cup_YW\). By Kawasaki's orbifold index theorem \cite{Kawasaki81},
together with the local spin calculation as in \cite{fukumoto2001w}, we have
\[
\begin{aligned}
\operatorname{ind}_{\C}D_{Z_{\check X,W}}
={}&-\frac{\sigma(\check X)+\sigma(W)}8\\
&+\sum_{i=1}^n\frac1{\wt a_i}
\sum_{l=1}^{|\wt a_i|-1}
\left[
\frac{(-1)^{e_i l}}
{(\wt\zeta_i^{l/2}-\wt\zeta_i^{-l/2})
(\wt\zeta_i^{\wt b_i l/2}-\wt\zeta_i^{-\wt b_i l/2})}
\right. \left.
+\frac18
\frac{
(\wt\zeta_i^{l/2}+\wt\zeta_i^{-l/2})
(\wt\zeta_i^{\wt b_i l/2}+\wt\zeta_i^{-\wt b_i l/2})
}{
(\wt\zeta_i^{l/2}-\wt\zeta_i^{-l/2})
(\wt\zeta_i^{\wt b_i l/2}-\wt\zeta_i^{-\wt b_i l/2})
}
\right]\\
={}&-\frac{\sigma(W)}8-\frac18
\left(
\sigma(\check X)
+
\sum_{i=1}^n\frac1{\wt a_i}
\sum_{l=1}^{|\wt a_i|-1}
\left(
\cot\frac{\pi l}{\wt a_i}
\cot\frac{\pi\wt b_i l}{\wt a_i}
+
2(-1)^{e_i l}
\csc\frac{\pi l}{\wt a_i}
\csc\frac{\pi\wt b_i l}{\wt a_i}
\right)
\right).
\end{aligned}
\]
By definition,
\[
w(\check X,\wt{\mathfrak{s}})
=
\operatorname{ind}_{\C}D_{Z_{\check X,W}}
+\frac{\sigma(W)}8.
\]
Moreover, the orbifold caps are rational \(4\)--balls, so $\sigma(\check X)=\sigma(X_0)$.
By \cref{prop:Q-intersection-form-on-X0}, the rational intersection form of \(X_0\) is represented by \(Q\), and hence $
\sigma(\check X)=\sign(Q)$.
Substituting this equality into the preceding index formula proves the result.
\end{proof}

Following Saveliev's combinatorial interpretation \cite{Sa02} of the
\(w\)--invariant, together with the even continued-fraction computation
in \cite[Propositions~6 and~9]{fukumoto2001w}, we obtain the following
purely combinatorial form of \Cref{thm:explicit-w-formula}.

For integers \(\alpha_1,\ldots,\alpha_r\), we use the convention
\[
[\alpha_1,\ldots,\alpha_r]^-
:=
\alpha_1-\cfrac{1}{
\alpha_2-\cfrac{1}{
\ddots-\cfrac{1}{\alpha_r}}}.
\]

\begin{prop}
\label{prop:continued-fraction-w-formula}
Assume the hypotheses of \Cref{thm:explicit-w-formula}.  For each
\(i=1,\ldots,n\), put
\[
A_i:=|\wt a_i|,
\qquad
\widehat b_i:=\rho'_i a_i-\nu'_i b_i\in\Z,
\qquad
s_i:=\operatorname{sgn}(\wt a_i).
\]

Suppose first that \(A_i>1\).  There is a unique integer \(c_i\) modulo
\(2\) satisfying
\[
c_i\equiv e_i+1\pmod2
\]
and a unique representative $
B_i
:=
s_i\bigl(\widehat b_i+c_i\wt a_i\bigr)$
satisfying $
0<|B_i|<A_i.$
Moreover,
\[
\gcd(A_i,B_i)=1,
\qquad
A_i+B_i\equiv1\pmod2.
\]
Consequently, there is a unique continued-fraction expansion
\[
\frac{A_i}{B_i}
=
[\alpha_{i,1},\ldots,\alpha_{i,r_i}]^-
\]
such that $
\alpha_{i,j}\in2\Z, 
|\alpha_{i,j}|\geq2$
for every \(j\).

If \(A_i=1\), set \(r_i=0\), so that the corresponding sum below is
empty.  Then
\[
w(\check X,\wt{\mathfrak s})
=
\frac18
\left(
\sum_{i=1}^n\sum_{j=1}^{r_i}
\operatorname{sgn}(\alpha_{i,j})
-
\sign(Q)
\right).
\]

\end{prop}

\begin{proof}
For relatively prime integers \(p,q\), with \(p\neq0\), and
\(\varepsilon\in\{\pm1\}\), define
\[
\sigma(q,p,\varepsilon)
:=
\frac1p
\sum_{\ell=1}^{|p|-1}
\left(
\cot\frac{\pi\ell}{p}
\cot\frac{\pi q\ell}{p}
+
2\varepsilon^\ell
\csc\frac{\pi\ell}{p}
\csc\frac{\pi q\ell}{p}
\right).
\]
Thus \Cref{thm:explicit-w-formula} can be written as
\[
w(\check X,\wt{\mathfrak s})
=
-\frac18
\left(
\sign(Q)
+
\sum_{i=1}^n
\sigma(\widehat b_i,\wt a_i,(-1)^{e_i})
\right).
\tag{\(\ast\)}
\]

The elementary transformation rules
\[
\sigma(q+cp,p,\varepsilon)
=
\sigma(q,p,(-1)^c\varepsilon)
\]
and
\[
\sigma(-q,p,\varepsilon)
=
\sigma(q,-p,\varepsilon)
=
-\sigma(q,p,\varepsilon)
\]
imply
\[
\begin{aligned}
\sigma(\widehat b_i,\wt a_i,(-1)^{e_i})
&=
\sigma(s_i\widehat b_i,A_i,(-1)^{e_i})\\
&=
\sigma(B_i,A_i,-1),
\end{aligned}
\tag{\(\dagger\)}
\]
because $c_i\equiv e_i+1\pmod2$.

We verify the parity condition needed for the even continued fraction.
If \(A_i\) is even, then \(B_i\) is odd since
\(\gcd(A_i,B_i)=1\).  If \(A_i\) is odd, the condition that the local
spin structure extend over the cone gives
$
e_i\equiv s_i\widehat b_i-1\pmod2$.
Since \(c_i\equiv e_i+1\pmod2\), it follows that
$
B_i
\equiv
s_i\widehat b_i+c_i
\equiv0
\pmod2.$ 
Thus, in both cases,
$ 
A_i+B_i\equiv1\pmod2.$

The even Euclidean algorithm therefore gives a unique expansion
\[
\frac{A_i}{B_i}
=
[\alpha_{i,1},\ldots,\alpha_{i,r_i}]^-,
\qquad
\alpha_{i,j}\in2\Z,
\quad
|\alpha_{i,j}|\geq2.
\]
By the continued-fraction formula for the local correction term,
\[
\sigma(B_i,A_i,-1)
=
-\sum_{j=1}^{r_i}
\operatorname{sgn}(\alpha_{i,j}).
\tag{\(\ddagger\)}
\]
Substituting \((\dagger)\) and \((\ddagger)\) into \((\ast)\) gives
\[
w(\check X,\wt{\mathfrak s})
=
\frac18
\left(
\sum_{i=1}^n\sum_{j=1}^{r_i}
\operatorname{sgn}(\alpha_{i,j})
-
\sign(Q)
\right).
\]
\end{proof}

\begin{rem}\label{prop:w-formula-half-tunnels}
The formulas of \Cref{prop:continued-fraction-w-formula} remain valid in the presence of
half tunnels, provided that each half tunnel with Seifert datum
$(a',b')$ is formally regarded as a full tunnel whose missing end has
data $
(a,b,\nu,\rho)=(1,0,0,1).$
Thus, for such a half tunnel,
$
d=1, p=1,q=a',
\Delta=b', 
\widetilde a=-b',
\widetilde b=\rho'$.
\end{rem}

\subsection{Application}

We give a proof of \cref{periodic}. 

\begin{proof}[Proof of \Cref{periodic}]
Set $
K:=T(p,q)$ and $
Y:=S^3_0(K)$.
Let $
Y_j\longrightarrow Y$
denote the connected cyclic cover of degree $2^j$, with $Y_0=Y$.

The cyclic-cover version of the zero-surgery trace argument in
\cite[proof of Theorem~4.10]{MPT25} gives
\begin{equation}\label{eq:knot-zero-surgery-local-equivalence}
[SWF_R^{(j)}(K)]_{\mathrm{loc}}
=
[SWF_R^{(j)}(Y)]_{\mathrm{loc}}
\end{equation}
for every $j\geq1$. Here one applies the same argument to the
$2^j$-fold cyclic branched cover of the zero-surgery trace.

By definition, the double cover $
Y_j\longrightarrow Y_{j-1}$
is the cover used to define the first real Floer homotopy type of
$Y_{j-1}$. Hence
\[
[SWF_R^{(j)}(Y)]_{\mathrm{loc}}
=
[SWF_R^{(1)}(Y_{j-1})]_{\mathrm{loc}}.
\]

We next recall the surface-bundle description of $Y$. The torus knot
$T(p,q)$ is the link of the isolated plane-curve singularity $
z_1^p+z_2^q=0$.
By the Milnor fibration theorem, it is a fibered knot with fiber
$\Sigma_{g,1}$, where
\[
g=\frac{(p-1)(q-1)}{2};
\]
see \cite[Sections~4 and~9]{Milnor68}. Let $
\psi\colon\Sigma_{g,1}\longrightarrow\Sigma_{g,1}$
be its monodromy, and let $
\bar\psi\colon\Sigma_g\longrightarrow\Sigma_g$
be the closed monodromy obtained by capping the boundary component.

Since the boundary of a fiber is the preferred longitude of the knot,
zero-framed Dehn filling caps every fiber by a disk. Consequently, $
Y\cong M_{\bar\psi}$,
where $M_f$ denotes the mapping torus of $f$; see also
\cite[Chapter~10]{Rolfsen03}. For the weighted homogeneous polynomial $z_1^p+z_2^q$, the geometric
monodromy is represented by
\[
h(z_1,z_2)
=
\left(
e^{2\pi i/p}z_1,
e^{2\pi i/q}z_2
\right);
\]
see \cite[Section~9]{Milnor68}. Its action on the boundary of the
Milnor fiber is a rotation and therefore extends over the capping disk.
Thus the closed monodromy $\bar\psi$ may be chosen to be periodic.
Since $p$ and $q$ are relatively prime, this periodic map has exact
order
\[
\operatorname{ord}(\bar\psi)=\operatorname{lcm}(p,q)=pq.
\]

The fibration $Y\to S^1$ induces an isomorphism on the free part of
$H_1(Y;\mathbb Z)$. Hence the connected cyclic cover of degree $2^j$
is obtained by pulling back the fibration along the degree-$2^j$ cover
of $S^1$. Therefore $
Y_j\cong M_{\bar\psi^{\,2^j}}$,
and in particular $
Y_{k-1}\cong M_{\bar\psi^{\,2^{k-1}}}$.
Write $
pq=2^{a_{p,q}}m_{p,q},
\ 
m_{p,q}\ \text{odd}$.
Let
\[
d_{p,q}:=\min\{d\geq1\mid 2^d\equiv1\pmod{m_{p,q}}\}.
\]
If $k\geq a_{p,q}+1$, then
\[
2^{k+d_{p,q}-1}-2^{k-1}
=
2^{k-1}(2^{d_{p,q}}-1)
\]
is divisible by $pq$. Since $\bar\psi$ has order $pq$, it follows that $
\bar\psi^{\,2^{k+d_{p,q}-1}}
=
\bar\psi^{\,2^{k-1}}$.
This implies $
Y_{k+d_{p,q}-1}\cong Y_{k-1}$.

We therefore obtain
\[
\begin{aligned}
[SWF_R^{(k+d_{p,q})}(Y)]_{\mathrm{loc}}=
[SWF_R^{(1)}(Y_{k+d_{p,q}-1})]_{\mathrm{loc}}=
[SWF_R^{(1)}(Y_{k-1})]_{\mathrm{loc}}=
[SWF_R^{(k)}(Y)]_{\mathrm{loc}}.
\end{aligned}
\]
Combining this with
\eqref{eq:knot-zero-surgery-local-equivalence}, we conclude that
\[
[SWF_R^{(k+d_{p,q})}(K)]_{\mathrm{loc}}
=
[SWF_R^{(k)}(K)]_{\mathrm{loc}}.
\]

The assertions for
$\delta_R^{(k)}$,
$\underline{\delta}_R^{(k)}$,
$\overline{\delta}_R^{(k)}$, and
$\kappa_R^{(k)}$
follow immediately. 
\end{proof}

\subsection{Linear independence} 

\begin{lem}\label{lem:delta-values-two-bridge-torus}
For \(n\geq2\), set $
K_n:=T(2,2^n-1)$.
Then
\[
\delta^{(k)}_R(K_n)
=\begin{cases}
\frac{2^{n-1}-1}{8} \qquad k=1 \\ 
-\frac18
\qquad
(2\leq k\leq n-1) \\ 
\frac{2^n-1}{8} \qquad k=n. 
\end{cases} 
\]
\end{lem}

\begin{proof}
Put $
r:=2^n-1, 
Y:=S^3_0(K_n)$,
and let $
Y_k\longrightarrow Y$
be the connected cyclic cover of degree \(2^k\).  Denote by
\(\tau_k\) the order-two deck transformation of
\(Y_k\to Y_{k-1}\), and by \(\mathfrak{s}_k\) the distinguished spin
structure used in the definition of the real invariants.  By the
zero-surgery comparison and \Cref{thm:Quantities},$ 
\delta_R^{(k)}(K_n)
=
\delta_R(Y_k,\mathfrak{s}_k,\tau_k)$. 

We have 
\[
Y
=
S^3_0(T(2,r))
\cong
M\left(
0;
(2,1),
\left(r,\frac{1-r}{2}\right),
(2r,-1)
\right).
\]
from \cite[Proposition 11.3.11]{MartelliGT}. 

We next determine the Seifert presentation of its cyclic covers.
Let \(F_r\) be the closed fiber of \(T(2,r)\).  It can be identified
with the compactification of
\[
F_r=\{(x,y)\in\mathbb C^2\mid y^2=1-x^r\}.
\]
Writing \(\zeta_r=e^{2\pi i/r}\), the closed monodromy is $
\overline{\psi}(x,y)
=
(\zeta_r x,-y)$,
which has order \(2r\).  Since zero-framed surgery caps off the boundary
of every page, $
Y\cong M_{\overline{\psi}}$ which implies  $ 
Y_k\cong M_{\overline{\psi}^{\,2^k}}$.

Fix \(2\leq k\leq n\), and put $
u:=2^{n-k}, 
v:=2^k$.
Then $
uv=2^n=r+1, 
uv\equiv1\pmod r.
$ 
Since \(v\) is even, $
\overline{\psi}^{\,v}(x,y)
=
(\zeta_r^v x,y)$.
This periodic map has order \(r\).  Its quotient is a sphere, and it
has three fixed points: the two points $
(0,1), 
(0,-1)$,
and the point at infinity.  The local rotation numbers at these points
are $
v, v, -\frac v2 \in \mathbb Z/r\mathbb Z$,
where division by \(2\) is taken in \(\mathbb Z/r\mathbb Z\).  Indeed,
\(x\) is a local coordinate at the first two points, while at infinity
one may use a local coordinate \(t\) satisfying \(x=t^{-2}\).

With the convention for Seifert invariants used in this paper, a local
rotation number \(q\) contributes a pair \((r,b)\) satisfying $ 
qb\equiv-1\pmod r$.
Since \(uv\equiv1\pmod r\), we may therefore take
$
b_1=-u,
b_2=-u, 
b_3=2u.$ 
The rational Euler number is zero, and hence
\[
Y_k
\cong
M\bigl(
0;
(r,-u),
(r,-u),
(r,2u)
\bigr).
\tag{1}
\]

We also identify the involution.  The action induced by \(\tau_k\) on
the base of the Seifert fibration is represented by
\(\overline{\psi}^{\,v/2}\).  Since
$\left\langle\overline{\psi}^{\,v}\right\rangle
=
\left\langle\overline{\psi}^{\,2}\right\rangle$
and \(v/2\) is even for \(k\geq2\), we have
$
\overline{\psi}^{\,v/2}
\in
\left\langle\overline{\psi}^{\,v}\right\rangle.$ 
Thus \(\tau_k\) acts trivially on the Seifert base.  It is the
nontrivial fiberwise gauge transformation and is therefore of type C.
In particular, the pairs of Seifert entries used below are choices in
the full-tunnel construction; they are not pairs exchanged by
\(\tau_k\).

The distinguished spin structure \(\mathfrak{s}_k\) does not descend
through $
Y_k\longrightarrow Y_{k-1}$.
Since \(\tau_k\) is the fiberwise half-rotation, this is the spin
structure whose restriction to a regular Seifert fiber is bounding.
Thus its regular-fiber parameter is $ 
\mu=1$. 

For \(k=1\), we use 
$\delta_R^{(1)}(K)
=
-\frac1{16}\sigma(K)$ proven in \cite{KMT:2023}.
Since $
\sigma(T(2,r))=-(r-1)$,
we obtain
$
\delta_R^{(1)}(K_n)
=
\frac{r-1}{16}
=
\frac{2^{n-1}-1}{8}$. 

We now suppose that \(2\leq k\leq n-1\).  In this case \(u\) and \(v\)
are even.  Starting from (1), apply standard changes of Seifert
invariants and add one regular fiber to obtain
\[
Y_k
\cong
M\bigl(
0;
(r,-r-u),
(r,r-u),
(r,2u-r),
(1,1)
\bigr).
\tag{2}
\]

For the construction of our spin 4-orbifold \(\check X_k\), group the first two entries
and the last two entries in (2) into two full-tunnel pairs.  All four
multiplicities and coefficients in (2) are odd.  Since \(\mu=1\), the
meridional compatibility conditions force $ 
\lambda_1=\lambda_2=\lambda_3=\lambda_4=0$.
Thus the distinguished spin structure extends over the corresponding
\(X_0\).

For the first pair, choose the gluing data $ 
(\rho_1,\nu_1)=(-v-1,-v),
(\rho'_1,\nu'_1)=(v-1,-v)$.
For the second pair, choose
$
(\rho_2,\nu_2)
=
\left(\frac v2-1,\frac v2\right),
(\rho'_2,\nu'_2)=(1,0) $.  Since \(v\) is divisible by \(4\), the spin parameters are $
e_1=e'_1=e_2=e'_2=1.$ 

For the first tunnel, we have
\begin{align*}
&\Delta_1
=
(r-u)+(-r-u)
=
-2u, \ 
\widetilde a_1
=
-r(r-u)-r(-r-u)
=
2ru, \\
&
\widehat b_1
=
(v-1)r-(-v)(-r-u)
=
-(2r+1).
\end{align*}
For the second tunnel,
\[
\Delta_2
=
r+(2u-r)
=
2u,\ 
\widetilde a_2
=
-r-(2u-r)
=
-2u, \text{ and }
\widehat b_2=r.
\]
The normalized pairs in
\Cref{prop:continued-fraction-w-formula} are therefore
\[
(A_1,B_1)
=
\bigl(2ru,-(2r+1)\bigr),
\qquad
(A_2,B_2)
=
(2u,1).
\]

The intersection matrix is
\[
Q
=
\left(
-\Delta_1r
-
\frac{\Delta_1^2}{\Delta_2}r
\right)
=
\left(
2ur-2ur
\right)
=
(0). 
\]

Using \(r=uv-1\), we obtain
\[
\frac{A_1}{B_1}
=
\frac{2ru}{-(2r+1)}
=
[-u,-2v,-u]^-,
\]
because
\[
[-u,-2v,-u]^-
=
-\frac{2u(uv-1)}{2uv-1}
=
-\frac{2ur}{2r+1}.
\]
Moreover, $\frac{A_2}{B_2}
=
2u
=
[2u]^-$.
Hence $
\sum_j\operatorname{sgn}(\alpha_{1,j})=-3,
\sum_j\operatorname{sgn}(\alpha_{2,j})=1$.
By \Cref{thm:Quantities,prop:continued-fraction-w-formula},
$\delta_R^{(k)}(K_n)
=
-\frac18$
for every \(2\leq k\leq n-1\).

It remains to consider \(k=n\).  In this case \(u=1\) and
\(v=r+1=2^n\).  Starting from
\[
Y_n
\cong
M\bigl(
0;
(r,-1),
(r,-1),
(r,2)
\bigr),
\]
we use the equivalent presentation
\[
Y_n
\cong
M\bigl(
0;
(r,-2r-1),
(r,-1),
(r,r+2),
(1,1)
\bigr).
\tag{3}
\]
Here the changes in the rational Seifert coefficients are $
-2, 0, 1, 1$,
whose sum is zero.

Again, we group the first two entries and the last two entries in (3)
into two full-tunnel pairs.  The distinguished spin structure is
represented by $
\mu=1,
\lambda_1=\lambda_2=\lambda_3=\lambda_4=0$. 

For the first pair, choose the gluing data
$
(\rho_1,\nu_1)=(-2,-1),
(\rho'_1,\nu'_1)=(0,-1),$ 
and for the second pair choose $
(\rho_2,\nu_2)
=
\left(
-\frac{r+3}{2},
\frac{r+1}{2}
\right), 
(\rho'_2,\nu'_2)=(1,0)$.
The associated spin parameters are $ 
e_1=e'_1=0,
e_2=e'_2=1$.

For the first tunnel,
\[
\Delta_1=-2(r+1),
\qquad
\widetilde a_1=2r(r+1),
\qquad
\widehat b_1=-(2r+1).
\]
Since \(e_1=0\), the normalization requires adding one copy of
\(\widetilde a_1\), and hence
\[
(A_1,B_1)
=
\bigl(
2r(r+1),\,2r^2-1
\bigr).
\]
For the second tunnel,
\[
\Delta_2=2(r+1),
\qquad
\widetilde a_2=-2(r+1),
\qquad
\widehat b_2=r,
\]
and therefore
\[
(A_2,B_2)
=
\bigl(
2(r+1),-r
\bigr).
\]
As before,
\[
Q
=
\left(
-\Delta_1r
-
\frac{\Delta_1^2}{\Delta_2}r
\right)
=
(0).
\]

We use the identities
\[
[\underbrace{2,\ldots,2}_{s}]^-
=
\frac{s+1}{s}
 \text{ and }
[
\underbrace{2,\ldots,2}_{s},
4,
\underbrace{2,\ldots,2}_{s}
]^-
=
\frac{2(s+1)(s+2)}{2s^2+4s+1}.
\]
Taking \(s=r-1\), we obtain
\[
\frac{A_1}{B_1}
=
[
\underbrace{2,\ldots,2}_{r-1},
4,
\underbrace{2,\ldots,2}_{r-1}
]^-.
\]
This expansion has \(2r-1\) positive coefficients.  Moreover,
\[
\frac{A_2}{B_2}
=
\frac{2(r+1)}{-r}
=
\left[
-2,\frac{r+1}{2},2
\right]^-.
\]
Since $ 
\frac{r+1}{2}=2^{n-1}$
is a positive even integer, the sign sum of the second expansion is $
-1+1+1=1$.
Consequently, again from \Cref{thm:Quantities,prop:continued-fraction-w-formula},  we get $  
\delta_R^{(n)}(K_n)
= 
\frac r8
=
\frac{2^n-1}{8}$.
This completes the proof.
\end{proof}

\begin{proof}[Proof of \cref{thm:integral-linear-independence-delta}]
Set $\overline{\delta}^{(k)}:=8\delta^{(k)}$. For $n\geq2$, let $K_n:=T(2,2^n-1)$ and $D_n:=K_n\mathbin{\#}-K_{n+1}$. By \Cref{lem:delta-values-two-bridge-torus},
\[
\overline{\delta}^{(1)}(D_n)=-2^{n-1},\qquad \overline{\delta}^{(k)}(D_n)=0\quad(2\leq k<n),\qquad \overline{\delta}^{(n)}(D_n)=2^n.
\]
Suppose that
$
\sum_{k=1}^N c_k\overline{\delta}^{(k)}=0,\  c_k\in\mathbb Z$.
Evaluating on $D_{N+1}$ gives $-2^Nc_1=0$, so $c_1=0$. We now proceed by descending induction. If $c_{j+1}=\cdots=c_N=0$ for some $2\leq j\leq N$, then evaluation on $D_j$ gives
$
0=c_j\overline{\delta}^{(j)}(D_j)=2^jc_j$,
because $\overline{\delta}^{(k)}(D_j)=0$ for $2\leq k<j$. Hence $c_j=0$. Therefore $c_1=\cdots=c_N=0$, proving the assertion.
\end{proof}

\appendix

\section{Foundations of orbifolds} \label{Apx:orbifold}

\subsection{Fundamental definitions}

We briefly recall the basic notions of orbifolds, orbifold bundles, and
orbifold spin structures used in this paper. For general background on
orbifolds, see, for example, \cite{Satake,AdemLeidaRuan}; for the
gauge-theoretic setting of spin orbifolds used here, see also
\cite{FF00,fukumoto2001w}.

A smooth \(n\)-orbifold \(X\) consists of a Hausdorff second countable
topological space \(|X|\), together with an orbifold atlas as follows.
For every open set \(U\subset |X|\) in the atlas, there is a triple $
        (\widetilde U,G_{\widetilde U},\phi_U)$,
where \(\widetilde U\subset \R^n\) is an open set, \(G_{\widetilde U}\) is a
finite group acting smoothly and effectively on \(\widetilde U\) from left, and $
        \phi_U:\widetilde U\to U$
is a continuous map inducing a homeomorphism $
        \widetilde U/G_{\widetilde U}\cong U.$  Compatiblity of charts is described as follows: if
\(U\subset U'\), then, after possibly shrinking \(U\), there is a smooth
embedding $
        \lambda_U:\widetilde U\hookrightarrow \widetilde U'$
such that $
        \phi_{U'}\circ \lambda_U=\phi_U.$
Moreover, for each \(g\in G_{\widetilde U}\), there is an injective homomorphism $\iota_{\lambda_U}:G_{\widetilde U}\hookrightarrow G_{\widetilde U'}$ satisfying
$
        \lambda_U(g\tilde x)
        =
        \iota_{\lambda_U}(g)\lambda_U(\tilde x).$ 
The coordinate change \(\lambda_U\) is unique up to the action of
\(G_{\widetilde U'}\), and the coordinate changes are required to satisfy
the usual cocycle compatibility condition on triple overlaps.

For a point \(x\in |X|\), choose a chart
\((\widetilde U,G_{\widetilde U},\phi_U)\) and a lift
\(\tilde x\in \widetilde U\) with \(\phi_U(\tilde x)=x\). The isotropy group
of \(x\) is written as 
$
        G_x:=\{g\in G_{\widetilde U}\mid g\tilde x=\tilde x\}$,
which is well-defined up to conjugacy. We call \(x\) a regular point if
\(G_x\) is trivial.

A smooth orbifold \(X\) is oriented if every chart
\(\widetilde U\) is oriented, each \(G_{\widetilde U}\)-action is
orientation-preserving, and all coordinate changes are orientation-preserving.
Let \(X\) and \(Y\) be smooth orbifolds.

\begin{defn}
 A smooth orbifold map \(f:X\to Y\) consists of a continuous map \(f:|X|\to |Y|\) which is smooth at regular points, together with the following local data. For every point \(x\in |X|\), there exist orbifold charts \((\widetilde U,G_{\widetilde U},\phi_U)\) around \(x\) and \((\widetilde V,G_{\widetilde V},\phi_V)\) around \(f(x)\), a smooth map \(\widetilde f:\widetilde U\to \widetilde V\), and a homomorphism \(\theta_f:G_{\widetilde U}\to G_{\widetilde V}\), such that \(f(U)\subset V\), $$\widetilde f(g\cdot \tilde x)=\theta_f(g)\cdot \widetilde f(\tilde x)$$ for all \(g\in G_{\widetilde U}\), and the induced map on the quotient \(\widetilde U/G_{\widetilde U}\to \widetilde V/G_{\widetilde V}\) agrees with \(f\). These local data are required to be compatible with coordinate changes of the orbifold atlases, in the following sense. Suppose \(\lambda:\widetilde U\hookrightarrow \widetilde U'\) and \(\mu:\widetilde V\hookrightarrow \widetilde V'\) are coordinate changes, and let \((\widetilde f,\theta_f)\) and \((\widetilde f',\theta_f')\) be local representatives of \(f\) on these charts. Then, after possibly shrinking \(\widetilde U\), there exists an element \(h\in G_{\widetilde V'}\) such that the following diagram commutes 
 \[
\begin{CD}
\widetilde U @>{\widetilde f}>> \widetilde V \\
@V{\lambda}VV @VV{h\circ\mu}V \\
\widetilde U' @>{\widetilde f'}>> \widetilde V'.
\end{CD}
\]
\end{defn}
\begin{defn}\label{equal=Def}
Let \(f_0,f_1:X\to Y\) be smooth orbifold maps. We say that \(f_0=f_1\) if
the underlying maps \(|X|\to |Y|\) agree and the following condition holds.
For every \(x\in |X|\), choose local charts
\((\widetilde U,G_{\widetilde U})\) around \(x\) and
\((\widetilde V,G_{\widetilde V})\) around \(f_0(x)=f_1(x)\) such that
\(f_i\) is represented by \((\widetilde f_i,\theta_i)\), \(i=0,1\). Then
there exists \(h\in G_{\widetilde V}\) such that the diagrams
\[
\begin{CD}
\widetilde U @>{\widetilde f_0}>> \widetilde V \\
@V{\id_{\widetilde U}}VV @VV{h}V \\
\widetilde U @>{\widetilde f_1}>> \widetilde V
\end{CD}
\qquad
\begin{CD}
G_{\widetilde U} @>{\theta_0}>> G_{\widetilde V} \\
@V{\id}VV @VV{\operatorname{Ad}_h}V \\
G_{\widetilde U} @>{\theta_1}>> G_{\widetilde V}
\end{CD}
\]
commute.
\end{defn}

\begin{defn}
Let \(H\) be a compact Lie group and let \(X\) be a smooth orbifold. An orbifold principal \(H\)-bundle over \(X\) is a smooth orbifold \(P\) equipped with a smooth surjective orbifold map \(\pi:P\to X\) and a smooth right \(H\)-action on \(P\) preserving the fibers of \(\pi\), such that the following local condition holds. For every orbifold chart \((\widetilde U,G_{\widetilde U},\phi_U)\) of \(X\), there is an orbifold chart of \(P\) over \(U\) of the form \((\widetilde U\times H,G_{\widetilde U})\), under which the projection \(\pi\) is represented by the projection $\pr:\widetilde U\times H\to \widetilde U,$ the right \(H\)-action is represented by $(\tilde x,h)\cdot k=(\tilde x,hk)$, and the \(G_{\widetilde U}\)-action on \(\widetilde U\times H\) covers the given \(G_{\widetilde U}\)-action on \(\widetilde U\) and commutes with the right \(H\)-action. Namely, for every \(g\in G_{\widetilde U}\), the diagram $$\begin{CD} \widetilde U\times H @>{\pr}>> \widetilde U \\ @V{g}VV @VV{g}V \\ \widetilde U\times H @>{\pr}>> \widetilde U \end{CD}$$ commutes, and $$g\cdot((\tilde x,h)\cdot k)=(g\cdot(\tilde x,h))\cdot k$$ for all \(g\in G_{\widetilde U}\), \((\tilde x,h)\in \widetilde U\times H\), and \(k\in H\). These local models are required to be compatible with coordinate changes of the orbifold atlas of \(X\).
\end{defn}

Let $X$ be a smooth, oriented orbifold. 
Let $\operatorname{Fr}(X)$ be an orbifold framed bundle with respect to an orbifold Riemann metric $g$.  This is constructed as follows: 
An orbifold chart of \( \operatorname{Fr}(X) \) is associated to an orbifold chart \( \{ \tilde{U}, G_U \} \) of \( X \), and is given by the chart
\[
\{ \tilde{U} \times SO(n), G_U \}.
\]
Here, the \( G_U \)-action on \( \tilde{U} \times SO(n) \) is defined as follows. Suppose the action of \( \sigma \in G_U \) on \( \tilde{U} \) is denoted by
$\sigma_U : \tilde{U} \to \tilde{U}$,
then the linearlization $d\sigma_U : \tilde{U} \to SO(n)$
gives an action on \( \tilde{U} \times SO(n) \) is given by
\[
\sigma^*(\tilde{x}, q) = (\sigma(\tilde{x}), d\sigma_U (\tilde{x}) \cdot q).
\]

The gluing data for \( \operatorname{Fr}(X) \) is given as follows. Suppose the gluing of the orbifold charts of \( X \) is determined by an inclusion of open sets \( U \subset U' \), and a smooth embedding $
\lambda_U : \tilde{U} \to \tilde{U}'$
defined up to the action of \( G_{U'} \). Using the linearization of the orbifold coordinate change, we have 
$g_{\lambda_U} : \tilde{U} \to SO(n)$ 
and 
\[
\tilde{U} \times SO(n) \to \tilde{U}' \times SO(n), \quad (\tilde{x}, q) \mapsto (\lambda_U(\tilde{x}), g_{\lambda_U}(\tilde{x}) \cdot q).
\]

Define a right \( SO(n) \)-action on \( \tilde{U} \times SO(n) \) by $
(\tilde{x}, q) \cdot g := (\tilde{x}, qg)$ ,
then this action is preserved under the gluings described above and defines a right \( SO(n) \)-action on the total space \( \operatorname{Fr}(X) \). This right action commutes with the \( G_U \)-action from the orbifold chart. This give an orbifold $SO(n)$-bundle, the framed bundle of $X$.

Orbifold spin structures are defined in analogous ways to the smooth case. 

Let $\tau$ be a smooth involution on a smooth oriented orbifold $X$.
Take a \(\tau\)-invariant orbifold Riemannian metric \(g\).  Then the
differential of \(\tau\) induces an involution $
d\tau:\operatorname{Fr}(X)\to \operatorname{Fr}(X)$ 
as an orbifold principal \(SO(n)\)-bundle morphism.  One can see $(d\tau)^2= \id$.

Let $
\rho:P\to \operatorname{Fr}(X)$ 
be an orbifold spin structure. 
\begin{defn}
A lift of \(\tau\) to the spin structure \(P\) is called even or odd if the lift of $d\tau$ is of order two or four, i.e. there is a lift $\wt{\tau}: P \to P$ covers $\tau$ and $\wt{\tau}^2 = \id$ or $\wt{\tau}^2=-\id$ in the sense of 
\cref{equal=Def}. 
We say a spin $n$-orbifold equipped with an odd involution {\it real spin orbifold}. 
\end{defn}

\bibliographystyle{alpha}
\bibliography{tex}

\end{document}